\documentclass[a4paper,12pt]{amsart}
\usepackage{amsmath}
\usepackage{thmtools}
\usepackage{graphicx}
\usepackage{xspace}
\usepackage{amsthm}
\usepackage{amssymb}
\usepackage[all]{xy}
\usepackage[english]{babel}
\usepackage{csquotes}
\usepackage{enumitem}
\usepackage{url}
\usepackage[mathscr]{eucal}
\usepackage{mathrsfs}
\usepackage[scr=esstix,cal=boondox]{mathalfa}
\usepackage{calligra}

\usepackage{xcolor}
\usepackage{floatrow}
\usepackage{ltablex}
\usepackage{hyperref}

\DeclareFontFamily{OT1}{pzc}{}
\DeclareFontShape{OT1}{pzc}{m}{tt}{<-> s * [1.10] pzcmi7t}{}
\DeclareMathAlphabet{\mathpzc}{OT1}{pzc}{m}{tt}

\newcommand{\Conv}{{\rm Conv}}
\newcommand{\hHom}{{\mathcal Hom}}
\newcommand{\myRed}[1]{\textcolor{red}{#1}}
\newcommand{\myBlue}[1]{\textcolor{blue}{#1}}
\newcommand{\myGreen}[1]{\textcolor{green}{#1}}
\newtheorem{thm}{Theorem}[section]
\newtheorem{cor}[thm]{Corollary}
\newtheorem{lem}[thm]{Lemma}

\newtheorem{prop}[thm]{Proposition}

\newtheorem{defn}[thm]{Definition}
\newtheorem{Conjecture}[thm]{Conjecture}
\theoremstyle{definition}
\newtheorem{prop-defn}[thm]{Proposition and Definition}

\theoremstyle{remark}
\newtheorem{rem}[thm]{Remark}
\newtheorem{rems}[thm]{Remarks}
\newtheorem{example}[thm]{Example}
\newtheorem{examples}[thm]{Examples}

\newtheorem{Question}[thm]{Question}
\newtheorem*{Question*}{\bf Question}

\newcommand{\Id}{{\rm Id}}

\newcommand{\Image}{{\rm Im}}

\newcommand{\Hom}{{\rm Hom}}
\newcommand{\Homeo}{{\rm Homeo}}

\newcommand{\Diff}{{\rm Diff}}

\newcommand{\osc}{{\rm osc}}

\newcommand{\Mor}{{\rm Mor}}

\newcommand{\F}{{\mathcal F}}
\newcommand{\G}{{\mathcal G}}
\newcommand{\Hh}{{\mathcal H}}
\newcommand{\I}{{\mathcal I}}
\newcommand{\cI}{{\mathcal I}^{\bullet}}
\newcommand{\cH}{{\mathcal H}^{\bullet}}

\newcommand{\J}{{\mathcal J}}

\newcommand \id {{\rm id}}

\newcommand{\supp}{{\rm supp}}

\newcommand {\gr}{\mathrm {gr}}
\newcommand {\gra}{\mathrm {gr}}

\newcommand {\GFQI}{G.F.Q.I.\xspace}
\newcommand{\DHam }{\mathfrak{DHam}}
\newcommand{\HH }{\mathfrak{Ham}}
\newcommand {\LL}{{\mathfrak L}}

\DeclareMathOperator{\gammasupp}{\gamma-supp}
\def \dispdot {}
\def \dispdot {.}
\DeclareMathAlphabet{\mathpzc}{OT1}{pzc}{m}{tt}
\usepackage[percent]{overpic}
\newcommand\blfootnote[1]{
    \begingroup
    \renewcommand\thefootnote{}\footnote{#1}
    \addtocounter{footnote}{-1}
    \endgroup
}

\usepackage[
  backend=biber,
     style=ext-alphabetic,
    maxbibnames=5,
    sorting=anyt,
    giveninits,
    articlein=false,
    url=true,  
    doi=true,
    eprint=true
]{biblatex}
\title[On the supports in the  Humili{\`e}re completion]{On the supports in the  Humili{\`e}re completion \\and\\ $\gamma$-coisotropic sets}
\author{C. Viterbo}
\email{claude.viterbo@universite-paris-saclay.fr}
\address{Laboratoire de mathématiques d’Orsay, Université Paris-Saclay and UMR 8628 du CNRS, 91405-Orsay, France.}
\thanks{ Part of this paper was written as the author was a member of DMA, \'Ecole Normale Sup\'erieure, 45 Rue d'Ulm, 75230 Cedex 05, FRANCE. We also acknowledge support from ANR MICROLOCAL (ANR-15-CE40-0007) and COSY (ANR-21-CE40-0002)}

\begin{document}
\def \Z {\mathbb Z}
\def \H {\mathcal H}
\def \cF {\F^{\bullet}}
\def \cG {\G^{\bullet}}
\def \cI {\I^{\bullet}}
\def \cJ {\J^{\bullet}}
\def \Char {{\rm Char}}
\def \card {{\rm card}}
\def \cstar {\star}

\begin{abstract}
The symplectic spectral metric on the set of Lagrangian submanifolds or Hamiltonian maps can be used to define a completion of these spaces. For an element of such a completion, we define its $\gamma$-support. We also define the notion of
$\gamma$-coisotropic set and prove that every 
$\gamma$-support is 
$\gamma$-coisotropic. We then establish numerous properties of 
$\gamma$-supports and 
$\gamma$-coisotropic sets.
We construct examples of Lagrangians in the completion having large $\gamma$-support and we study those (called “regular Lagrangians”) having small $\gamma$-support. Finally we try to understand which singular Hamiltonians (i.e. a Hamiltonian continuous on the complement of a set) have a well-defined flow in the Humilière completion. 
\end{abstract}

\maketitle
\blfootnote{AMS classification: 53D05, 53D12, 53D35, 37J11. Keywords: Symplectic topology, spectral invariants, Lagrangian submanifolds, Coisotropic submanifolds, Hamiltonian maps}

\tableofcontents
\section{Introduction}\index{$\LL(T^*N)$} \index{$\DHam_c (T^*N)$} \index{$\widehat \DHam_c (T^*N)$} \index{$\widehat \LL (T^*N)$}

In \cite{Viterbo-STAGGF}, a metric $\gamma$ was introduced on the space ${\LL}_0 (T^*N)$ of Lagrangians Hamiltonianly isotopic to the zero section in $T^*N$, where $N$ is a compact manifold, and on $\DHam_c (T^*N)$, the group of compactly supported Hamiltonian diffeomorphisms, for $N = {\mathbb R}^n$ or $T^n$. A similar metric on $\DHam_c (T^*N)$ for general compact $N$ was subsequently defined in \cite{Viterbo-Montreal}. This metric was extended to general symplectic manifolds $(M, \omega)$ by Schwarz and Oh via Floer cohomology in the Hamiltonian setting (see \cite{Schwarz, Oh-spectrum1}), and by Leclercq--Zapolsky \cite{Leclercq, Leclercq-Zapolsky} in the Lagrangian setting\footnote{Under the asphericity assumptions $[\omega]\pi_2(M,L)=0$ and $\mu_L \pi_2(M,L)=0$.}. In particular, this distance is defined for elements within the same Floer--Fukaya class --- for instance, on ${\LL}_{L_0}(M, \omega)$, the set of exact Lagrangians in $(M, \omega)$ Hamiltonianly isotopic to $L_0$ --- and can sometimes be extended to the full space ${\LL} (M, \omega)$ of exact Lagrangians, notably when $M = T^*N$ with $N$ compact, where \cite{Fukaya-Seidel-Smith} guarantees a single Floer--Fukaya class.

The completions of $\LL_0(M, \omega)$ and $\DHam_{c} (M, \omega)$ with respect to $\gamma$ were first studied\footnote{Primarily in the Hamiltonian case and for ${\mathbb R}^{2n}$, though most results of \cite{Humiliere-completion} extend to the general setting.} in \cite{Humiliere-completion}. We denote these completions by $\widehat {\LL}_0(M, \omega)$ and $\widehat {\DHam}(M, \omega)$, and refer to them as the \emph{Humili\`ere completions}.

One of the main goals of this paper is to deepen the study of these completions. We work with the $\gamma$-distance defined via Floer cohomology; however, for $M = T^*N$, we show in Section~\ref{Section-6} that this distance admits an equivalent definition via sheaf theory, which proves useful when applying results from \cite{Viterbo-inverse-reduction}.

Elements of the Humili\`ere completion arise naturally in several contexts within symplectic topology. One example is symplectic homogenization (see \cite{SHT}), where flows of merely continuous Hamiltonians appear as homogenized Hamiltonians. Another is the graph of $df$ for a continuous function $f$, which gives rise to a notion of subdifferential that will be shown in \cite{AGHIV} to coincide with the one defined by Vichery \cite{Vichery} via the microlocal theory of sheaves. More generally, a continuous Hamiltonian whose singular set is sufficiently small admits a well-defined flow in $\widehat{\DHam}_c (T^*N)$ (see \cite{Humiliere-completion} and Section~\ref{Section-10}).

The main results of the paper are presented in Sections~\ref{Section-Defining-gamma-support} and~\ref{Section-Coisotropic}. In Section~\ref{Section-Defining-gamma-support}, we introduce the notion of \emph{$\gamma$-support} for a Lagrangian in $\widehat {\LL}(M, \omega)$ (see Definition~\ref{Def-support}) and establish its basic properties. Notably, although an element of $\widehat {\LL}(M, \omega)$ is a priori only an abstract Cauchy sequence, we can associate to it a concrete geometric object: the $\gamma$-support, a  subset of the ambient manifold.

In Section~\ref{Section-Coisotropic}, we introduce \emph{$\gamma$-coisotropic subsets} (see Definition~\ref{Def-coisotropic}), a notion originally proposed in a slightly different form by Usher \cite{Usher} under the name ``locally rigid.'' We establish a number of their properties (see Proposition~\ref{Prop-7.5}). In particular, we prove that the $\gamma$-support of any element of $\LL(M,\omega)$ is $\gamma$-coisotropic, and that many $\gamma$-coisotropic subsets arise as $\gamma$-supports --- including examples of Lagrangians containing an open $\gamma$-support. Characterizing which $\gamma$-coisotropic sets occur as $\gamma$-supports remains an open problem.

In Section~\ref{Section-Regular}, we study elements of $\LL(M,\omega)$ with minimal $\gamma$-support, i.e., those whose $\gamma$-support is itself a Lagrangian. When the $\gamma$-support is smooth, it is natural to ask whether $L$ coincides with $\gammasupp(L)$; we are able to prove this for a certain class of manifolds\footnote{Further results in this direction can be found in \cite{AGIV}.}. This perspective also yields a new definition of $C^0$-Lagrangian submanifold: an $n$-dimensional topological submanifold is called a \emph{$C^0$-Lagrangian} if it is the $\gamma$-support of some element of $\LL(M,\omega)$. We further show that our notion of $\gamma$-coisotropic is strictly more restrictive than previously considered ones (see Proposition~\ref{Prop-various-coisotropic} and \cite[Section 8.2]{Guillermou-Viterbo}).

In Section~\ref{Section-10}, we show that singular Hamiltonians whose singular set is nowhere $\gamma$-coisotropic admit a well-defined flow in $\widehat {\DHam}(M,\omega)$, extending earlier results of \cite{Humiliere-completion}.

Let us also mention several related works in which the ideas introduced here have been further developed. 

With St\'ephane Guillermou \cite{Guillermou-Viterbo}, we prove that singular supports of sheaves in $D^b(N)$ are $\gamma$-coisotropic, replacing the ``involutivity'' condition of \cite{K-S} by one that is invariant under symplectic homeomorphisms. 

Together with Tomohiro Asano, St\'ephane Guillermou, Vincent Humili\`ere, and Yuichi Ike \cite{AGHIV}, we prove that for the quantization $\mathcal F_{L}$ of an element $\widetilde L \in \widehat {\mathcal L}(T^*N)$ (see \cite{Guillermou-Viterbo} for the definition), the $\gamma$-support of $\widetilde L$ coincides with the singular support of $\mathcal F_{L}$. In \cite{AGIV}, we prove that if $L \in \widehat{\LL}(T^*N)$ has $\gamma$-support equal to an exact Lagrangian in $T^*N$, then $L$ coincides with that $\gamma$-support. 

In joint work with M.-C. Arnaud and Vincent Humili\`ere \cite{MCA-VH-CV}, we explore connections between the $\gamma$-support and conformal symplectic dynamics, showing in particular that $\gamma$-supports provide a natural generalization of the Birkhoff attractor --- a notion originally defined in dimension two \cite{Birkhoff-attractor, Charpentier-1} --- to arbitrary dimensions. We also establish that any compact $\gamma$-support in $T^*N$:
\begin{enumerate}
    \item has cohomology containing an injective image of the cohomology of $N$ \cite{MCA-VH-CV};
    \item is connected \cite{AGIV}.
\end{enumerate}

\section{Acknowledgements}I am grateful to St\'ephane Guillermou, Vincent Humili\`ere, Alexandru Oancea, Sobhan Seyfaddini, and Michele Stecconi for stimulating conversations related to this work. I am particularly indebted to Vincent Humili\`ere for writing \cite{Humiliere-these}. I also thank the anonymous referee and Maral Hatamizadeh for carefully pointing out a number of errors and misprints.

\section{Notations}

\begin{tabularx}{1.\textwidth}{l X}
$\mathscr{FD}(M,\omega)$ & the Fukaya--Donaldson category of Lagrangian branes (see Definition~\ref{Def-4.1})\\
${\mathbf{FD}}(M,\omega)$ & the set of Floer--Fukaya isomorphism classes (see Definition~\ref{Def-4.1})\\
$(T^*N, d\lambda)$ or $T^*N$ & the cotangent bundle of a closed manifold $N$, with Liouville form $\lambda$ (written $p\,dq$ in local coordinates)\\
$(T^*N, -d\lambda)$ or $\overline{T^*N}$ & the cotangent bundle of a closed manifold equipped with the opposite symplectic form\\
$\gr(df)$ & for $f \in C^1(N, \mathbb{R})$, the graph of $df$ in $T^*N$, i.e.\ $\gr(df) = \{(x, df(x)) \mid x \in N\}$\\
$(M,\omega)$ & an aspherical symplectic manifold satisfying $[\omega]\pi_2(M)=0$ and $c_1(TM)\pi_2(M)=0$, either closed or convex at infinity\\
$(M, d\lambda)$ & an exact symplectic manifold with Liouville form $\lambda$, convex at infinity\\
${\LL}(M,\omega)$ & the set of compact Lagrangian submanifolds of $(M,\omega)$
(see Definition~\ref{Def-5.1})\\
$\mathcal{L}(M,d\lambda)$ & the set of compact Lagrangian branes $\widetilde{L}$ in $(M,d\lambda)$ (see Definition~\ref{Def-5.1})\\
${\LL}_A(M,\omega)$ & a component of ${\LL}(M,\omega)$ in a fixed Floer-Fukaya class\\
$\mathcal{L}_A(M,d\lambda)$ & the subset of $\mathcal{L}(M,d\lambda)$ consisting of branes whose underlying Lagrangian lies in ${\LL}_A(M,\omega)$ (see Definition~\ref{Def-5.1})\\
$c_{+}(\widetilde{L}_1, \widetilde{L}_2),\ c_{-}(\widetilde{L}_1, \widetilde{L}_2)$ & spectral numbers associated to pairs of elements in $\mathcal{L}_A(M,\omega)$\\
$c$ & the spectral metric on $\mathcal{L}_A(M,\omega)$ from Definition~\ref{Def-5.1}\,(\ref{Def-5.1-d})\\
$\gamma$ & the spectral metric on ${\LL}_A(M,\omega)$\\
$\widehat{\LL}(M,\omega)$ (resp. $\widehat{\LL}_A(M,\omega)$) & the Humili\`ere completion of $\LL(M,\omega)$ (resp. $\LL_A(M,\omega)$) with respect to $\gamma$\\
$\widehat{\mathcal{L}}(M,d\lambda)$ (resp. $\widehat{\mathcal{L}}_A(M,d\lambda)$) & the Humili\`ere completion of $\mathcal{L}(M,d\lambda)$ (resp. ${\mathcal{L}}_A(M,d\lambda)$) with respect to the metric $c$ from Definition~\ref{Def-5.1}\,(\ref{Def-5.1-d})\\
$\DHam_c(M,\omega)$ & the group of time-one maps of compactly supported Hamiltonian flows, endowed with the $\gamma$-metric\\
$\widehat{\DHam}(M,\omega)$ & its Humili\`ere completion with respect to $\gamma$\\
$\mathscr{DHam}_c(M,\omega)$ & the group of compactly supported Hamiltonian diffeomorphisms (see p.~\pageref{Def-Hamiltonian-maps})\\
$\widehat{\mathscr{DHam}}(M,\omega)$ & its Humili\`ere completion with respect to the metric $c$ from Definition~\ref{Def-5.1}\,(\ref{Def-5.1-d})\\
$\mathcal{H}_\gamma(M,\omega)$ & the group of homeomorphisms preserving $\gamma$ (see the discussion preceding Corollary~\ref{Cor-6.23})
\end{tabularx}

  \section{Spectral invariants for sheaves and Lagrangians}\label{Section-6}
 
 Readers familiar with the content of this and the next section  can jump to Section \ref{Section-Defining-gamma-support}. 
 
Let $(M,\omega)$ be a symplectic manifold such that $2c_1(M)=0$. Let $\Lambda(M)$ be the bundle of Lagrangians subspaces of the tangent bundle to $M$, with fiber the Lagrangian Grassmannian $\Lambda (T_xM)\simeq \Lambda (n)$ and denote by $\widetilde \Lambda (M)$ the bundle induced by the universal cover $\widetilde \Lambda (n) \longrightarrow \Lambda(n)$ (see \cite{Seidel-graded}). 
Given a Lagrangian $L$, we assume we have a lifting of the Gauss map $G_L: L \longrightarrow \Lambda(M)$ given by $x\mapsto T_xL$ to a map $\widetilde G_L: L \longrightarrow \widetilde\Lambda (M)$. This is called a {\bf grading of $L$} (see \cite{Seidel}). 
When using coefficients in a field of characteristic  different from $2$ we assume the Lagrangians have a relative spin structure (see \cite{Seidel}). 

We assume the characteristic of the fields is $2$ in order to avoid discussions of Spin structures, but extending our results to the other characteristics is easy. 

Given a graded $L$, the canonical automorphism of the covering  induces a new grading and we denote it as $T(L)$ or $L[1]$, and its $q$-th iteration as $T^q(L)$ or $L[q]$. The grading yields an absolute grading for the Floer homology of a pair $(L_1,L_2)$ and hence for the complex of sheaves in the Theorem stated below. We shall almost never mention explicitly the grading, but notice that for exact Lagrangians in $T^*N$, a grading always exists since the obstruction to its existence is given by the Maslov class, and for exact Lagrangians in $T^*N$ the Maslov class vanishes, as was proved by  Kragh and Abouzaid (see \cite{Kragh-Abouzaid}, and also the sheaf-theoretic proof by \cite{Guillermou, Guillermou-Asterisque}).
For $(M, d\lambda)$ we consider the set $\mathcal L (M,d\lambda)$ of {\bf Lagrangian branes}, that is triples $\widetilde L=(L,f_L, \widetilde G_L)$ where $L$ is a compact exact graded Lagrangian, and  $f_L$ a primitive of $\lambda_{\mid L}$. We sometimes talk about an {\bf exact Lagrangian}, and this is just the pair $(L, f_L)$.
  
When $f_L$ is implicit we only write $L$, for example $\gra (df)$ means $(\gra (df), f)$ and in particular  $0_N$ means $(0_N, 0)$. In both cases there is an obvious canonical grading.  For $\widetilde L=(L, f_L)$ and $c$ a real constant, we write $\widetilde L+c$ for $(L, f_L+c)$. 
 Considering compact supported  Hamiltonian diffeomorphisms as special correspondences, that is Lagrangians in $\overline{M}\times M$ we can consider the corresponding branes, and denote this space by $\mathscr{DHam}_{c}(M, \omega)$\label{Def-Hamiltonian-maps}. 

More generally 
\begin{defn}\label{Def-4.1} The  Fukaya-Donaldson category $\mathscr{FD}(M,d\lambda)$ has for objects the elements of $\mathcal L (M,d\lambda)$
and for morphisms $\Mor_{\mathscr{FD}}(L,L')=FH^*(L,L')$. We set $\mathbf{FD}(M,d\lambda)$ to be any subcategory containing a single object for each  isomorphism class in $\mathscr{FD}(M,d\lambda)$ : this is also called the {\bf skeleton} of $\mathscr{FD}(M,d\lambda)$. It has for objects the set of Floer-Fukaya classes. 
\end{defn} 
In fact two exact Lagrangians $L_0,L_1$  are in the same Floer-Fukaya class if the PSS map yields isomorphisms $H^*(L_0)  \longrightarrow FH^*(L_0,L_1) \longrightarrow H^*(L_1)$. In particular in $T^*N$ according to \cite{Fukaya-Seidel-Smith, Kragh3} there is a single equivalence class. 
By abuse of language if $L$ is in $\mathcal L (M,d\lambda)$ and $A$ is an object in $\mathbf{FD}(M,d\lambda)$ we write $L\in A$ if $L$ is isomorphic to $A$. 

\subsection{Sheaf theoretic approach in \texorpdfstring{$T^{*}N$}{T*N}.}
Let $N$ be a closed and $k$-oriented manifold of dimension $n$. 
Let $(L,f_L)$ be  an exact  Lagrangian in $\mathcal L(T^*N)$ and $$\widehat L=\left\{(q,\tau p, -f_L(q,p), \tau) \mid (q,p)\in L, \tau > 0\right \}$$
the homogenized Lagrangian in $T^*(N\times {\mathbb R})$. 
We denote by $D^b(N)$ the derived category of bounded complexes of sheaves on $N$. On $D^b(N\times {\mathbb R} )$ we  define  $\cstar$ as follows. First we set $s: N\times {\mathbb R}\times N \times {\mathbb R} \longrightarrow N\times N \times {\mathbb R}  $ given by $s(x_1,t_1, x_2,t_2)=(x_1,x_2,t_1+t_2)$
and $d: N\times {\mathbb R} \longrightarrow N\times N \times {\mathbb R} $ given by $d(x,t)=(x,x,t)$ and now
$$ \cF \cstar \cG = d^{-1}(Rs)_!(\cF \boxtimes \cG)$$
Then $R\hHom^\cstar$ is the adjoint of $\cstar$ in the sense that 
$$\Mor_{D^b} (\cF, R\hHom^\cstar(\cG, \cH))=\Mor_{D^b} (\cF \cstar \cG, \cH)$$

According to \cite{Guillermou, Guillermou-Asterisque} (for \ref{TQ1}), (\ref{TQ2}), (\ref{TQ4})) and  \cite{Viterbo-Sheaves} (for the other properties),  we have
\begin{thm} \label{Thm-Quantization}
To each $\widetilde L\in \mathcal L(T^*N)$ we can associate $\cF_L \in D^b(N\times \mathbb R)$ such that 
\begin{enumerate} 
\item\label{TQ1} $SS(\cF_L)\cap \dot{T}^*(N\times {\mathbb R})=\widehat L$
\item\label{TQ2} $\cF_L$ is pure (cf. \cite{K-S} page 309), $\F_L^\bullet=0$ near $N\times \{-\infty\}$ and $\F_L^\bullet=k_N$ near $N\times \{+\infty\}$
\item\label{TQ3} We have an isomorphism
$$FH^\bullet(L_0,L_1;a,b)=H^*\left (N\times [a,b[, R\hHom^{\cstar}(\cF_{L_0},\cF_{L_1})\right)$$
\item \label{TQ4}  $\cF_L$ is unique satisfying properties (\ref{TQ1}) and (\ref{TQ2}). 
\item \label{TQ5} There is a natural product map  $$
 R\hHom^\cstar(\cF_{L_1},\cF_{L_2}) \otimes R\hHom^\cstar(\cF_{L_2},\cF_{L_3}) \longrightarrow  R\hHom^\cstar(\cF_{L_1},\cF_{L_3})
$$
inducing in cohomology a map
\begin{gather*}
H^*(N\times [\lambda , +\infty [, R\hHom^{\cstar}(\cF_{L_1},\cF_{L_2})) \otimes H^*(N\times [\mu , +\infty [, R\hHom^{\cstar}(\cF_{L_2},\cF_{L_3}))
\\  \Big\downarrow \cup_{\cstar} \\ H^*(N\times [\lambda + \mu , +\infty [, R\hHom^{\cstar}(\cF_{L_1},\cF_{L_3}))
\end{gather*}
that coincides through the above identifications to the triangle product in Floer cohomology.\end{enumerate} 
\end{thm} 
\begin{rem} 
The grading $\widetilde G_L$ defines the grading of $\cF_{L}$, hence of the Floer cohomology. 

\end{rem} 
Note that for $X$ open, we denote by $H^*(X\times [\lambda, \mu[, \cF)$  the relative cohomology of sections on $X \times ]-\infty, \mu[$ vanishing on $X \times ]-\infty, \lambda[$ and fitting in the exact sequence
$$
H^*(X\times [\lambda, \mu[, \cF) \longrightarrow H^*(X\times ]-\infty, \mu[, \cF) \longrightarrow H^*(X\times ]-\infty, \lambda[, \cF)
$$
It is also equal to the cohomology associated to the derived functor $R\Gamma_Z$ where $Z$ is the locally closed set $X\times [\lambda, \mu[$. We should have  written $H^*_{X\times [\lambda, \mu[}(X\times {\mathbb R}, \cF)$   and $\cF_{\widetilde L}$ instead of $\cF_{L}$ but this would be too cumbersome and our abuse of notation should be harmless.

\begin{defn}\label{Def-mock-Tamarkin}
We denote by $\mathcal T_0(N)$ the set of 
 $\cF$ in $D^b(N \times {\mathbb R} )$ such that  $SS(\cF) \subset \{\tau \geq 0\}$,  $\cF=0$ near $N\times \{-\infty\}$ and $\cF=k_N$ near $N\times \{+\infty\}$. \end{defn} 
Note that the set of $\cF$ such that $SS(\cF)  \subset \{\tau \geq 0\}$ contains the Tamarkin category (see \cite{Tamarkin, Guillermou-Asterisque}).

 \begin{defn}[see \cite{Vichery-these}, Section 8.3]\label{Def-4.4} Let $\cF$be an element in $\mathcal T_0(N)$. 
Let $\alpha \in H^*(N \times {\mathbb R}, \cF)\simeq H^*(N)$ be a nonzero class. We define 
$$c(\alpha; \cF)= \sup \left\{ t \in {\mathbb R}  \mid \alpha \in \Image ( H^*(N \times [t, +\infty[, \cF))\right \} \dispdot $$
We denote $c(\alpha; \cF, \cG)=c(\alpha; R\hHom^{\cstar}(\cF,\cG))$ and 

$$c(\cF, \cG)=\max\{c(\mu_N;\cF, \cG),0\}-\min\{c(1_N; \cF, \cG),0\}$$ where $1_N, \mu_N$are the generators of $H^0(N,k), H^n(N,k)$ respectively. 
Finally we set $$\gamma(\cF, \cG)=c(\mu_N;\cF, \cG)-c(1_N; \cF, \cG)= \inf _{c \in \mathbb R} c(\cF, T_c\cG)\dispdot $$
\end{defn} 

Note that $\cF_L$ satisfies (\ref{TQ2}), so $H^*(N \times {\mathbb R}, \cF)\simeq H^*(N)$ and thus we have the following Proposition, using the canonical map
$$H^*(N \times [t, +\infty[, \cF) \longrightarrow H^*(N \times  {\mathbb R}, \cF)$$
and Theorem \ref{Thm-Quantization}, (\ref{TQ3})

\begin{prop}\label{Prop-4.5}
For $\cF_L$ given by Theorem \ref{Thm-Quantization} then $c(\alpha, \cF_L)$ coincides with the spectral invariant $c(\alpha,L)$ associated to $\alpha$ by using Floer cohomology (see \cite{Schwarz, Oh-spectrum1}). 
\end{prop}  
As a consequence the $c(\alpha, \cF)$ satisfy the properties of the Floer homology Lagrangian spectral invariants, and in particular the triangle inequality, since this holds in Floer homology (see \cite{Hum-Lec-Sey3}, theorem 17). However we shall sometimes need to extend the triangle  inequality to situations where $\cF$ is in $\mathcal T_0(X)$ but does not necessarily  correspond to an exact embedded Lagrangian.

\begin{prop} [Triangle inequality for sheaves (see \cite{Vichery-these}, proposition 8.13)]\label{Prop-triangle}
Let  $\cF_1, \cF_2, \cF_3$ be sheaves on $X\times {\mathbb R} $  such that $\cF_j \in \mathcal T_0(X)$. 
Then we have 
$$ c(\alpha \cup_\cstar \beta; \cF_1,\cF_3) \geq c(\alpha ; \cF_1,\cF_2) + c( \beta; \cF_2,\cF_3)$$
\end{prop} 
\begin{proof} 
We set $\cF_{i,j}= R\hHom^{\cstar}(\cF_{L_i},\cF_{L_j})$ and we have a product $$\cup_{\cstar} :\cF_{1,2} \otimes \cF_{2,3} \longrightarrow \cF_{1,3}$$ inducing the cup-product 
$$H^*( X\times [s,+\infty[ ; \cF_{1,2}) \otimes H^*( X\times [t,+\infty[ ; \cF_{2,3}) \longrightarrow H^*( X\times [s+t,+\infty[ ; \cF_{1,3})$$
(see \cite{Viterbo-Sheaves}, section 9).
Then we have the diagram

\xymatrix{
H^*( X\times [s,+\infty[ ; \cF_{1,2}) \otimes H^*( X\times [t,+\infty[ ; \cF_{2,3})\ar[d] \ar[r]& H^*( X\times [s+t,+\infty[ ; \cF_{1,3}) \ar[d]\\
H^*( X\times  {\mathbb R}  ; \cF_{1,2}) \otimes H^*( X\times {\mathbb R}  ; \cF_{2,3}) \ar[r]& H^*( X\times {\mathbb R}  ; \cF_{1,3})
}

where horizontal arrows are cup-products and vertical arrows restriction maps. So if $\alpha\otimes \beta$ is in the image of the left-hand side, which is equivalent to $s\leq c(\alpha, \cF_{1}, \cF_2), t \leq  c(\beta, \cF_{2},\cF_3)$, we have $\alpha\cup \beta$ is in the image of the right hand side, so that
$ s+t\leq c(\alpha\cup \beta , \cF_{1,3})$. 
 \end{proof} 
  Similarly we have
 \begin{prop} [Lusternik-Shnirelman for sheaves]
 Let $\cF$ as above. Let $\alpha \in H^*(N\times {\mathbb R}, \cF_{})$ and $\beta \in H^*(N\times {\mathbb R}, \cG)$. We have a product $\alpha \cup_{\cstar} \beta \in H^*(N\times {\mathbb R}, \cF_{}\cstar \cG)$.
 Then $$c(\alpha \cup_{\cstar} \beta, \cF\cstar \cG) \geq c(\alpha, \cF)$$ and equality implies that $\beta \neq 0$ in $H^*(\pi(SS(\cF) \cap  \{t=c\}), \cG)$ where $c=c(\alpha, \cF)=c(\alpha \cup_{\cstar} \beta, \cF\cstar \cG)$ is the common critical value and $\pi: T^*(N\times \mathbb R) \longrightarrow N\times \mathbb R$ the standard projection. In particular if $c(1; \cF)=c(\mu_N;\cF)=c$ and $\cF$ is constructible, then  $SS(\cF)\supset 0_N\times T^*_{\{c\}} {\mathbb R}\cap \{\tau >0\}$.
 \end{prop} 
 \begin{proof} 
 By assumption $\alpha$ vanishes in $H^*(N\times ]-\infty, c- \varepsilon [, \cF)$ but not in  $H^*(N\times ]-\infty, c+ \varepsilon [, \cF)$. Assume $\beta$ vanishes in $(N\setminus U)\times ]c- \varepsilon , c+ \varepsilon [$. Then $\alpha\cup \beta$ vanishes in    $H^*(N\times ]-\infty, c- \varepsilon [ \cup (N\setminus U)\times ]c- \varepsilon , c+ \varepsilon [, \cF\cstar \cG)$ and by assumption does not vanish in 
 $$H^*(N\times ]-\infty, c+ \varepsilon [  , \cF\cstar \cG)$$
But deforming $N\times ]-\infty, c- \varepsilon [ \cup (N\setminus U) \times ]-\infty, c+ \varepsilon [ $ to $N\times ]-\infty, c+ \varepsilon [$ can be done through a family of hypersurfaces bounding $W_t$ such that 
$$W_0= N\times ]-\infty, c- \varepsilon [ \cup (N\setminus U) \times ]-\infty, c+ \varepsilon [$$ while
$W_1= N\times ]-\infty, c+ \varepsilon [$,  and we assume 
$$SS(\cF\cstar \cG)\cap \overline{\{(x,p) \mid x \in \bigcap_{s>t} W_{s}\setminus W_{t}\}} \subset 0_{N\times {\mathbb R} } \dispdot$$
According to the microlocal deformation lemma (\cite[proposition 2.7.2, p. 117 and corollary 5.4.19, p. 239]{K-S}), this implies that the natural map $H^*(W_1, \cF\cstar \cG) \longrightarrow H^*(W_0, \cF\cstar \cG)$ is an isomorphism. In our case this implies that $\alpha \cup_{\cstar}\beta$ is zero in 
$$H^*(N\times ]-\infty, c+ \varepsilon [; \cF\cstar\cG)$$ hence $c(\alpha\cup \beta; \cF\cstar \cG) \geq c+ \varepsilon $ a contradiction. 
 \end{proof} 
 
 Obviously when $\cF=\cF_{L}$ the equality $c(1; \cF_{L})=c(\mu_N;\cF_{L})$ implies $\widehat L \supset 0_{N}\times T_{c}^{*}{\mathbb R}$ hence $L=0_{N}$. 
\begin{cor} 
As a result  $\gamma$ defines a pseudo-metric on $\mathcal T_0(N)$.  It restricts to a metric on the image of $\mathcal L(T^*N)$ by the embedding $L \mapsto \cF_{L}$ which yields a bi-Lipschitz embedding. 
\end{cor} 
\begin{proof} 
We have $$c(1;\cF_1,\cF_3) \geq c(1;\cF_1,\cF_2)+ c(1;\cF_2,\cF_3)$$
and 
$$0=c(\mu; \cF_1,\cF_1)\geq c(1;\cF_1,\cF_2)+ c(\mu;\cF_2,\cF_1)$$
$$c(\mu; \cF_1,\cF_2) \geq c(1;\cF_1,\cF_2) + c(\mu; \cF_2,\cF_2)= c(1;\cF_1,\cF_2)$$
so that $\gamma(\cF_1,\cF_2)\geq 0$. 
\end{proof} 
\begin{rem} 
We refer to \cite{Guillermou-Viterbo, Asano-Ike} for different definitions of  the metric $\gamma$ on $D^{b}(N\times {\mathbb R})$. 
As pointed out in \cite{Guillermou-Viterbo}, the map $Q: L\mapsto \mathcal F_{L}^\bullet$ extends to an isometry $\widehat Q: \widehat {\mathcal L}(T^{*}N) \longrightarrow D^{b}(N\times {\mathbb R})$. In \cite{AGHIV} it is proved that $\gammasupp(\widetilde L)=RS(\widehat Q (L))$, where for a homogeneous set $X$ in $T^{*}(N\times {\mathbb R})$, $RS ({X})=X\cap \{\tau=1\}/(t)$ ($(t,\tau)$ are the canonical coordinates in $T^*{\mathbb R}$). 
\end{rem} 
 Let $\mu_N\in H^n(N)$ be the fundamental class of $N$ and $1_N\in H^0(N)$ the degree $0$ class. 
 \begin{defn} We set for $\cF$ in $\mathcal T_0(N)$
 $$c_+(\cF)=c(\mu_N; \cF)$$
 $$c_-(\cF)=c(1_N; \cF)$$
 $$\gamma (\cF)=c_+(\cF)-c_-(\cF) \dispdot$$
 We set $\mathbb D \cF$ to be the Verdier dual of $\cF$ and $s(x,t)=(x,-t)$ and $\check{\cF}$ is quasi-isomorphic to  $0 \to k_{N\times {\mathbb R} } \to s^{-1}(\mathbb D \cF)\to 0 $.
 \end{defn} 
 We notice that $SS(\mathbb D \cF)=-SS(\cF)$ where for $A \subset T^*(N\times {\mathbb R})$, we set $-A=\{(x,-p, t, -\tau) \mid (x,p,t,\tau)\in A\}$ (see \cite[Exercise V.13, p. 247]{K-S} ). As a result,  $\check{\cF}_L=\cF_{-L}$ where $-L=\{(q,-p) \mid (q,p)\in L\}$. We then have $\gamma (\cF)=\gamma(k_{[0,+\infty[}, \cF)$. The triangle inequality then implies
 \begin{prop} \label{Prop-3.6}
 We have for $\cF$ constructible in $\mathcal T_0(N)$
  \begin{enumerate} 
 \item \label{Prop-3.6-i} $c_+(\cF) \geq c_-(\cF)$ 
 \item \label{Prop-3.6-ii}
 $c_+(\cF)=-c_-(\check{\cF})$ 
 $c_-(\cF)=-c_+(\check{\cF})$ 
 so that $\cF \longrightarrow \check{\cF}$ is a $\gamma$-isometry. 
 \end{enumerate} 
 And of course if $L\in \mathcal L(T^*N)$ we have  
 $c_\pm(\cF_L)=c_\pm(L)$ and $\gamma(\cF_L)=\gamma(L)$, where $c_+(L)=c(\mu_N;L), c_-(L)=c(1_N;L), \gamma(L)=c_+(L)-c_-(L)$. 
 \end{prop} 
 \begin{proof} 
 Note that for $\cF$ cohomologically  constructible, we have $\check{\check \cF}=\cF$
 
 \end{proof} 
 Of course if $L\in \mathfrak L (T^*N)$ the $c_\pm(L)$ are not well defined (they are only defined up to constant) however, $\gamma(L)$ is well-defined. 
\subsection{Persistence modules and Barcodes as sheaves on the real line}

By a {\bf  Persistence module}, we mean a constructible sheaf on $( {\mathbb R}, \leq)$ the real line with the topology having for open sets the $]t,+\infty[$.  We refer to \cite{Barannikov, Chazal-CS-G-G-O, E-L-Z, K-S-distance, Z-C} for the theory and applications. 
This persistence module is uniquely defined by the spaces $V_t= \mathcal F (]-\infty, t[)$ and the linear maps $r_{s,t}: V_s \longrightarrow V_t$ defined for $s\leq t$, such that 
\begin{enumerate} 
\item for $s<t<u$ we have $r_{t,u} \circ r_{s,t}=r_{s,u}$
\item $\lim_{t>s} V_s=V_t$ where the limit is that of the directed system given by the $r_{s,t}$
\item $r_{t,t}=\Id$
\item the spaces $V_t$ are finite dimensional and the set of $s$ such that $\lim_{t<s} V_t\neq V_s$ is finite
\end{enumerate} 
Obviously $k_{[a,b[}$ is such a persistence module. As a consequence of Gabriel's theorem on quivers (see \cite{Gabriel}), we have
\begin{prop} [\cite{Crawley-Boevey}]
Any persistence module is isomorphic to a unique sum $$\bigoplus_j k_{[a_j, b_j[}$$
where $a_j \in {\mathbb R} \cup \{-\infty\}, b_j \in {\mathbb R} \cup \{+\infty\}$ and $a_j<b_j$.
\end{prop} 
It will be useful to remind the reader that 
\begin{lem}[see \cite{K-S-distance}, (1.10)]
$$\Mor (k_{[a,b[}, k_{[c,d[})=\left\{\begin{array}{ll} k\; \text{for \;} a\leq c <b\leq d \\ 0 \; \text{otherwise}  \end{array}\right . \dispdot $$
\end{lem} 

There is a graded version, when we consider $D^b( ({\mathbb R}, \leq))$. We denote by $k_{[a,b[}[n]$ the element in $D^b( ({\mathbb R}, \leq))$ given by 
the complex $0 \longrightarrow k_{[a,b[}[n] \longrightarrow 0$ concentrated in degree $n$. We set $D_{c}^{b}(N)$ to be the category of constructible sheaves on $N$

\begin{prop} [see \cite{Guillermou-Asterisque, K-S-distance}]\label{Prop-Decomposition-2}
Any element $\cF$ in $D_{c}^b( ({\mathbb R}, \leq))$ is isomorphic to a unique sum $$\bigoplus_j k_{[a_j, b_j[}[n_j]$$
where $a_j \in {\mathbb R} \cup \{-\infty\}, b_j \in {\mathbb R} \cup \{+\infty\}$,  $a_j<b_j$ and $n_j\in \mathbb Z$. 
If $\cF$ vanishes at $-\infty$ and is equal to $k_X^d$ at $+\infty$, then all the $a_j$ are different from $-\infty$ and the
number of $j$ such that $b_j=+\infty$ is exactly $d$. 
\end{prop} 
We refer to \cite{Guillermou-Viterbo} for extensions to the non-constructible case. 
It is quite clear that elements of $D_{c}^b( ({\mathbb R}, \leq))$ are obtained by considering graded persistence modules $V_t^d$ for $d$ in a finite range (in $\mathbb Z$) and with the maps $r_{s,t}$ we have maps $\delta_t : V_t^d \longrightarrow V_t^{d+1}$ so that 
\begin{enumerate} 
\item $\delta_t^2=0$
\item $\delta_s \circ r_{s,t}=r_{s,t}\circ \delta_t$ \dispdot
\end{enumerate} 
 \begin{rem} 
 One has to be careful, it is {\bf not true} that \\ $\Mor_{D^b(( {\mathbb R}, \leq))}(k_{[a, b[}[m], k_{[c, d[}[n])=0$ for $m\neq n$ ! 
 \end{rem} 
 
In particular, recall that given two elements $L_0,L_1 \in \mathcal L (M,\omega)$ the Floer complex of $L_0,L_1$ is generated by the intersection points of $L_0\cap L_1$. This complex is filtered by $f_{L_0,L_1}(x)=f_{L_1}(x)- f_{L_0}(x)$ for $x\in L_0\cap L_1$. Since the coboundary map is given by counting holomorphic strips, it increases the filtration, and 
we can define   $V_t=FC^*(L_0,L_1; t)$ the subspace of the Floer complex generated by the elements with filtration  greater than $t$. There is a natural map for $s<t$ from $V_s=FC^*(L_0,L_1; s)$ to $V_t=FC^*(L_0,L_1; t)$ that defines a persistence module. 
This is the persistence module associated with $V_t=FC^*(L_0,L_1; t)$ and yields an isomorphism between this persistence module and some $\bigoplus_j k_{[a_j, b_j[}[n_j]$. 

Since $V_t$  vanishes for $t$ small enough, the $a_j$ must all be finite. Since up to a shift in grading we have by Floer's theorem  $FH^k(L_0,L_1;t)=H^k(L_0)=H^j(L_1)$ for $t$ large enough, the number of $j$ such that $b_j=+\infty$ and $n_j=k$ is given by $\dim H^k(L_0)=\dim H^k(L_1)$. Finally, if $\cF_j$ is associated to $L_j$ by Theorem \ref{Thm-Quantization}
 we have that the above persistence module is given by the sheaf $(Rt)_*(R\hHom^\cstar(\cF_1, \cF_2))$. We thus get
 
 \begin{prop} 
 Let us set (according to Proposition \ref{Prop-Decomposition-2}) for $\cF_j=\cF_{L_j}$ and
 $$(Rt)_*(R\hHom^\cstar(\cF_1, \cF_2))=\bigoplus_j k_{[a_j, b_j[}[n_j] \dispdot $$
 Then there is a unique  $j_-$ such that $b_{j_-}=+\infty$ and $n_{j_-}$ is minimal and then $a_{j_-}=c_-(L_0,L_1)$ and a unique $j_+$ such that $b_{j_+}=+\infty$ and $n_{j_+}$ is maximal  and then $a_{j_+}=c_+(L_0,L_1)$. Moreover $n_{j_+}-n_{j_-}=n=\dim (L_0)=\dim (L_1)$.
 In particular if $\Omega $ is a connected open set with smooth boundary, we have 
 $$H^{k}(\Omega \times {\mathbb R} , R\hHom^\cstar(\cF_1, \cF_2))= H^{k-n_{j-}}(\Omega)$$  
 and
 \begin{gather*} H^{k}_c(\Omega \times {\mathbb R}; R\hHom^\cstar(\cF_1, \cF_2))=\\ H^{k}_c(\Omega \times {\mathbb R}, \partial \Omega \times {\mathbb R}; R\hHom^\cstar(\cF_1, \cF_2))= H^{k-n_{j-}}(\Omega)\dispdot 
 \end{gather*}  
 \end{prop} 
 \begin{rem} 
 The actual values of $n_j^+,n_j^-$ depend on the grading of $L_1,L_2$: shifting the grading shifts the values by the same quantity. Note that $\cF=R\hHom^\cstar(\cF_1, \cF_2)$ satisfies 
 $SS(\cF)\subset \{\tau \geq 0\}$ and $\cF$ equals $k_X[n_{j^-}]$ over $X \times \{+\infty\}$ (we will shorten this by writing $\cF=k_X[n_{j^-}]$ at $+\infty$). Note that we sometimes assume $\cF=k_X$, which is equivalent to normalizing $n_{j^-}=0$.
 \end{rem} 
 \begin{proof} 
 Only the last statement needs a proof. But by assumption  $$SS(R\hHom^\cstar(\cF_1, \cF_2)) \subset \{ \tau \geq 0\}$$ so 
 $$H^{k}(\Omega \times {\mathbb R} , R\hHom^\cstar(\cF_1, \cF_2))= H^{k}(\Omega \times \{+\infty\} , R\hHom^\cstar(\cF_1, \cF_2))= H^{k-n_{j_-}}(\Omega)  \dispdot $$
  \end{proof} 

 \section{Various spaces, metrics and completions}
  
We remind the reader that according to \cite{Fukaya-Seidel-Smith, Kragh-Abouzaid}, $\mathbf {FD}(T^*N,d\lambda)$ has a single object.  
We shall use the following. 

 \begin{defn} \label{Def-5.1}
 We set for $A\in \mathbf {FD}(M,d\lambda)$ \begin{enumerate} 
 \item \label{Def-5.1-a} $\mathcal L_A(M,d\lambda)$ to be the set of triples $\widetilde L =(L,f_L, \widetilde G_{L})$ of graded exact closed Lagrangians branes $L\in A$ where  $f_L$ is a primitive of $\lambda_{\mid L}$, i.e. $df_L=\lambda_L$. 
 \item  \label{Def-5.1-b}  The action of ${\mathbb R}$ on $\mathcal L_A(M, d\lambda)$  given by $(L, f_L,\widetilde G_{L}) \longrightarrow (L, f_L+c,\widetilde G_{L})$ is denoted $T_c$. The action of the generator of the deck transformation $\widetilde \Lambda (M, d\lambda) \longrightarrow  \widetilde \Lambda (M, d\lambda)$ is denoted by $\widetilde L \longrightarrow \widetilde L[1]$. 
 \item  \label{Def-5.1-c}  For $\widetilde L_1,\widetilde L_2\in \mathcal L(T^*N)$, and $\alpha \in H^*(N)$, we denote the spectral invariants obtained  using Floer cohomology (see \cite{Hum-Lec-Sey3}) by $c(\alpha, \widetilde L_1, \widetilde L_2)$. They   are usually shortened in $c(\alpha; L_1,L_2)$ if $f_{L_1}, f_{L_2}$ are implicit. 
 \item  \label{Def-5.1-d}  On  $\mathcal L(M,d\lambda)$ we define $c_+(\bullet, \bullet) =c(\mu;\bullet, \bullet)$,  $c_-(\bullet, \bullet)=c(1;\bullet, \bullet)$  and $\gamma (\bullet, \bullet)=c_+(\bullet,\bullet)-c_-(\bullet, \bullet)$.  The metric $c$ given by
 $$c(\widetilde L_1, \widetilde L_2)= \max\{ c_+(\widetilde L_1, \widetilde L_2), 0 \} - \min \{c_-(\widetilde L_1, \widetilde L_2),0\} \dispdot  $$
 \item  \label{Def-5.1-e}  We denote by $\LL (M,d \lambda)$ the set of exact  closed Lagrangians. There is a forgetful functor ${\rm unf}: \mathcal L(M,d\lambda) \longrightarrow \LL (M,d\lambda)$ obtained by forgetting $f_L$ and the grading. The quantity $\gamma$ descends to a metric on $\LL (M, d\lambda)$, still denoted $\gamma$.
 \item We denote by $\gamma-\lim_k (L_k)$ and $c-\lim_k (\widetilde L_k)$ the  $\gamma$-limits and $c$-limits of the sequences $(L_k)_{k\geq 1}$ and $(\widetilde L_k)_{k\geq 1}$. 
 \end{enumerate} 
 \end{defn} 
 \begin{rem} 
 We let the reader check that $c(\widetilde L_1, \widetilde L_2)$ indeed defines a metric. The definition of $c$ ensures that it is      always non-negative and satisfies the triangle inequality. When $c_-(\widetilde L_1,\widetilde L_2)\leq 0\leq c_+(\widetilde L_1,\widetilde L_2)$, $ \gamma(\widetilde L_1, \widetilde L_2) = c(\widetilde L_1, \widetilde L_2)$. 
 \end{rem} 
 As we mentioned, Proposition \ref{Prop-4.5} implies that for complexes of sheaves of the form $\cF_L$, the spectral invariants defined by Definition \ref{Def-4.4} coincide with those defined using Floer cohomology. 
 It will be useful to define some modified metrics and completions.
 
 In the sequel we fix a class $A \in \mathbf{FD}(M,d\lambda)$ and write $\LL (M,d \lambda)$ and $\mathcal L(M,d\lambda)$ instead of $\mathfrak L_A(M,d\lambda)$ and $\mathcal L_A(M,d\lambda)$.
 
 We may now state 
 \begin{prop}\label{Prop-5.2} 
 The following holds
 \begin{enumerate} 
 \item Set $\HH_c(M,\omega)$ to be the metric space of compact supported Hamiltonians on $[0,1]\times M$ and $\DHam_{c}(M,\omega)$ the space of compact supported Hamiltonian diffeomorphisms. Then the compact supported  isotopy $(\varphi_H^t)_{t\in [0,1]}$  uniquely defines the Hamiltonian, so there is a  fibration $\HH_c(M,\omega) \longrightarrow \DHam_c(M,\omega)$ (which is essentially the path-space fibration). It factors through $\mathscr{DHam}_{c}(M, \omega)$, so we have in fact 
 $$\HH_c(M,\omega) \longrightarrow \mathscr{DHam}_{c}(M, \omega) \longrightarrow \DHam_c(M,\omega) \dispdot $$
 \item There is an action of $\HH_c(M,d\lambda)$ on $\mathcal L(M,d\lambda)$ given by 
 $$\varphi_H(L,f_L)=(\varphi_H(L), H\#f_L)$$ where 
 $$(H\#f_L)(z)=f_L(z)+\int_0^1 [p(t)\dot q(t)-H(t,q(t),p(t))]dt$$ where $q(t),p(t))=\varphi_H^t(z)$ factors through $\mathscr{DHam}_{c}(M, \omega)$. 
 The action does is the obvious action on the grading (and of course commutes with the shift $L \mapsto L[k]$). 
This action descends to the canonical action of $\DHam_c(M,d\lambda)$ on $\LL (M,d\lambda)$. 
 \item The above action commutes with $T_c$. We have $$c_\pm (T_c\widetilde L_1, \widetilde L_2)= c_\pm (\widetilde L_1, T_{-c}\widetilde L_2)= c_\pm (\widetilde L_1, \widetilde L_2)+c$$
 \item The relationship between the metric $c$
on $\mathcal L(M,d\lambda)$
 (which remembers the primitive) and the metric
$\gamma$ on $\LL (M,d\lambda)$
 (which forgets the primitive) is given by
 $$\gamma (L_1,L_2)= \inf \left\{c(\widetilde L_1, \widetilde L_2) \mid unf(\widetilde L_1)=L_1, unf(\widetilde L_2)=L_2 \right\}$$
 \item For $H\leq K$ we have $c_+(\varphi_H(L_1), L_2) \leq c_+(\varphi_K(L_1), L_2)$ and the same holds for $c_-$. 
 \end{enumerate} 
\end{prop} 
 \begin{proof} 
 Proofs of the first three  statements  are left to the reader. For the fourth, we just notice that we may change $f_{L_1}, f_{L_2}$ to $f_{L_1}+c_1, f_{L_2}+c_2$ for any two constants $c_1,c_2$. Then $c_\pm(\widetilde L_1, \widetilde L_2)$ is changed to $c_\pm(\widetilde L_1, \widetilde L_2)+c_1-c_2$, and  
 $c(\widetilde L_1, \widetilde L_2)$ is changed to $$\max\{c_+(\widetilde L_1, \widetilde L_2)+c_1-c_2,0\} - \min \{ c_-(\widetilde L_1, \widetilde L_2)+c_1-c_2, 0\} $$
 and this is minimal when $c_-(\widetilde L_1, \widetilde L_2)<c_2-c_1<c_+(\widetilde L_1, \widetilde L_2)$ and takes the value
 $$c_+(\widetilde L_1, \widetilde L_2)-c_-(\widetilde L_1, \widetilde L_2)=\gamma (L_1,L_2)$$
 The last statement is equivalent to $c(\alpha, \varphi_H^t(\widetilde L_1), \widetilde L_2)$ is increasing for $H\geq 0$. But this follows from the formula (see \cite{Viterbo-STAGGF}, prop 4.6 and Lemma 4.7) for almost  all $t\in [0,1]$ we have 
 $$ \frac{d}{dt}c(\alpha, \varphi_H^t(\widetilde L_1), \widetilde L_2)=H(t,z)$$ for some  point $z$ of $\varphi_H^t(L_1)\cap L_2$. 
 \end{proof} 
  We may now set
 \begin{defn}\label{Defn-5.4} \leavevmode %
  \index{$\widehat \LL_c (M, \omega)$} We define the Humili{\`e}re completions as
 \begin{enumerate} 
 \item \label{Defn-5.4-i} $\widehat {\mathcal L}(M,d\lambda)$ is the $c$-completion of $\mathcal L (M,d\lambda)$ and $\widehat \LL (M, \omega)$  the $\gamma$-completion of $\LL (M, \omega)$. 
 \item \label{Defn-5.4-ii} $\widehat{\mathscr{DHam}}(M, \omega)$   is the $c$-completion of $\mathscr{DHam}_{c} (M, \omega)$ and
 $\widehat \DHam (M, \omega)$ is the $\gamma$-completion of $\DHam_{c} (M, \omega)$
 \end{enumerate} 
 \end{defn}
 \begin{rems} \leavevmode
 \begin{enumerate} 
\item Elements in  $\widehat {\mathcal L}(M,d\lambda)$ or $\widehat \DHam (M, \omega)$ are not necessarily compact supported : they could be limits of sequences with larger and larger support. 
\item According to the continuity of the spectral distance in terms of the $C^0$-distance, proved in \cite{BHS-C0}, an element in the group of Hamiltonian homeomorphisms, that is a $C^0$-limit of elements of $\DHam_{c} (M, \omega)$, belongs to $\widehat {\DHam}(M, \omega)$. Moreover an element in $\widehat \LL (T^*N)$ (or $\widehat {\DHam}(M, \omega)$) has a barcode, as follows from the Kislev-Shelukhin theorem (see \cite{Ki-Sh} and the Appendix in \cite{Viterbo-inverse-reduction}) and was pointed out in \cite{BHS-C0}. 
\end{enumerate} 
 \end{rems} 

The relationship between the two completions $\widehat {\mathcal L}(M, d\lambda)$ and $\widehat\LL (M, d\lambda)$ is clarified by the following: 
 \begin{prop}\label{Prop-5.6} \leavevmode
 \begin{enumerate} 
 \item \label{Prop-5.6a} There is a   continuous map   $$\widehat{\rm unf}: \widehat{\mathcal L} (M, d\lambda) \longrightarrow \widehat\LL (M, d\lambda)$$ extending  $\rm {unf}$ 
 and a continuous map $$\widehat T_c : \widehat {\mathcal L} (M, d\lambda) \longrightarrow \widehat {\mathcal L} (M, d\lambda)$$ extending $T_c$. 
 \item \label{Prop-5.6b} If $ L\in \widehat\LL (M, d\lambda)$, there exists $\widetilde L \in  \widehat {\mathcal L} (M, d\lambda)$ such that $\widehat{\rm unf}(\widetilde L)=L$ (i.e  $\widehat{\rm unf}$ is onto !) 
 \item \label{Prop-5.6c}  $\widehat{\rm unf}(\widetilde L_1)=\widehat{\rm unf}(\widetilde L_2)$ if and only if $\widetilde L_2=\widehat T_c \widetilde L_1$ for some $c\in {\mathbb R} $. 
 \end{enumerate} \end{prop} 
 \begin{proof} 
 \begin{enumerate} 
 \item This is obvious since Lipschitz maps extend to the completion. Now $\rm unf$ is $1$-Lispchitz, while $T_c$ is an isometry. 
 \item Let $(L_j)$ be a Cauchy sequence in $\LL(M, d\lambda)$. We may assume that $\gamma (L_j,L_{j+1})<2^{-j}$. There is a lift $\widetilde L_j$ of $L_j$, well determined up to a constant, and we can recursively adjust $\widetilde L_{j+1}$ so that $c(L_j,L_{j+1})< 2^{-j}$. Indeed, assume the $\widetilde L_j$ are defined for $1\leq j \leq k$. Since if $\widetilde L'_{k+1}$ is in the preimage of $L_{k+1}$ we have $\gamma (L_k, L_{k+1})= \inf_{c\in {\mathbb R} } c(\widetilde L_k, \widetilde L'_{k+1}+c)$ we just choose $\widetilde L_{k+1}=\widetilde L'_{k+1}+c_{k+1}$, so that $c(\widetilde L_k, \widetilde L_{k+1})<2^{-k}$. As a result $(\widetilde L_j)_{j\in \mathbb N}$ is a Cauchy sequence so defines an element in $\widehat {\mathcal L} (T^*N)$. Moreover its limit $\widetilde L$ projects on $L$.  
\item  Set $L_j = \widehat{\rm unf}(\widetilde L_j)$ with $L_2=L_1$. Let $(\widetilde L_1^k)_{k\geq 1}, (\widetilde L_2^k)_{k\geq 1}$ be Cauchy sequences in $\mathcal L (T^*N)$, converging to $\widetilde L_1, \widetilde L_2$, so that the $L_j^k$ converge to $L_j$. We showed that $\gamma (L_1^k,L_2^k)= \inf_{c\in {\mathbb R} } c(\widetilde L_1^k, \widetilde L_2^k+c)$ and since by assumption, $\gamma(L_1^k,L_2^k)$ goes to zero as $k$ goes to infinity, we have a sequence $c_k$ such that 
 $\lim_k c(\widetilde L_1^k, \widetilde L_2^k+c_k)=0$. But  $c_\pm(\widetilde L_1^k, \widetilde L_2^k+c_k)=c_\pm(\widetilde L_1^k, \widetilde L_2^k)+c_k$, and since 
 $\lim_k c_\pm(\widetilde L_1^k, \widetilde L_2^k)=c_\pm(\widetilde L_1,\widetilde L_2)$, we must have $\lim_k c_k=-c_\pm(\widetilde L_1,\widetilde L_2)$. As a result $$\widetilde L_1=\widetilde L_2-c(\widetilde L_1,\widetilde L_2)$$
 \end{enumerate} 
 \end{proof} 
 \begin{rem} 
 Notice that to a pair $\widetilde L_1,\widetilde L_2$ in $\widehat{\mathcal L}(M,\omega)$  we may associate a Floer homology as a filtered vector space. Indeed, by the Kislev-Shelukhin inequality (see \cite{Ki-Sh}, or the Appendix of \cite{Viterbo-inverse-reduction}), the bottleneck distance between the persistence module, denoted $V(L_1^k,L_2^k)$ associated to $FH^*(L^k_1, L^k_2 ; a,b)$ satisfies 
 $$\beta (V(L_1^k,L_2^k), V(L_1^l,L_2^l) \leq  \gamma (L_1^k, L_1^l)+\gamma (L_2^k, L_2^l)$$ where $\beta$ is the bottleneck distance (depending on the exact definition of $\beta$, there could be an extra factor $2$  on the right hand-side, which is however irrelevant for our purposes) so we get a Cauchy sequence of
 persistence modules, and this has a limit as well (however if we started with  persistence modules with finite barcodes, the limit will only have finitely many bars of size $>  \varepsilon >0$, but possibly an infinite number in total).
 \end{rem}

 \section{Defining the \texorpdfstring{$\gamma$}{gamma}-support}\label{Section-Defining-gamma-support}
 An element $L$ in $\widehat {\LL}(M, \omega)$ is not a subset of $M$ (and an element in $\widehat{\DHam} (M, \omega)$ does not define a map), but we may define its {\bf $\gamma$-support} as follows:

\begin{defn}[The $\gamma$-support]\label{Def-support}\index{$\gamma$-support}\leavevmode
\begin{enumerate} 
\item
Let $L\in \widehat {\LL}(M, \omega)$. Then $x\in \gammasupp (L)$ if for any neighbourhood $U$ of $x$, there exists a Hamiltonian map $\varphi$ supported in $U$ such that $\gamma (\varphi(L),L)>0$.    
\item 
Let $\widetilde L$ is in $\widehat{\mathcal L} (M, d\lambda)$. Then  $x\in c-\supp(\widetilde L)$ if  for any neighbourhood $U$ of $x$ there exists a Hamiltonian $H$ supported in $U$ such that $c_+ (\varphi_H(\widetilde L),\widetilde L)>0$. 
\end{enumerate} 
\end{defn} 
\begin{rems} 
\begin{enumerate} 
\item Observe that if $c(\varphi_H(\widetilde L),\widetilde L)>0$, then either $ c_+(\varphi_H(\widetilde L),\widetilde L)>0$ or 
$c_-(\varphi_H(\widetilde L),\widetilde L)<0$ and then 
$$c_+(\varphi_H^{-1}(\widetilde L),\widetilde L,)=c_+(\widetilde L,\varphi_H(\widetilde L))= -c_-(\varphi_H(\widetilde L),\widetilde L) >0$$ Thus our definition is unchanged if we replace $c_+$ by $c$.
\item The definition automatically implies that $\gammasupp(L)$ and $c-\supp(\widetilde L)$ are closed.
 \item  Notice that for $\widehat{\rm unf}(\widetilde L)=L$ we may have $c(\varphi_H(\widetilde L),\widetilde L)>0$ but $\gamma (\varphi_H(L),L)=0$. For example if $H\equiv c$ on $L$(but $H$  is compact supported)  then $c_+(\varphi_H(\widetilde L), \widetilde L)=c, c_-(\varphi_H(\widetilde L), \widetilde L)=c$ so that $c(\varphi_H(\widetilde L), \widetilde L) =2c$, while of course $\gamma (\varphi_H(L),L)=0$.  Note that in this case $\varphi_H(\widetilde L)=\widetilde L+c$ hence $c_+(\varphi_H^k(\widetilde L),\widetilde L)=k\cdot c$. 
\end{enumerate} 
 \end{rems} 
We can however  get rid of the second definition
\begin{prop} \label{Prop-6.3}
We have for  $\widehat {\rm unf}(\widetilde L)=L$,  $$c-\supp(\widetilde L)=\gammasupp (L)$$
\end{prop} 
\begin{lem} \label{Lemma-6.4}
Let $U$ be an open set containing the support of the isotopy $(\varphi^t)_{t\in [0,1]}$ and $\psi$ be a Hamiltonian map displacing $L$ outside of  $U$ (i.e. $\psi(L)\cap U=\emptyset$): 
$$c(\mu; \varphi^1(\widetilde L), \widetilde L)\leq c(\mu, \psi(\widetilde L),\widetilde L)-c(1,\psi(\widetilde L),\widetilde L)$$
\end{lem}
\begin{proof} 
Since $\psi\varphi^t(L)\cap L=\psi (\varphi^t(L) \cap \psi^{-1}(L))$ does not depend on $t$, $c(\mu; \psi\varphi^t(\widetilde L),\widetilde L)$ is constant. The triangle inequality implies
\begin{gather*} c(\mu;\psi(\widetilde L), \widetilde L)=c(\mu;\psi\varphi^1(\widetilde L), \widetilde L)= c(\mu; \varphi^1(\widetilde L),\psi^{-1}\widetilde L) \geq \\  c(\mu;  \varphi^1(\widetilde L),\widetilde L) +c(1; \widetilde L, \psi^{-1}\widetilde L)
\\ 
\end{gather*}  hence
$$c(\mu;  \varphi^1(\widetilde L),\widetilde L) \leq c(\mu;\psi(\widetilde L), \widetilde L)-c(1; \psi(\widetilde L), \widetilde L)
$$
\end{proof}  
\begin{proof} Let us first prove the existence of $\psi$ when $U$ is a ball. Indeed, we may assume $L$ avoids the center of the ball, so that the conformal vector field given in local coordinates by  $X(x)=x$ (with flow inside the ball given by $x \mapsto e^t\cdot x$) extends to a conformal vector field on $M$. This problem is equivalent to extending a primitive $\alpha$ of $\omega$ on $U$ to a primitive of $\omega$ on $M$ coinciding with $\lambda$ near infinity. Since $\lambda-\alpha$ is closed hence exact on $U$, there is a function $f$ such that $df=\lambda-\alpha$. By truncating $f$ inside $U$, we get that $\alpha+ df$ coincides with $\alpha$ in $U$ and with $\lambda$ at infinity. Then the dual vector field $Y$, defined by $i_Y\omega=\alpha$ has a flow $\rho^t$ coinciding with  $x \mapsto e^t\cdot x$ inside $U$, so for $t$ large enough $\rho^t(L)\cap U=\emptyset$. By rescaling we may assume $t=1$. 
By Banyaga's ``extension of symplectic isotopies's theorem''  (\cite[theorem II.2.3]{Banyaga}),    $\rho^t$ induces on the exact Lagrangian $L$ a Hamiltonian isotopy, that we denote $\psi^t$. Note that $\psi=\psi^1$ only depends on $L$ and $U$, but not on $\varphi$. 

We, may now conclude the proof of Proposition \ref{Prop-6.3}. Assume by contradiction that  for some $\widetilde L$ the $c$-support and $\gamma$-support were different.  We then would have a point $z$ and a Hamiltonian $H$ supported in a small ball centered at $z$ such that $c(\varphi_H(\widetilde L),\widetilde L)>0$ while $L=\varphi_H(L)$. According to Proposition \ref{Prop-5.6} (\ref{Prop-5.6c}), this implies that $\varphi_H(\widetilde L)=\widetilde L+c$. This is impossible  for $H$ supported in the ball $U$, since we would  have $\varphi_H^k(\widetilde L) =\widetilde L +k\cdot c$. Possibly replacing $H$ by $-H$, we may assume $c>0$. Then according to Lemma \ref{Lemma-6.4}, we would have 
 $$k\cdot c=c(\mu; \varphi_H^k(\widetilde L),\widetilde L)\leq c(\mu;\varphi_H^k(\widetilde L), \widetilde L) \leq c(\mu;\psi(\widetilde L), \widetilde L)-c(1; \psi(\widetilde L), \widetilde L) < +\infty $$
a contradiction. 
\end{proof} 

We  now define compact supported Lagrangians and Hamiltonians in the completion
  \begin{defn} \label{Def-compact-support}
  \begin{enumerate} 
   \item We define $\widehat {\mathcal L}_{c}(M,\omega)$ (resp. $\widehat \LL_c (M,\omega)$) to be the set of $\widetilde L \in \widehat {\mathcal L}(M, \omega)$ (resp. $L\in \widehat \LL (M,\omega)$)  such that $\gammasupp(L)$ is bounded (hence compact). 
\item %
We define $\widehat \DHam_c (M, \omega)$ to be the set of $\varphi \in \widehat {\DHam} (M, \omega)$ 
such that the $\gamma$-support of  $\Gamma(\varphi)$ is contained in the union of $\Delta_{M}$ and a bounded subset of $M\times \overline M$.
  \end{enumerate} 
  \end{defn} 
  \begin{rem} 
  For an element in $\widehat {\DHam} (M, \omega)$, the $\gamma$-support is just the closure of the $\gamma$-support of  its graph  in $(M\times M, \omega\ominus \omega)$  with the diagonal removed. However there is a natural smaller support that one can define, using the action of $\DHam (M, \omega)$ by conjugation on 
   $\widehat{\DHam} (M, \omega)$. We say that $z$ is in the {\bf restricted $\gamma$-support} of $\varphi \in \widehat{\DHam} (M, \omega)$ if for all $ \varepsilon >0$ there exist $\rho \in \DHam_{c}(B(z, \varepsilon ))$ such that $\rho\varphi\rho^{-1}\neq \varphi$. It is easy to show that the restricted $\gamma$-support is contained in the
   $\gamma$-support. 
  \end{rem} 
  Clearly, by Proposition \ref{Prop-lim-support}, given a sequence $\widetilde L_{k}$ such that the $\gammasupp(L_{k})$ are contained in a fixed bounded set and $\gamma$-converge to $\widetilde L_{\infty}$,  $\widetilde L_{\infty}$ is in 
  $\widehat {\mathcal L}_{c}(M, d\lambda)$.  However the converse is not clear.
    \begin{Question}
  Is an element in  $\widehat {\mathcal L}_{c}(M, d\lambda)$ the limit of a sequence $(L_{k})_{k\geq 1}$ in  ${\mathcal L}(M, d\lambda)$ such that their support is uniformly bounded ?  Same question for $\widehat \DHam_c (M, \omega)$. 
  \end{Question}

We shall make repeated use of  the fragmentation lemma 
\begin{lem} [\cite{Banyaga} Lemma III.3.2] \label{Lemma-frag}
Let $(M, \omega)$ be a closed symplectic manifold and $(U_j)_{\in [1,N]}$ an open cover of $M$. Then any Hamiltonian isotopy $(\varphi^t)_{t\in [0,1]}$ can be written as a product of Hamiltonian isotopies $(\varphi_j^t)_{t\in [0,1]}$ with Hamiltonian supported in some $U_{k(j)}$. The same holds for compact supported Hamiltonian isotopies. 
\end{lem} 
\begin{rems} 
\begin{enumerate} 
\item The number of isotopies is not bounded by the number of open sets: we may have more than one isotopy for each open set. 
\item The lemma is stated in \cite{Banyaga} for compact manifolds, and for $U_j$ symplectic images of ball, but a covering can always be replaced by a finer one by balls. Moreover the proof works for compact supported isotopies with fixed support inside an open manifold and this is how it is stated in \cite{Banyaga-book}, p. 110. 
\item In the sequel, by support of an isotopy we mean the closure of the set $\{z \in M \mid \exists t, \varphi^t(z)\neq z\}$. If the complement of the support is connected and the isotopy is generated by a Hamiltonian $H(t,z)$, this is also the projection on $M$ of the  support of $H$  in $[0,1]\times M$. When the isotopy is implicit, we still write $\supp (\varphi)$ for the support of the (implicit) isotopy and $\widetilde \supp(\varphi)$ (or $\widetilde \supp (\varphi_H)$ or $\supp (H)$)  for the support of the implicit Hamiltonian. Note that in Banyaga's theorem we may assume the support of the Hamiltonians are in the $U_{k(j)}$ (since the complement of a small ball is always connected). 
\end{enumerate}
\end{rems} 

This implies
\begin{lem} \label{Lemma-6.9} We have the following properties :
\begin{enumerate} 
\item  \label{Lemma-6.9-a}  Let $L$ be an element in $\widehat{\LL}(M, \omega)$  and $\varphi\in \DHam_c(M,\omega)$ be such that $\gamma (\varphi(L),L)>0$. Then  $\supp (\varphi)\cap \gammasupp (L) \neq \emptyset$. 
\item  \label{Lemma-6.9-b}  Let $\widetilde L$ be an element in $\widehat{\mathcal L}(M, \omega)$  and $H$ be a compact supported Hamiltonian such that $c(\varphi_H(\widetilde L),\widetilde L)>0$. Then  $ \supp (H)\cap c-\supp (\widetilde L) = \supp (H)\cap \gammasupp (L) \neq \emptyset$. 
\end{enumerate} 
\end{lem} 
\begin{proof} 
\begin{enumerate} 
\item 
Indeed, if this was not the case, for each $x \in \supp (\varphi)$ there would be an open set $U_x$ such that for all isotopies $\psi^t$ supported in $U_x$ we have 
$\gamma (\psi^1(L),L)=0$. But by a compactness argument, we may find finitely many $x_j$ and $U_j=U_{x_j}$ ($1\leq j \leq k$) such that $\supp(\varphi)$ is covered
by the $U_j$. Then $\varphi^1$ is a product of $\psi_j$ supported in $U_{k(j)}$, but since $\gamma (\psi_j(L),L)=0$, we get by induction $$\gamma (\psi_1\circ \psi_2\circ... \circ \psi_{j-1}\circ \psi_j(L), L)=0 $$
and finally $\gamma (\varphi_H(L), L)=0$, a contradiction. 
 \item The second statement is analogous, since $\varphi_H$ is the product of $\varphi_{H_j}$ with $H_j$ supported in $U_{k(j)}$. If we had $c(\varphi_{H_j}(\widetilde L), \widetilde L)=0$ for all  $j$, then we would have $\varphi_H(\widetilde L)=\widetilde L$, a contradiction. Note that here we do not assume that the support of $H$ is small, so we may well have 
 $ \widetilde \supp (H) \neq \supp (\varphi_H) $.
 \end{enumerate} 
\end{proof} 

 \begin{prop} \label{Prop-6.7}
 Let $L_1\in \widehat{\LL} (M,d\lambda), L_2 \in \widehat{\LL}_c(M,d\lambda)$. Then $\gammasupp(L_1) \cap \gammasupp (L_2) \neq \emptyset $. In particular $\gammasupp(L)$ is not displaceable, and intersects any exact Lagrangian. If $L\in \widehat{\LL}_c(T^*,d\lambda)$  it also intersects any fiber $T_x^*N$. 
 \end{prop} 
 \begin{rem} 
 Of course, unless we know somme singular support which does not contain -or is not a $C^0$-limit- of exact smooth Lagrangians, this does not add anything to the known situation, that any two closed exact Lagrangians intersect, a consequence of the Fukaya-Seidel-Smith theorem (see \cite{Fukaya-Seidel-Smith}).   
 \end{rem} 
The proof of Proposition \ref{Prop-6.7} will make use of the following Lemma
 \begin{lem} \label{Lemma-5.9}
 Let $\widetilde L \in \widehat{\mathcal L} (M,d\lambda)$ and $H$ be a Hamiltonian equal to a constant $a$ in a neighbourhood of $\gammasupp (L)$. Then $\varphi_H(\widetilde L)=\widetilde L+a$ (hence $\varphi_H(L)=L$).
 \end{lem} 
 \begin{proof} 
 The conclusion is obvious if $L$ is in $\mathcal L (T^*N)$ by definition of the action of $\DHam_c(T^*N)$ on $\mathcal L (T^*N)$. 
 It is also obvious if $a=0$, since then $H$ is supported in the complement of $c-\supp(L)$ hence $\varphi_H(L)=L$ by 
Lemma \ref{Lemma-6.9}. 
 Now fix $k$ such that  $c(L,L_k) < \varepsilon $ and assume $H_1=a$ on a neighbourhood of $L_k\cup c-\supp(L)$. Then 
 $$c(T_aL_k,\varphi_{H_1}(L))=c(\varphi_{H_1}(L_k),\varphi_{H_1}(L))=c(L_k,L)< \varepsilon 
 $$
 On the other hand $\varphi_{H_1}^{-1}\varphi_H$ is generated by the Hamiltonian $K(t,z)=H(t,\varphi_H^t(z))-H_1(t,\varphi_H^t(z))$ which vanishes on $c-\supp (L)$, so $\varphi_{H_1}^{-1}\varphi_H(L)=\varphi_K(L)=L$, hence $\varphi_H(L)=\varphi_{H_1}(L)$. As a result we have 
 $$c(T_aL_k,\varphi_{H}(L))=c(T_aL_k,\varphi_{H_1}(L)) < \varepsilon $$ so taking the limit for $k$ we get  $c(T_aL,\varphi_{H}(L)=0$ so $\varphi_{H}(L)=T_aL$. 

\end{proof}
\begin{rem} 
Let $H$ be a Hamiltonian equal to $a$ on $\gammasupp (L)$ (but not necessarily in a neighbourhood). Then $H$ is a $C^0$-limit of a sequence $(H_k)_{k\geq 1}$ of Hamiltonians equal to $a$ in a neighbourhood of $\gammasupp(L)$. Therefore $\varphi_{H_k}(\widetilde L)=\widetilde L+a$ and since $\varphi_{H_k} \overset{\gamma}\longrightarrow \varphi_H$ we get $\varphi_H(\widetilde L)=\widetilde L+a$. 
\end{rem} 
 \begin{proof} [Proof of Proposition \ref{Prop-6.7}]
Let $\widetilde L_1, \widetilde L_2$ be elements in $\widehat{\mathcal L}(M,d\lambda)$ having image by $\widehat{\rm unf}$ equal to $L_1,L_2$. Assume their $\gamma$-supports are disjoint and the $\gamma$-support of $L_2$ is compact. 
Let $H$ be a compact supported Hamiltonian, and assume $H$  has support in the complement of $\gammasupp(L_1)$ and equals $a<0$ in a neighbourhood of $\gammasupp(L_2)$. Then 
\begin{gather*} c_+ (\varphi_H^1(\widetilde L_1),\widetilde L_1)\geq c_+(\varphi_H^1(\widetilde L_1), \widetilde L_2) - c_+(\widetilde L_1, \widetilde L_2)= \\ c_+(\widetilde L_1,\varphi_H^{-1}(\widetilde L_2)) - c_+(\widetilde L_1, \widetilde L_2)=-a
\end{gather*} 
since $\varphi_H^{-1}(\widetilde L_2)=\widetilde {L_2}-a$ by the previous Lemma.  As a result $\supp (H)$ intersects $\gammasupp (L_1)$ and this implies that $\gammasupp(L_1)\cap \gammasupp (L_2)\neq \emptyset$. 

It is easy to prove that for $L \in \LL_c(T^*N)$, we have that $\gammasupp(L)$ intersects any vertical fiber $T_x^*N$. Indeed, if this was not the case, we could find a small ball $B(x, \varepsilon )$ such that $T^*(B(x, \varepsilon )) \cap \gammasupp(L) =\emptyset$. 
Now if $f$ is a smooth function such that all critical points of $f$ are in $B(x, \varepsilon )$, then for any bounded set $W$ contained in the complement of $T^*(B(x, \varepsilon ))$ we have $\gra (t df) \cap W=\emptyset $ for $t$ large enough, so $\gra(tdf) \cap L=\emptyset$.  But this contradicts our first statement. 

\end{proof} 

\begin{rem}  Of course, unless we know some singular support which does not contain -or is not a $C^0$-limit- of exact smooth Lagrangians, Proposition \ref{Prop-6.7} does not add anything to the known situation, that any two closed exact Lagrangians in $T^*N$ intersect, a consequence of the Fukaya-Seidel-Smith (see \cite{Fukaya-Seidel-Smith}).
\end{rem} 
 \begin{Question}  What can $\gammasupp (L)$ be ? 
 \end{Question}
One of the goals of this paper is to partially answer this question. 
  \begin{example}
  Let $f\in C^0(N, {\mathbb R})$. Then $\gammasupp (\gra (df)) = \partial f$ where $\partial f$ is the subdifferential defined by Vichery in \cite{Vichery} (see \cite{ AGHIV} for a proof). Therefore $\partial f$ intersects any exact Lagrangian $L$. If $L$ is isotopic to the zero section, 
  then $L$ has a \GFQI  and  $\partial f \cap L$ is given by the critical points of $S(x,\xi)-f(x)$. 
  
  \end{example}

In \cite{Humiliere-completion} section 2.3.1. a different notion of support was presented:
\begin{defn} [H-support, \cite{Humiliere-completion}] \index{$H$-support} A point $x\in M$ is in the complement of the  $H$-support of $L \in \widehat {\LL}(M, \omega)$ if there is a sequence of smooth Lagrangians $(L_k)_{k\geq 1}$ converging to $L$ and a neighbourhood $U$ of $x$ such that $L_k\cap U = \emptyset$. 
\end{defn} 
We shall first prove that our definition of the support yields an {\it a priori} smaller set than Humili{\`e}re's. 
 We shall actually prove something slightly more general.
 \begin{defn}  Let $X_{k}$ be a sequence of subsets in the topological space $Z$. We define  its topological upper limit  as
\begin{align*}
  \limsup_{j} X_{j}&=\bigcap_n \overline{\bigcup_{j\geq n}X_{j}} \\
  &=\left\{x\in Z \mid \exists (x_j)_{j\geq 1}, \;x_j\in X_j\;  \text{for infinitely many}\; j, \; \lim_j x_j=x\right \}
\end{align*}
and its topological lower limit as 
$$\liminf_j X_j=\left\{x\in Z \mid \exists (x_j)_{j\geq 1}, x_j\in X_j, \lim_j x_j=x\right \}$$
\end{defn} 
Note that if $\liminf_j X_j=\limsup_j X_j$ then this is the topological limit, which for a compact metric space coincides with the Hausdorff limit (see \cite{Kechris}, p.25-26), but we shall not need this here. Also 
$x\in \liminf_jX_j$ if and only if $\lim_j d(x,X_j)=0$. Note that this limit could in principle be empty, even if the $X_j$ are non-empty. 
An easy result is now
\begin{prop} \label{Prop-lim-support}
Let $(L_k)_{k\geq 1}$ be  sequence in $\widehat {\LL} (M, \omega)$ of Lagrangians such that $\gamma-\lim_k (L_k)=L$. Assume $\gammasupp (L_k)\cap U=\emptyset$ for infinitely many $k$. Then $$\gammasupp(L)\cap U=\emptyset$$
In other words
$$ \gammasupp (L) \subset \liminf_k (\gammasupp (L_k))$$
 \end{prop} 
\begin{proof} 
For the first statement let $\varphi$ be supported in $U$. Then for $k$ in the subsequence $U\cap \gammasupp(L_k)=\emptyset$. This implies that $\gamma (\varphi(L_k),L_k)=0$, and passing to the limit, that $\gamma (\varphi(L),L)=0$. 

Now if $x\notin \liminf_k (\gammasupp (L_k))$, there must be a subsequence such that 
$d(x, \liminf_k (\gammasupp (L_k)))> \varepsilon _0>0$. But then on this subsequence $B(x, \varepsilon_0/2)\cap \gammasupp(L_k)=\emptyset$ 
and applying the previous result, we have $B(x, \varepsilon_0/2)\cap \gammasupp(L)=\emptyset$, so $x \notin \gammasupp(L)$. 
\end{proof} 
For the case where $L$ and the $L_{k}$ are  smooth, this remark was previously made by  Seyfaddini and the author (see \cite{Vichery}) and is also a consequence of lemma 7 in \cite{Hum-Lec-Sey2}
From  Proposition \ref{Prop-lim-support} we immediately conclude
\begin{cor} \label{Cor-3.15}
We have for $L \in \widehat {\LL}(M, \omega)$ the inclusion
$$ \gammasupp (L) \subset H-\supp (L)$$
\end{cor}

Our definition of support has one  advantage compared to  Humili{\`e}re's definition :
if $W\cap H-\supp(L)=\emptyset$, we do not know whether there is a sequence $(L_k)_{k\geq 1}$ $\gamma$-converging to $L$ and such that $L_k\cap W= \emptyset$, we only know that $W$ can be covered by sets $W_j$ such that for each $j$ there is a sequence $(L_k^j)_{k\geq 1}$ such that $\gamma-\lim_k L_k^j=L$ and $L_k^j\cap W_j=\emptyset$. On the other hand if $W \cap \gammasupp (L)=\emptyset$, we know that for all $\varphi$ supported in $W$ we have $\gamma (\varphi(L),L)=0$. This still leaves open the following
\begin{Question} Do we have equality in Corollary \ref{Cor-3.15}, in other words
is any $ L\in \widehat {\LL}(M, \omega)$ the limit of a sequence $(L_{k})_{k\geq 1}$ contained in a neighbourhood of $ \gammasupp (L) $ ? 

\end{Question}

Note that the Hamiltonians we need to consider to determine whether a point is in the support can be restricted to a rather small family. 
Let $\chi(r)$ be a non-negative smooth  function function equal to $1$ on $[-1/2, 1/2]$ and supported in $[-1,1]$. We set $H^\chi_{z_0, \varepsilon }(z)=\rho ( \frac{1}{ \varepsilon }d(z,z_0)$) and $\varphi_{z_0, \varepsilon }^t$ to be the flow associated to $H_{z_0, \varepsilon }^\chi$.
Since for any $\varphi$ supported in $B(z_0, \varepsilon/2 )$ we can find a positive constant $c$ so that $\varphi \preceq \varphi_{z_0, \varepsilon }^c$, that is $c_-(\varphi_{z_0, \varepsilon }^c\circ \varphi^{-1})=0$, we have for a lift $\widetilde L$ of $L$, 
$c_+(\varphi(\widetilde L),\widetilde L) \leq c_+(\varphi_{z_0, \varepsilon }^c(\widetilde L), \widetilde L)$, so $0< c_+(\varphi(\widetilde L),\widetilde L)$ implies $0 < c_+(\varphi_{z_0, \varepsilon }^c(\widetilde L), \widetilde L)$. Note that if $c_-(\varphi(\widetilde L),\widetilde L)<0$ we have 
$0<c_+(\widetilde L, \varphi(\widetilde L))= c_+(\varphi^{-1}(\widetilde L),\widetilde L)$, so applying the same argument to $\varphi^{-1}$ we almost get
\begin{prop}[Criterion for the $\gamma$-support]\label{Prop-3.8}
A point $z$ is in the $\gamma$-support of $\widetilde L \in \widehat {\mathcal {L}} (M, \omega)$ if and only if we have 
$$ \forall \varepsilon >0,\; \; c_+(\varphi_{z_0, \varepsilon }(\widetilde L),\widetilde L) )>0$$ where $\varphi_{z_0, \varepsilon }$ is the time-one map associated to $H_{z_0, \varepsilon }^\chi$.
 \end{prop}
 \begin{proof} 
 The only point missing is to show that we can take $c=1$. But let $n$ be an integer such that $c<n$. Assume we had $c_+(\varphi_{z_0, \varepsilon }(\widetilde L), \widetilde L)=0$. Then, since $\varphi_{ z_0, \varepsilon }^{n-c}\succeq \Id$, we have
 $$ c_+(\varphi_{z_0, \varepsilon }^c(\widetilde L), \widetilde L) \leq c_+(\varphi_{z_0, \varepsilon }^n(\widetilde L), \widetilde L) \leq \sum_{j=1}^n c_+(\varphi_{ z_0, \varepsilon }^j(\widetilde L), \varphi_{ z_0, \varepsilon }^{j-1}(\widetilde L)) =n 
 c_+(\varphi_{z_0, \varepsilon }(\widetilde L), \widetilde L)=0$$ This contradicts our assumption that $ c_+(\varphi_{z_0, \varepsilon }^c(\widetilde L), \widetilde L)>0$. 
 \end{proof} 
 Let $A\subset M_1\times \overline{M_2}, B\subset  M_2\times \overline{M_3}$. We set
 $A\circ B$ to be the projection of  $\left ( A\times B \right ) \cap \left (M_1\times \Delta_{M_2}\times M_3\right )$ on
 $M_1\times M_3$. 
We then have
\begin{prop} \label{Prop-6.21}
We have
 \begin{enumerate} 
 \item \label{Prop-6.21-c} For $L$ a smooth Lagrangian in $(M, \omega)$ we have $\gammasupp(L)=L$. 
\item \label{Prop-6.21-a} $\gammasupp (L)$ is non-empty
\item \label{Prop-6.21-b}
For $\psi$ a symplectic map,  $\gammasupp(\psi(L))=\psi (\gammasupp(L))$
\item \label{Prop-6.21-d}  $\gammasupp (L_1\times L_2) \subset \gammasupp(L_1)\times \gammasupp (L_2)$ for $L_i\in \widehat{\mathfrak L}(M_i,d\lambda_i)$
\item \label{Prop-6.21-f} Let $K$ be coisotropic in $(M, \omega)$  and $\mathcal K$ its isotropic foliation. If $K/\mathcal K=P$ is a symplectic manifold, and $L\in {\widehat\LL} (M,\omega)$ having reduction $L_{K} \in {\widehat\LL} (P, \omega_{K})$, then $$\gammasupp(L_{K})\subset (\gammasupp (L))_{K}$$ where for a coisotropic submanifold $K$ having a well defined reduction (i.e. its quotient by the isotropic foliation is a smooth manifold) and for $X\subset M$ we set $X_K=X\cap K/{\mathcal K} $ (where ${\mathcal K}$ is the equivalence relation defined by the isotropic foliation). 
\item  \label{Prop-6.21-e} Let $\Lambda_1, \Lambda_2$ be correspondences, that is elements in $\widehat{\mathfrak L}(M_1\times \overline{M_2})$ and  $\widehat{\mathfrak L}(M_2\times \overline{M_3})$ respectively. 
Then provided $\Lambda_1\circ \Lambda_2 \in \widehat{\mathfrak{L}} ( M_1\times \overline{M_3})$  it satisfies  $\gammasupp (\Lambda_1\circ \Lambda_2) \subset \gammasupp(\Lambda_1) \circ \gammasupp (\Lambda_2)$
\end{enumerate} 
\end{prop} 
\begin{rem} 
The condition ``$\Lambda_1\circ \Lambda_2 \in \widehat{\mathfrak{L}} ( M_1\times \overline{M_3})$'' means that there exists sequences $(L_1^\nu)_{\nu \geq 1}, (L_2^\nu)_{\nu \geq 1}$ such that  $\gamma-\lim L_j^\nu =\Lambda_j$, the $L_j^\nu\circ L_2^\nu$ are exact Lagrangian branes and $\gamma-\lim L_j^\nu\circ L_2^\nu$ converges in  $\widehat{\mathfrak{L}}$. Its limit is then denoted  $\Lambda_1\circ \Lambda_2$ and it is easy to check that the limit does not depend on the choice of the sequences $(L_1^\nu)_{\nu \geq 1}, (L_2^\nu)_{\nu \geq 1}$. 
\end{rem} 
\begin{proof} 
 For (\ref{Prop-6.21-c}),  the  inclusion $\gammasupp(L)\subset L $ follows from Corollary \ref{Cor-3.15}. The converse can be reduced to the case of the zero section by using a Darboux chart $U$ in which $(U,L\cap U)$ is identified with $( B^{2n}(0,r), B^{2n}(0,r)\cap {\mathbb R}^n)$. But then if $f=0$ outside a neighbourhhood of $z$ and the oscillation of $f$ is  $\varepsilon^2 , \vert \nabla f \vert \leq \varepsilon $ 
we can locally deform the zero section to 
$\Gamma_{df}$
in a neighbourhood of $z$, and obtain $L'$ such that $\gamma (L',L)\geq \varepsilon ^2$. 
 \begin{figure}[ht]
\centering\begin{overpic}[width=5cm]{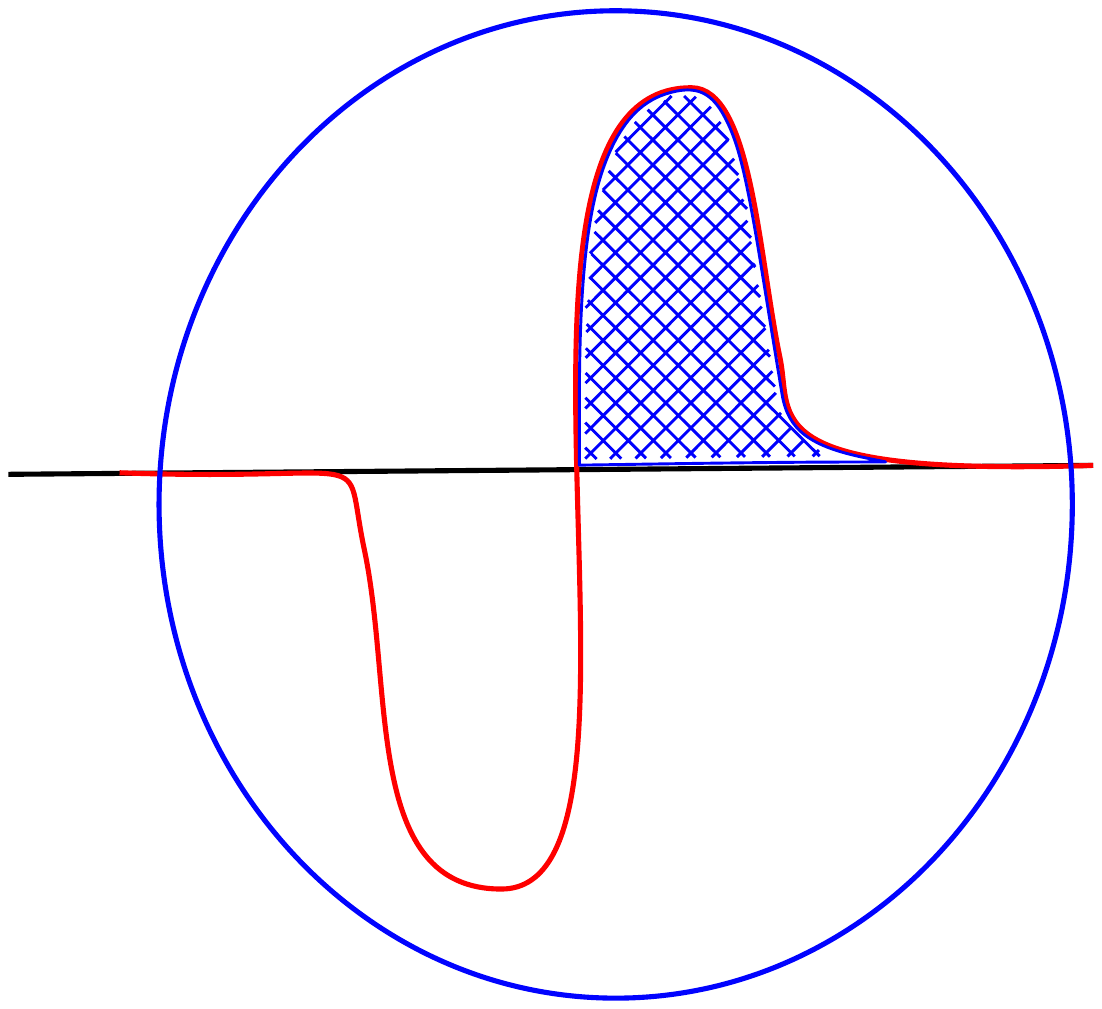}
  \put (40,40) {$ L$} 
   \put (70,55) {$\myRed {L'}$ } 
 \end{overpic}
\caption{ $L$ and $L'=\varphi(L)$ with $\varphi$ supported in the interior of the disc. The hatched region has area $ \varepsilon^{2}$ }
\label{fig-1n}
\end{figure}

For (\ref{Prop-6.21-a}) this follows from Proposition \ref{Prop-6.7}. 

For (\ref{Prop-6.21-b}), we have $\gamma (\varphi \psi(L), \psi(L))>0 \Longleftrightarrow \gamma (\psi^{-1}\varphi\psi (L), L ) >0$ and since $\supp (\psi^{-1}\varphi\psi)=\psi^{-1}( \supp (\varphi))$ this proves our statement. 

For (\ref{Prop-6.21-d}), we
note that if $\gamma (\varphi(L),L)>0$ we either have $c_+(\varphi(L),L)>0$ or $c_-(\varphi(L),L)>0$. Changing $\varphi$ to $\varphi^{-1}$, we can always assume 
$c_+(\varphi(L),L)>0$. 
Then suppose there is a $\psi$ supported in $U_1\times M$ such that $c_+(\psi(L_1\times L_2), L_1\times L_2)>0$. Then we can find $\varphi_1$ supported in $U_1$ such that  $$\varphi_1\times \Id \succeq \psi$$  just replace the Hamiltonian generating $\psi$, $K(z_1,z_2)$ by $H_1(z_1)$ such that $K(z_1,z_2) \leq H(z_1)$ (note that $K$ is compact supported). Then 
\begin{gather*} c_+(\psi(L_1\times L_2), L_1\times L_2) \leq c_+((\varphi_1\times \Id)(L_1\times L_2), L_1\times L_2)=\\ c_+(\varphi_1(L_1),L_1)+ c_+(L_2,L_2)=c_+(\varphi_1(L_1),L_1)=0
\end{gather*} 
 a contradiction. 
 So we proved that $U_1\cap \supp(L_1)=\emptyset$ implies $U_1\times M \cap \supp (L_1\times L_2)=\emptyset$. The same holds by exchanging the two variables, so we get $$\supp (L_1\times L_2) \subset \supp (L_1)\times \supp (L_2)$$ 
 
 For  (\ref{Prop-6.21-f}) let $z\in \gammasupp(L_{K})$ so there exists $\varphi\in \DHam_{c}(B(z, \varepsilon ))$ such that $\gamma(\varphi(L_K),L_K)>0$. We want to prove that $z\in (\gammasupp(L))_K$.
 
  Let $U_{ \varepsilon }$ be the preimage of $B(z, \varepsilon )$ in $K$ and $\tilde\varphi\in \DHam(M,\omega)$ be a lift of $\varphi$, i.e.  $\widetilde\varphi$ preserves  $K$  and projects to $\varphi$. If $\varphi$ is the flow of $H$ we set $\widetilde H$ to be an extension of $H\circ \pi$ (where $\pi: K \longrightarrow K/\mathcal K$ is the projection) and set $\widetilde \varphi$ to be the flow of $\widetilde H$. Now the neighbourhood of the leaf $I=\mathcal K_{z}$ can be identified to a neighbourhood of 
 $I\times \{z\}$ in $T^{*}I\times B(z, \varepsilon )$. Then by the reduction inequality \cite{Viterbo-book}, Prop 7.60 and the equality 
 $(\widetilde\varphi(L))_{K}=\varphi(L_{K})$, we have
 $$\gamma (\widetilde\varphi(L),L) \geq \gamma (\varphi(L_{K}),L_{K})$$
 This inequality holds {\it a priori} in the smooth case, but extends to the completion by continuity. 
 Indeed, if $L^{\nu}$ $\gamma$-converges to $L$, the sequence $L_{K}^{\nu}$ $\gamma$-converges to $L_{K}$ and 
$ \gamma (\varphi(L^{\nu}_{K}),L^{\nu}_{K})>\delta_{0}>0$ implies $$ \gamma (\widetilde\varphi(L^{\nu}),L^{\nu})\geq  \gamma (\varphi(L^{\nu}_{K}),L^{\nu}_{K})>\delta_{0}>0$$ that is 
$\gamma ((\widetilde\varphi(L),L)>\delta_{0}>0$. Therefore by Proposition \ref{Lemma-6.9} this implies $U_\varepsilon\cap \gammasupp(L)=\supp(\widetilde \varphi) \cap \gammasupp(L)\neq \emptyset$ which implies $\gammasupp(L)_K \cap B(z, \varepsilon )\neq \emptyset$ and finally $z\in \gammasupp(L)_K$.

For  (\ref{Prop-6.21-e}) we notice that $\Lambda_1\circ \Lambda_2$ is the composition of the product  $\Lambda_1\times \Lambda_2$ and a reduction of $M_1\times \overline{M_2}\times M_2 \times \overline{M_3}$  by $M_1\times \Delta_{M_2}\times \overline{M_3}$, where $\Delta_{M_2}$ is the diagonal of $M_2$.

 Thus it is enough to prove that the $\gamma$-support of the reduction is contained in the reduction of the $\gamma$-support, which we just proved. 
 
\end{proof} 

Let us consider the group $ \mathcal {H}_\gamma (M, \omega)$ \index{$ \mathcal {H}_\gamma (M, \omega)$} of homeomorphisms preserving $\gamma$. In particular since $\gamma$ is $C^0$-continuous (see \cite{BHS-C0}) this group contains the group of homeomorphisms which are $C^0$-limits of symplectic diffeomorphisms,  denoted ${\Homeo} (M, \omega)$\index{${\Homeo} (M, \omega)$}, usually called the group of  symplectic homeomorphisms. We have $ {\Homeo}(M, \omega)\subset \widehat{\mathscr{DHam}}(M, \omega)$
\begin{cor} \label{Cor-6.23}
Let $\psi \in \mathcal {H}_\gamma (M, \omega)$. Then we have 
$$\gammasupp (\psi(L))=\psi( \gammasupp(L))$$
\end{cor}
\begin{proof} 
A point $x$ is in $\gammasupp(\psi(L))$ if and only if for any neighbourhood $U$ of $x$ we can find $\rho\in \DHam_c(U)$ such that 
$\gamma(\rho(\psi(L),\psi(L))>0$. Since $\psi$ preserves $\gamma$ this is equivalent to $\gamma(\psi^{-1}\circ \rho\circ \psi(L), L) >0$. 
Since $\supp ((\psi^{-1}\circ \rho\circ \psi) = \psi^{-1}(U)$, we have $\psi(U)\cap \gammasupp(L)\neq \emptyset$.  Now given $V$ a neighbourhood of $\psi^{-1}(x)$ we can find a neighbourhood $U$ of $x$ such that $\psi^{-1}(U)\subset V$. Thus  $\psi^{-1}(x) \in \gammasupp(L)$
hence $x\in \psi(\gammasupp(L))$. 
This yields one inclusion. 
Applying the same argument to $\psi^{-1}$, gives the other inclusion. 
\end{proof} 
\begin{rem} 
 When $L$ is smooth then $\DHam_{c}(M,\omega)$ acts transitively on $\gammasupp (L)$, and even better, the restriction map $\DHam_{c} (M, \omega) \longrightarrow \Diff_0(L)$ is onto. This is not at all the case in general as the following example shows. 
\end{rem} 

 \begin{example} 
 Let us consider the sequence of smooth Lagrangians in $T^*S^1$ represented below
  \begin{figure}[H]\centering
 \begin{overpic}[width=5cm]{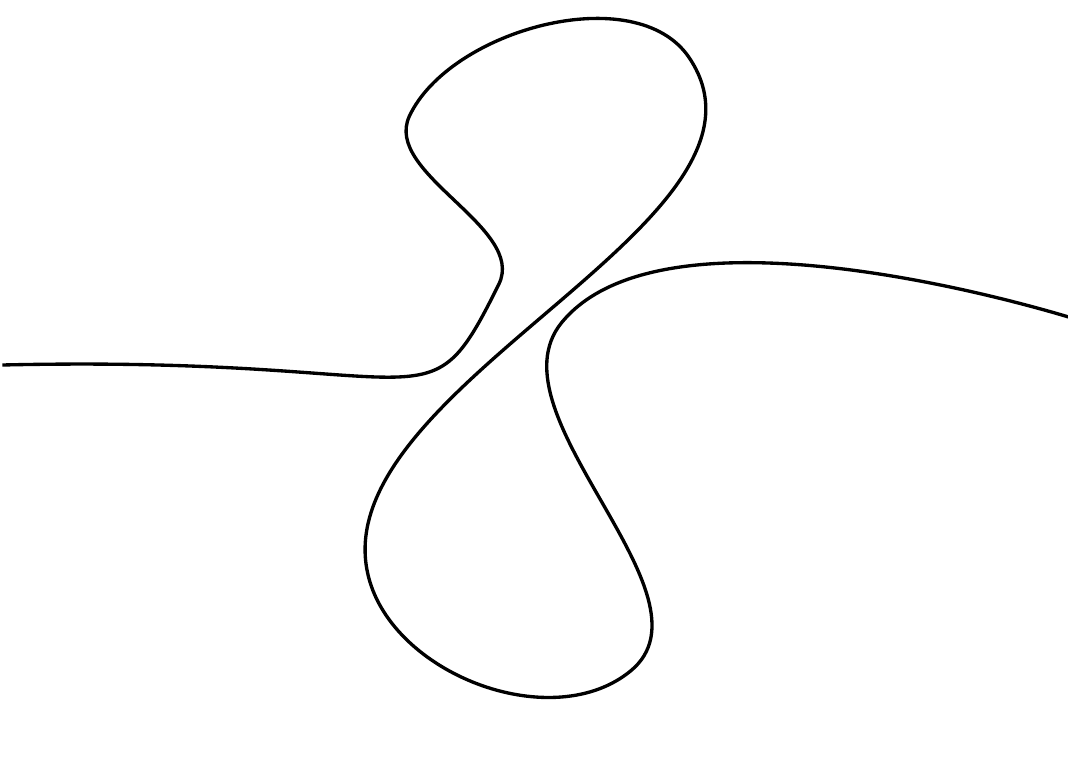}
  \put (20,30) {$L_k$} 
 \end{overpic}
\begin{overpic}[width=5cm]{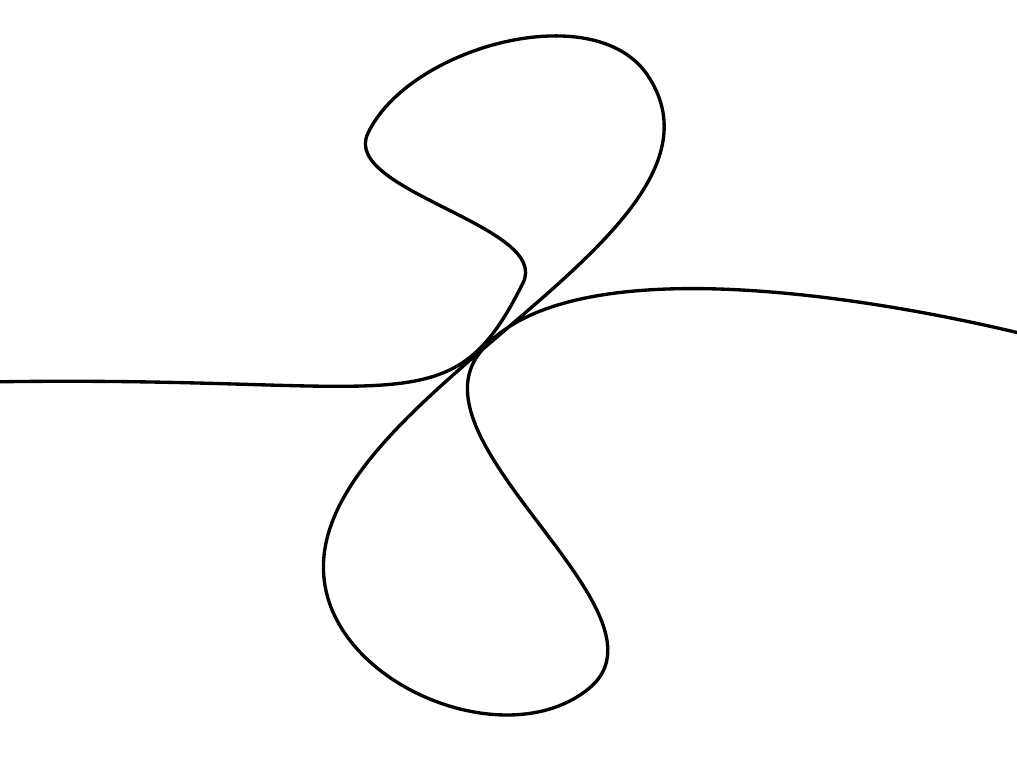}
  \put (20,30) {$\Lambda$} 
 \end{overpic}
\caption{The sequence $(L_k)_{k\geq 1}$ and the  $\gamma$-support of the limit $\Lambda$ of the sequence $(L_k)_{k\geq 1}$ }
\label{fig-2n}
\end{figure}
\begin{enumerate} 
\item 
Quite clearly there is a singular point, so $\DHam_{c}(M,\omega)$ cannot act transitively on $\gammasupp(\Lambda)$.

\item We can arrange that $\Lambda$ is symmetric with respect to the involution $(x,p)\mapsto (x,-p)$, i.e. $\overline {\gammasupp(\Lambda)}=\gammasupp(\Lambda)$. However $\overline \Lambda \neq \Lambda$ because
as we see from Figure \ref{fig-2n-c}, $\gamma (L_k, \overline L_k)$ remains bounded from below (by twice the area of the loop). As a result the elements $\Lambda$ and $\overline \Lambda$ in  $\widehat{\mathcal L}(T^*N)$  have the same $\gamma$-support while being distinct. 
  \begin{figure}[H]\centering
 \begin{overpic}[width=5cm]{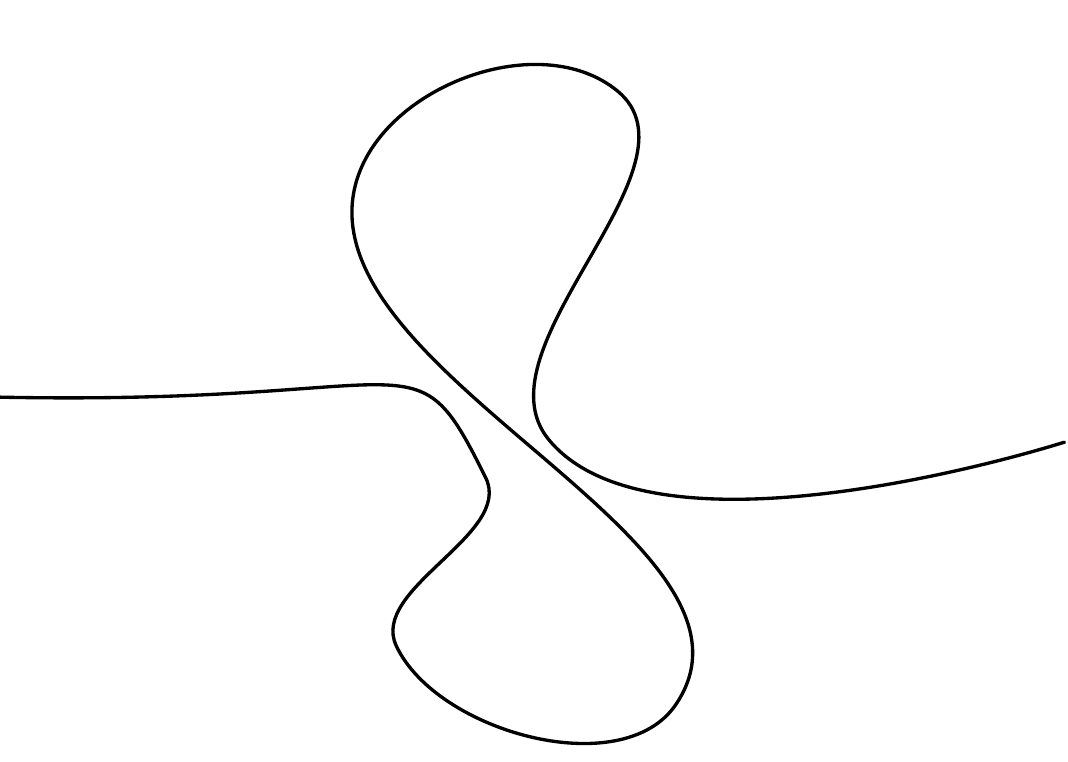}
  \put (20,20) {$\overline L_k$} 
 \end{overpic}
 \caption{The sequence $(\overline L_k)_{k\geq 1}$}
 \label{fig-2n-c}
 \end{figure}
\item We conjecture that the set of elements in $\mathfrak L(T^{*}S^{1})$ such that $\gammasupp(L)=\gammasupp(\Lambda)$ is $\{ \Lambda, \overline \Lambda\}$
\end{enumerate} 
  \end{example} 
    
  Finally note that we can define a kind of density of $L\in {\mathfrak L}(M, \omega)$ as follows
  \begin{defn} 
  Let $L \in {\mathfrak L}(M, \omega)$. We define $$\rho_{\varepsilon}(z,L)=\sup \left\{\gamma(\varphi(L),L) \mid \varphi \in \DHam_{c} (B(z, \varepsilon))\right \}$$ 
  and $\rho(z,L)$ to be the limit $$\limsup_{ \varepsilon \to 0} \frac{1}{-2\log (\varepsilon )}\log \rho_{ \varepsilon }(z,L)$$
  \end{defn} 
  It is easy to show that  for $\psi \in \DHam_{c} (M, \omega)$ we have 
  $\rho (\psi(z),\psi(L))=\rho(z,L)$ so that for a smooth Lagrangian, $\rho(z,L)=1$, since this is the case for a Lagrangian vector space in $ {\mathbb R}^{2n}$. 
  We leave the proof of the following to the reader
  \begin{prop} 
  Let $(L_{k})_{k\geq 1}$ be a converging sequence in $ \widehat {\LL} (M, \omega)$ having limit $L$. Assume there exists $ \varepsilon, \delta >0$  such that  $\rho_{ \varepsilon }(z,L_{k}) \geq \delta$ for all $k$. Then, $z \in \gammasupp(L)$ and $\rho_{ \varepsilon }(z,L)\geq \delta$. 
  
  \end{prop} 
\section{The \texorpdfstring{$\gamma$-coisotropic}{gamma-coisotropic} subsets in a symplectic manifold: definitions and first properties}\label{Section-Coisotropic}
We now define the  notion of $\gamma$-coisotropic subset in a symplectic manifold. Of course this notions will coincide with usual coisotropy in the case of smooth submanifolds. 
 We start with a new definition that will play a central role in this paper. After this paper was written, we realized that the analogue of $\gamma$-coisotropic- with $\gamma$ replaced by the Hofer distance (and under the name of ``locally rigid'') had already been defined in \cite{Usher}. Using such a notion, the property that for a submanifold, the equivalence between local rigidity and being coisotropic was stated by Usher and he proved that the Humili{\`e}re-Leclercq-Seyfaddini   theorem (\cite{Hum-Lec-Sey2}) follows from the definition and its invariance by symplectic homeomorphism. 
\subsection{Basic definitions}
\begin{defn} \label{Def-coisotropic}\index{$\gamma$-coisotropic}
Let $V$ be a subset of $(M, \omega)$ and $x\in V$. We shall say that $V$ is {\bf non-$\gamma$-coisotropic at $x$} if for any ball $B(x, \varepsilon )$ there exists a ball  $B(x,\eta) \subset B(x, \varepsilon )$ and a sequence $(\varphi_k)_{k\geq 1}$ of Hamiltonian maps supported in $B(x, \varepsilon )$ such that $\gamma-\lim \varphi_k=\Id$ and we have  $\varphi_k(V)\cap B(x, \eta )=\emptyset$.

In other words a  set  $V$ is {\bf $\gamma$-coisotropic at $x\in V$} if there exists $ \varepsilon >0$ such that for any ball $B(x,\eta)$ with $0< \eta< \varepsilon $ there is a $\delta (\eta)>0$ such that for all $\varphi \in \DHam_{c}(B(z, \varepsilon ))$ such that  %
$\varphi(V) \cap B(x, \eta)=\emptyset$  we have $\gamma (\varphi)> \delta(\eta)$. 

We shall say that $V$ is {\bf $\gamma$-coisotropic} if it is non-empty and $\gamma$-coisotropic at each $x\in V$. It is nowhere $\gamma$-coisotropic if each point $x\in V$ is non-$\gamma$-coisotropic. 

\end{defn} 

\begin{rem} 
\begin{enumerate} 
\item  A set is $\gamma$-coisotropic at a point if ``locally displacing it away' from itself'' in a neighbourhood of that point cannot be achieved by an arbitrarily small -as measured by $\gamma$- Hamiltonian diffeomorphisms. 
\item 
A variation of this definition assumes in the definition of non-$\gamma$-coisotropic, that $\varphi(V)=V'$ is fixed. This would be better in some instances, but the {\bf locally hereditary} property, explained  below becomes non-obvious (if at all true). 
Note that our definition should make the empty set to be $\gamma$-coisotropic, however we explicitly excluded this case. 
\item Note that we do not need to define the spectral distance on $(M,\omega)$ to define $\gamma$-coisotropic subsets. Indeed, we only need to consider elements in 
$\DHam_{c}(B(x, \varepsilon))$ for $ \varepsilon $ small enough. But for this, using a Darboux chart, it is enough to have $\gamma$ defined in $\DHam_{c}( {\mathbb R}^{2n})$. 
\end{enumerate} 
 \end{rem} 
\begin{examples}
\begin{enumerate} 
\item A point in $ {\mathbb R}^{2}$ is non-$\gamma$-coisotropic. Let $H(q,p)$ be such that $ \frac{\partial H}{\partial q_{1}}(0,p)$ is  sufficiently large, but $H$ is $C^{0}$-small. Then the flow of $H$ takes the origin outside of $B(0, \eta)$. By truncation, we can do this while being compact supported in $B(0, \varepsilon )$. 
\item If $I=[0,1]$ in $ {\mathbb R}^{2}$ it is easy to see that interior points of the interval are $\gamma$-coisotropic, while the boundary points are not. Thus $I$ is not $\gamma$-coisotropic. 
\item It is not hard to see that $\{(0,0)\} \times {\mathbb R}^{2n-2}$ is not $\gamma$-coisotropic, and neither is $I\times \{0\}\times {\mathbb R}^{2n-2}$. However this last set
is $\gamma$-coisotropic at the point $(x,0,...,0)$ if and only if $0<x<1$. More generally $D^{n+r}(1)\times\{0\}$ is $\gamma$-coisotropic for $1\leq r\leq n$ but not for $r=0$.
\item A smooth Lagrangian submanifold without boundary  is $\gamma$-coisotropic. 
  \end{enumerate} 
 \end{examples}
 \begin{Question}
Let $V$ be a $\gamma$-coisotropic subset. Does $V$ have positive displacement energy ? Even better, is it true that $\gamma(V)=\inf \{\gamma (\psi) \mid \psi(V)\cap V=\emptyset \} >0$ (see \cite{Bolle-1, Ginzburg-coisotropic, Tonnelier-these} for results in this direction in the smooth case). Note that the answer is in general negative (see \cite{Usher}). 
\end{Question}
One of the goals of this paper is to show a number of natural occurrences of $\gamma$-coisotropic sets, starting  from the $\gamma$-support of elements of $\widehat{\LL}(M,\omega)$ (see Theorem \ref{Thm-7.12} and \cite{Guillermou-Viterbo}).

We denote by $\mathcal L_\gamma(M, \omega)$ the set of images of smooth Lagrangians by elements of $ \mathcal {H}_\gamma (M, \omega)$ and by $\mathcal S_{2}(M, \omega)$ the set of images of  symplectic submanifolds (not necessarily closed) of codimension greater or equal to $2$ by elements of $ \mathcal {H}_\gamma (M, \omega)$. Note that these are subsets (and even topological submanifolds) of $M$, which is not the case for elements in $\widehat {{\LL}} (M,\omega)$.

The following is quite easy:
\begin{prop} \label{Prop-7.5}
We have the following properties
\begin{enumerate} 
\item \label{Prop-7.5-1} Being $\gamma$-coisotropic is invariant by $ \mathcal {H}_\gamma (M, \omega)$, hence by ${\Homeo} (M, \omega)$ (and also by conformally symplectic homeomorphisms). 
\item \label{Prop-7.5-2} Being $\gamma$-coisotropic is a {\bf local property} in $M$. It only depends on a neighbourhood of $V$ in $(M, \omega)$.
\item \label{Prop-7.5-3}  If $X,Y$ are $\gamma$-coisotropic then $X\cup Y$ is $\gamma$-coisotropic. 
\item \label{Prop-7.5-4} $\gamma$-coisotropic implies locally rigid (in the sense of \cite{Usher}).
\item \label{Prop-7.5-5} Being $\gamma$-coisotropic  is {\bf locally hereditary} in the following sense : if through every point $x \in V$ there is a $\gamma$-coisotropic subset $V_x\subset V$ containing $x$, then $V$ is $\gamma$-coisotropic. In particular if through any point of $V$ there is a coisotropic or a local Lagrangian submanifold,  then $V$ is $\gamma$-coisotropic. If  any point in $V$ has a neighbourhood contained in an element of $\mathcal S_2(M, \omega)$ then $V$ is non-$\gamma$-coisotropic. 

\end{enumerate} 
\end{prop} 
\begin{proof} 
The first two statements as well as (\ref{Prop-7.5-3})  are obvious from the definition. The assertion (\ref{Prop-7.5-4}) follows immediately from the inequality $\gamma \leq d_H$ between Hofer distance and spectral norm. For (\ref{Prop-7.5-4}),  the only non-obvious statement is that a smooth Lagrangian germ is coisotropic at any interior point,  while a positive codimension symplectic submanifold is everywhere non-$\gamma$-coisotropic. The first part of the fifth statement

For a Lagrangian,  if $B(x_0, 1 )$ is a ball and $L\cap B( x_{0},1 )= {\mathbb R}^n \times \{0\}$ and $L'\cap B(x_{0},1)= \emptyset$, then $\gamma (L, L') \geq \pi/2$. Indeed it is easy to construct a Hamiltonian isotopy supported in $B(x_{0},1)$ such that $\gamma (L_t,L)= \pi/2$ 
 and $L_t=L$ outside $B(1)$ (see Figure \ref{fig-1n}).

 As a result $$\pi/2=c_+ (L_1, L')\leq c_+ (L_1, L')+c_+(L',L) \leq c_+ (L, L')-c_-(L,L')=\gamma (L,L') $$
But this implies $ \gamma (\varphi) \geq \gamma (L',L) \geq \pi/2$ and this concludes the proof in the Lagrangian case, since by a conformal map we can always reduce to this case. 
 
Now let $S$ be symplectic of codimension $2$, we want to prove that $S$ is nowhere coisotropic. Up to taking a subset of $S$ and after a symplectic change of coordinates (and possibly a dilation), we may assume that $S, U$ are identified locally to  
 $$S_0=\left \{(0,0, \bar q, \bar p) \in D^{2}(1) \times D^{2n-2}(1) \right \} \subset D^{2}(1) \times D^{2n-2}(1)=U$$ 
 where $D^{2}(r)$ (resp. $D^{2n-2}(r)$) are the symplectic balls of radius $r$ in $({\mathbb R}^2, \sigma_{2})$ (resp.in $({\mathbb R}^{2n-2}, \sigma_{2n-2})$).
 Consider the isotopy $$t \mapsto ( t a(\bar q, \bar p), 0, \bar q, \bar p)$$ where $a(\bar q, \bar p)= A( \vert \bar q \vert ^2+ \vert \bar p \vert ^2)$ and $A$ is a compactly supported function  bounded by $1$, and equal to $1$ in $D^{2n-2}(1)$. This is a Hamiltonian isotopy generated by $H( q_1,p_1, \bar q, \overline p)=\chi (\overline q, \overline p)a(\widehat q, \widehat p)$, where $(\overline q, \overline p)\in  {\mathbb R}^{2k}$ and $(\widehat q, \widehat p)\in  {\mathbb R}^{2n-2k}$,
  sending $S_0$ to $S_1$, where $$S_1 =\left \{(1,0, \bar q, \bar p) \in D^{2k}(1) \times D^{2n-2k}(1) \right \}$$ and in particular $S_0\cap S_1=\emptyset$. 
 Let $$H(q_1,\bar q, p_1,\bar p)= \chi(q_1,p_1)\cdot a(\bar q, \bar p)$$ be such that  $\frac{\chi(q_1,0)}{\partial q_1} =0$ and $\frac{\chi(q_1,0)}{\partial p_1} =1$ for $ \vert q_1\vert \leq 1$ and  
such that $ \Vert \chi \Vert_{C^0} \leq \varepsilon $. The flow of $H$ is given by 
  \begin{displaymath}   \left \{
  \begin{array}{ll}
 \dot q_1 &= \frac{\partial \chi}{\partial p_1}(q_1(t), p_1(t))  A (\vert \bar q(t)\vert ^2+  \vert \bar p (t) \vert ^2)\\
 \dot p_1&=  - \frac{\partial \chi}{\partial q_1}(q_1(t), p_1(t)) A (\vert \bar q(t)\vert ^2+  \vert \bar p (t) \vert ^2)\\
 \dot {\bar q} (t) &= 2\chi(q_1(t),p_1(t)) A'(\vert \bar q(t)\vert ^2+  \vert \bar p (t) \vert ^2)\bar p (t)\\
 \dot {\bar p}(t)&=-2\chi(q_1(t),p_1(t)) A'(\vert \bar q(t)\vert ^2+  \vert \bar p (t) \vert ^2)\bar q (t)\\
\end{array}
\right .  \end{displaymath} 

The last two equations imply that $ \vert \bar q \vert ^2+ \vert \bar p \vert ^2$ is constant, hence $a( \bar q(t), \bar p(t))$ is constant. 
If we start from $p_1=q_1=0$, we have $p_1(t)=0$ and $q_1(t)= t A (\vert \bar q(0) \vert ^2+ \vert \bar p (0) \vert ^2)$. 
As a result we have $\varphi_H^1(S_0)=S_1$. Since $ \Vert H \Vert _{C^0} \leq \varepsilon $ this proves our claim. 
\end{proof} 

A first consequence is 
\begin{prop} \label{Prop-7.6}
Let $V$ be a smooth submanifold. Then it is $\gamma$-coisotropic if and only if it is coisotropic in the usual sense, i.e. for all points $x\in V$ we have  $(T_xV)^\omega \subset T_xV$.
\end{prop} 
\begin{proof} 
Assume $C$ is coisotropic. Locally $C$ can be identified to $$ \left\{ (x_1,...,x_n, p_1,...,p_k,0,...,0) \mid x_j, p_j \in {\mathbb R} \right \}$$ This contains the Lagrangian $p_1=....=p_n=0$ which we proved to be $\gamma$-coisotropic, hence $C$ is $\gamma$-coisotropic. 

Conversely assume $V$ is smooth but not coisotropic. Then locally we can embed $V$ in a codimension $2$ symplectic submanifold as follows : according to the Decomposition theorem (\cite{Viterbo-book} thm 2.14), we can write $T_xV=I\oplus S$ where $I$ is isotropic, $S$ symplectic, and there exists $K$, uniquely defined by the choice of $S$, such that $(K\oplus I)$ is symplectic and contained in $S^\omega$. Moreover $K\oplus I \oplus S=D(x)$ is symplectic, so choosing a continuously varying $S$, the same will hold for  $D$. If $V$ is not coisotropic then $D(x)\neq T_xM$. We thus have a symplectic distribution $D$ near $x_0$ such that $T_xV \subset D(x)$. Since being symplectic is an  open condition, $D(x)$ will be symplectic in a neighbourhood of $x_0$. Then, using the exponential map, we may find a symplectic manifold $W$ defined in a neighbourhood of $V$ such that $T_xW=D(x)$. Since we proved that a codimension $2$ symplectic submanifold is non-coisotropic, and being $\gamma$-coisotropic is locally hereditary, we conclude that $V$ is not $\gamma$-coisotropic. 
\end{proof}

\subsection{Links with other notions of coisotropy}
There are other definitions of coisotropic   subsets\footnote{In the litterature, synonyms of coisotropic are ``involutive'', ``locally-rigid'', etc.}. 
\begin{defn} [Poisson coisotropic]
We shall say that the set $V$ is {\bf Poisson coisotropic} if $\mathcal P_V=\left \{f \in C^\infty(M) \mid f=0\; \text{on}\; V\right \}$ is closed for the Poisson bracket. 
In other words if $f,g$ vanish on $V$ so does $\{f,g\}$. 
\end{defn}

Finally we define, following Bouligand (see \cite{Bouligand}),  for a subset  $V$ in a smooth manifold,  two cones :
\begin{defn} The {\bf paratingent cone} of a set $V$ at $x$ is 
$$C^+(x,V)= \left \{ \lim_n c_n(x_n-y_n) \mid x_n, y_n\in V, c_n \in {\mathbb R}\; \lim_nx_n=\lim_n y_n=x, \lim_n c_n=+\infty \right \}$$
The {\bf contingent cone} of a set $V$ at $x$ is 
$$C^-(x,V)= \left \{ \lim_n c_n(x_n-x)\mid x_n \in V, c_n \in {\mathbb R}, \; \lim_nx_n=x,  \lim_n c_n=+\infty \right \}$$
\end{defn} 
Clearly $C^-(x,V) \subset C^+(x,V)$. Note that $C^+(x,V)$ is invariant by $v \mapsto -v$, while it is not necessarily the case for $C^-(x,V)$. We then have the following definition, for which we refer to Kashiwara and Schapira

\begin{defn} [Cone-coisotropic, see \cite{K-S}, theorem 6.5.1 p. 271]
We shall say that $V$ is {\bf cone -coisotropic} if whenever a hyperplane $H$ is such that $C^+(x,V) \subset H$ then the symplectic orthogonal of $H$, $H^\omega$ is contained in $ C^-(x,V)$. 
\end{defn} 
Note that what we call here cone-coisotropic is called {\bf involutivity} in \cite{K-S}. 
\begin{figure}[H]
\centering
 \begin{overpic}[width=6cm]{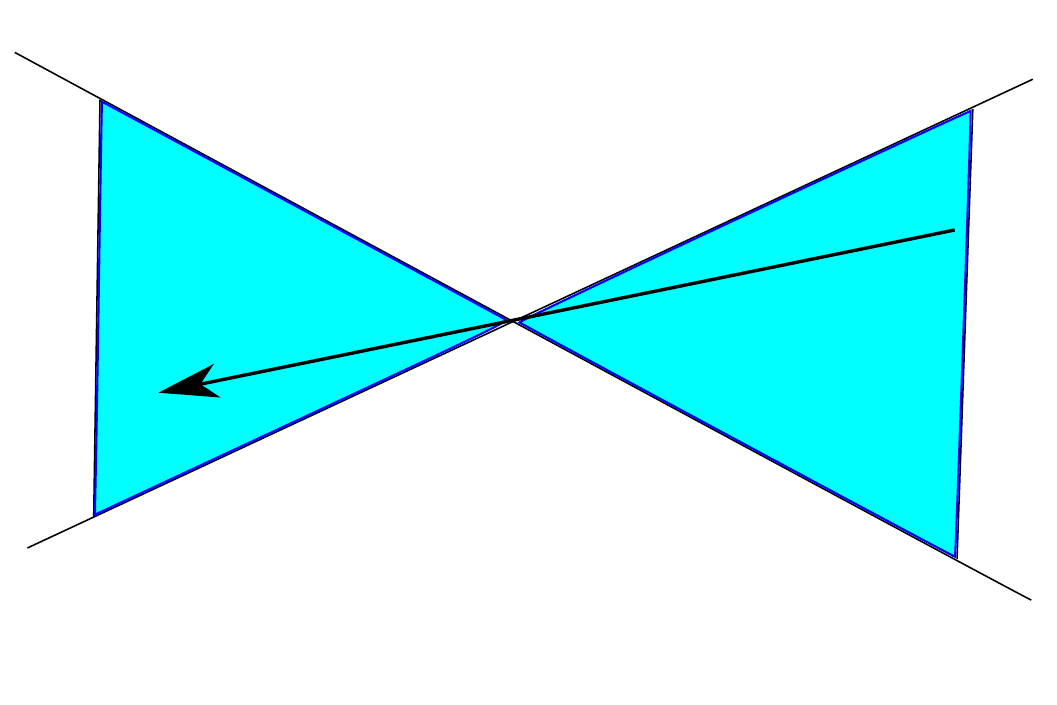}
 \put (75,35) {$\myBlue C$} 
  \put (12,45) {$\myBlue C$} 
   \put (10,35) {$ {H^\omega}$} 
 \end{overpic}
\begin{overpic}[width=6cm]{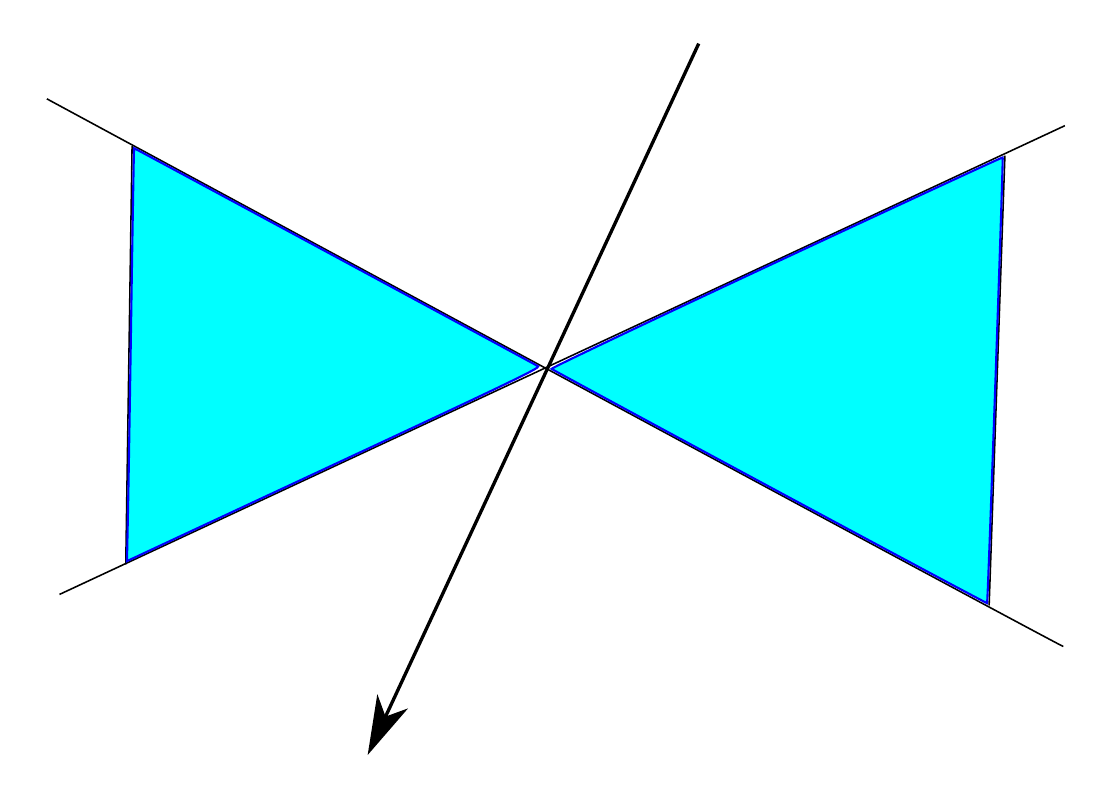}
 \put (75,35) {$\myBlue C$} 
  \put (12,35) {$\myBlue C$} 
   \put (27,10) {$ {H^\omega}$} 
 \end{overpic}
\caption{Two cones $C \subset H$. The one on the left is coisotropic, the one on the right is not.}
\label{fig-3n}
\end{figure}
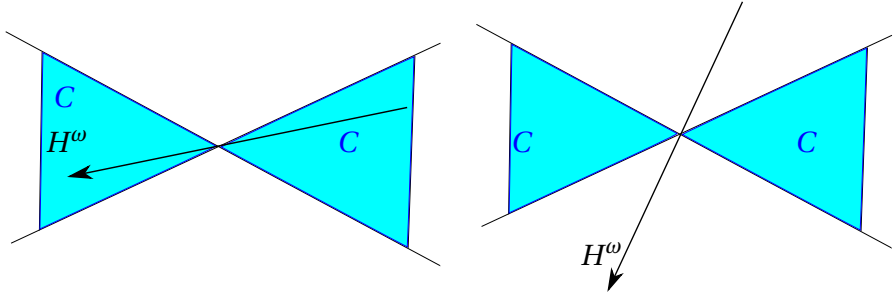
It is an elementary fact that in both cases a smooth submanifold is Poisson coisotropic or cone coisotropic if and only if it is coisotropic in the usual sense. 
We  proved in \cite{Guillermou-Viterbo}  the first part of  
\begin{prop} \label{Prop-various-coisotropic}
For a subset $V$ in $(M, \omega)$ we have the following implications
$$ \text{$\gamma$-coisotropic} \Longrightarrow \text{cone-coisotropic} \Longrightarrow \text{Poisson coisotropic} $$
\end{prop} 
\begin{proof} [Proof of the second implication  of the Proposition]

We argue by contradiction. Assume there are two functions $f,g$ vanishing on $V$ such that $\{f,g\}\neq 0$. Let $x_0$ be a point such that $\{f,g\}(x_0)\neq 0$.
Consider the set $S=f^{-1}(0)\cap g^{-1}(0)$ near $x_0$. Then since $df(x_0), dg(x_0)$ are non-zero and linearly independent $S$ is a codimension $2$ submanifold near $x_0$ and since $X_f, X_g$ are normal to $S$ and have non-zero symplectic product, we see that $(T_{x_0}S)^\omega = \langle X_f(x_0), X_g(x_0) \rangle$ is symplectic, so $S$ is symplectic and we may conclude  that $V$ is not cone-coisotropic, since in local coordinates, $S$ is given by 
$\{q_n=p_n=0\}$ and $T_{x_{0}}S \subset \{q_n=0\}$ but $ \frac{\partial}{\partial p_n} \notin T_{x_{0}}S$ and since in this case $T_{x_{0}}S=C^+(x_{0},S)=C^-(x_{0},S)$ this proves that $S$ is not cone-coisotropic. 
 \end{proof} 
 
 There are Poisson-coisotropic sets which are not cone-coisotropic. In \cite{Guillermou-Viterbo} we give an example of a cone-coisotropic set which is not $\gamma$-coisotropic. 
 \begin{example} 
 Let $V=\{(q,p) \mid p=0, q\geq 0\}$. It is Poisson coisotropic, as this is obvious on $V\cap \{q>0\}$ and the set of points where $V$ is Poisson coisotropic is closed. But $V$ is not cone-coisotropic at $(0,0)$ because $C^+(0,V)= {\mathbb R} \frac{\partial }{\partial q} $ while
$C^-(0,V)= {\mathbb R}_+ \frac{\partial }{\partial q}$ so the cone condition is violated. 
 
 \end{example} 
Our main result in this section is

\begin{thm}[Main Theorem] \label{Thm-7.12}
The $\gamma$-support satisfies the following:
\begin{enumerate} 
\item \label{Thm-7.12-1} For any $L \in \widehat {\LL}(M,\omega)$, $\gammasupp (L)$ is $\gamma$-coisotropic. 
\item \label{Thm-7.12-2}(Peano Lagrangian) \index{Peano Lagrangian} For any $r$ in $\{1,..., n\}$ we may find $L \in \widehat {\LL}_c(T^*N)$ such that $\supp(L)$  is equal to the $\gamma$-coisotropic set  $0_N\cup K_r$ where
$$K_r= \left \{ (q,p) \mid \vert q \vert \leq 1, \vert p \vert \leq 1, p_{r+1}=...=p_n=0\right \}$$
\item  \label{Thm-7.12-3}There exists $L\neq L'$ in $\widehat{\mathfrak L}(T^*N)$ such that $\gammasupp(L)=\gammasupp(L')$.
\end{enumerate} 
\end{thm} 

\begin{proof} [Proof of Theorem \ref{Thm-7.12} (\ref{Thm-7.12-1})]
Assume  $V=\supp(L)$ is non-$\gamma$-coisotropic at $z\in V$. Then there exists a sequence of Hamiltonian maps  $\varphi_k$ going to $\Id$ (again for $\gamma$), supported in $B(z, \varepsilon )$ and such that $\varphi_k(V)\cap B(z,\eta) = \emptyset$. 
Now on one hand
 $\gamma-\lim \varphi_{k} (L)= L$ since  $\gamma-\lim_{ k } \varphi_ k = \Id$. Thus $\gamma-\lim_k \varphi_{k} (L_k)=L$ has support $\supp(L)$. On the other hand
  $\varphi_k(L)$  has support $\varphi_k (\supp (L))$ and since  $\varphi_k (\supp (L))\cap B(z,\eta) = \emptyset$,  Proposition \ref{Prop-lim-support} implies that $\supp (L) \cap B(z,\eta)=\emptyset$. 
 We thus get a contradiction.
\end{proof} 
For  (\ref{Thm-7.12-2}),  we call such Lagrangians ``Peano Lagrangians'' because their supports, like the Peano curve, exhibit surprising dimensional properties while being associated to a Lagrangian.
Let us  consider the ``cube'' $K_r$ as above. 
Any embedding of $D^n$ in $N$ yields such a cube. Our goal is first to construct $L \in  \widehat {\LL}(T^*D^n)$ such that $\supp(L) \supset K_r$. 

\begin{thm} \label{Theorem-1}
There exists an element $L \in {\LL} (T^*N)$ with support satisfying $ \supp (L)=K_n\cup 0_N $. 
\end{thm} 
\begin{proof} 
Consider $L=\psi (0_N)$ where $\psi$ is supported in $K=K_n$. Let $A$ be a finite set in $\mathring K\setminus L$, $z$ a point in $\mathring K\setminus L$ not contained in $A$, $U \subset K\setminus (L\cup A)$ be the symplectic image of the product of two Lagrangian balls 
$\sigma : B^n(\varepsilon)\times B^n( \frac{1}{3})  \longrightarrow K\setminus {L\cup A}$. Notice that $B^n(\varepsilon)\times B^n( \frac{1}{3})$ is symplectically isotopic to  $B^{n}\left (\sqrt \frac{\varepsilon }{3}\right )\times B^{n}\left (\sqrt \frac{\varepsilon }{3}\right )$. We assume $z=\sigma (0)$ and set  $\Delta=\sigma (B^n ( \varepsilon )\times \{0\})$ be the image of the first Lagrangian ball. 
\begin{lem} \label{Lemma-7.14}
Let $L,K,A, U, \Delta$ as above. Then there is a constant $c$ such that for all $0< \varepsilon <c$, there is a symplectic isotopy $\rho$ such that 
\begin{enumerate} 
\item \label{Listprop1} $\rho$ is supported in $\mathring K\setminus A$
\item \label{Listprop2} $\rho(L)\cap U=\Delta$
\item \label{Listprop3} $\gamma (\rho) <\varepsilon $  hence $\gamma (\rho(L), L) <\varepsilon$
\item \label{Listprop4} there exists $\varphi$ supported in $U$ such that 
$\gamma (\varphi(\rho(L)),\rho(L)) > \frac{ \varepsilon }{5}$ 
\end{enumerate} 
\end{lem} 
\begin{proof} 
First of all we may apply an isotopy sending $L$ to the zero section. Then, since $A$ is discrete, we may push all its points by a symplectic isotopy in a neighbourhood of the boundary of $K$ and at the same time move  $z$ to $q_1=...=q_n=0, p_1= \frac{1}{2}, ...,p_n= \frac{1}{2}$, since $K\setminus 0_N$ is connected (except in dimension $2$, in which case we have two connected components, and $z$ is either $(0, \frac{1}{2})$ or $(0, - \frac{1}{2})$). As a result we may assume that $z$ is the center of a translate of  $B^n(\varepsilon)\times B^n( \frac{1}{3})$ contained in $K$ and avoiding both $A$ and $0_N$. We claim that we can move the zero section by a Hamiltonian isotopy $\tau_ \varepsilon $ generated by a Hamiltonian supported in $K$, with norm  $ \Vert H \Vert _{C^{0}} \leq \varepsilon $ and such that
$$\tau_ \varepsilon  (0_N)\cap \left (B^n(\varepsilon)\times B^n\left ( \frac{1}{3}\right )\right )=z+ \left (B^n(\varepsilon)\times \{0\}\right ) $$
Taking $\rho=\tau_ \varepsilon$, this proves (\ref{Listprop1}), (\ref{Listprop2}), (\ref{Listprop3}). 
Finally to prove (\ref{Listprop4}), it is enough to 
construct a Hamiltonian isotopy supported in $B^n(\varepsilon)\times B^n\left ( \frac{1}{3}\right )$ as is done in the proof of Proposition \ref{Prop-6.21}, see Figure \ref{fig-1n}, such that if $L\cap \left ( B^n(\varepsilon)\times B^n\left ( \frac{1}{3}\right )\right ) =B^n(\varepsilon)\times \{0\}$ we have $\gamma (\varphi(L), L) \geq \frac{ \varepsilon }{5}$. For this it is enough to replace $B^n(  \varepsilon )\times \{0\}$ by the graph of $df$ where $f$ is supported in $B^n(  \varepsilon )$ and $ \vert df \vert < \frac{1}{3} $ and $\osc f > \frac{\varepsilon 
}{5}$.  This induces a Hamiltonian isotopy such that $\gamma (\varphi(L),L) > \frac{ \varepsilon }{5}$.  
\end{proof} 
 \begin{figure}[H]
 \centering
 \begin{overpic}[width=6cm]{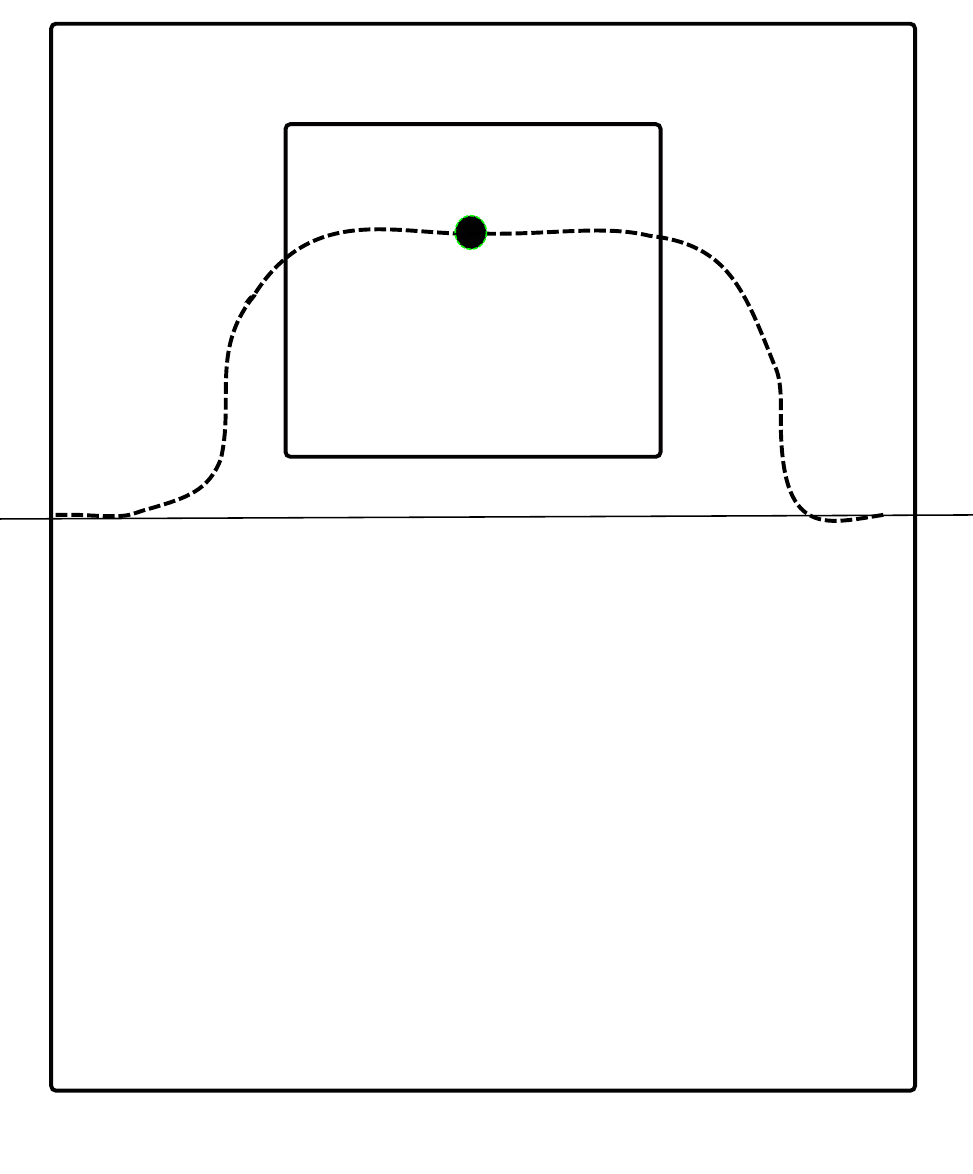}
 \put (80,35) {$K$} 
 \put (65,65){$\rho(L)$}
  \put (20,85) {$U$} 
   \put (40,85) {$z$} 
 \end{overpic}
 \caption{Illustration of Lemma \ref{Lemma-7.14}}\label{Figure-5}
\end{figure}

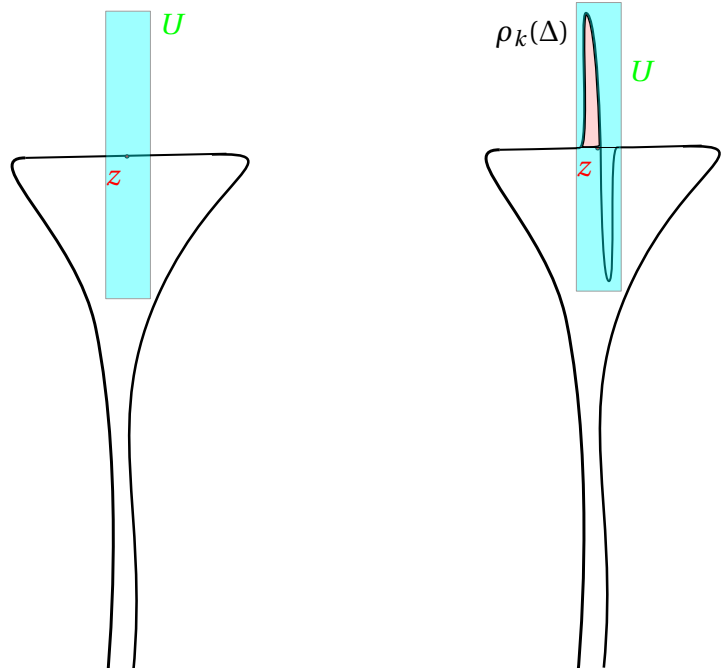
\begin{figure}[H]
\begin{overpic}[width=4.5cm]{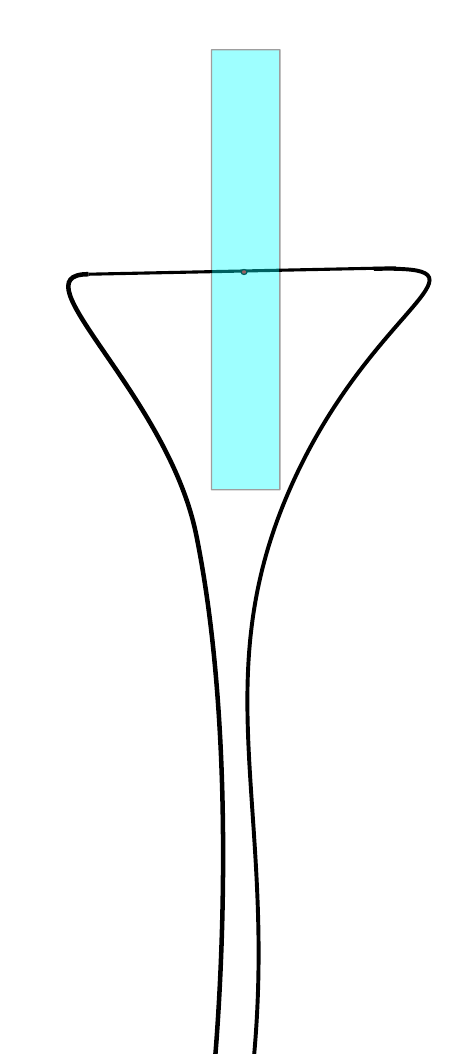}
\put (27,92) {$\myGreen U$} 
   \put (20,72) {$\myRed z$ } 
 \end{overpic}
 \hspace{2cm}
\begin{overpic}[width=4.2cm]{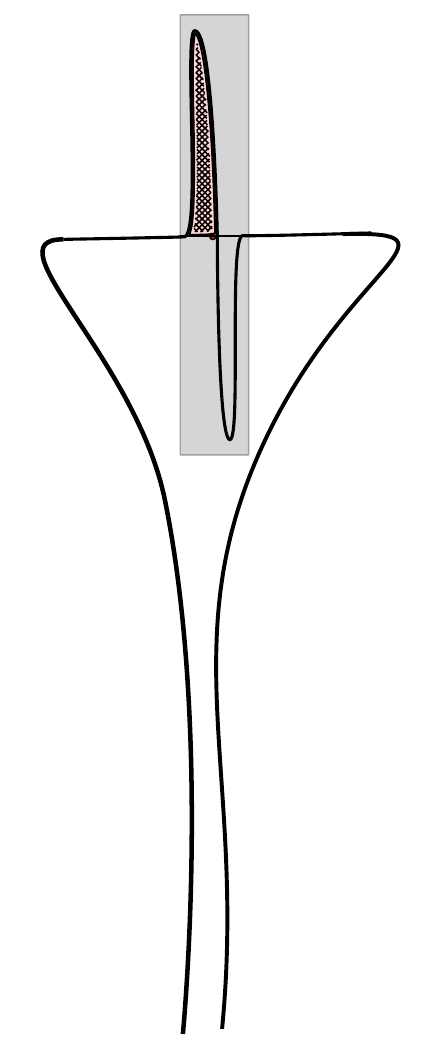}
 \put (17,73) {$\myRed z$} 
  \put (24,87) {$\myGreen U$} 
  \put (8, 93) {$\rho_{k}(\Delta)$}
 \end{overpic}
\caption{The Lagrangian  $\rho(0_N)\cap U$ and the path $\rho_{k}(\Delta)$. The hatched region has area less than $ \varepsilon_k $.}
\label{fig-4n}
\end{figure}

\begin{proof} [Proof of Theorem \ref{Theorem-1}]
Let now $(z_j)_{j\geq 1}$ be a dense subset in $\mathring K$. We apply the above Proposition inductively with $L=L_k, A_k=\{z_1,...,z_k\}, \varepsilon = \varepsilon _k$. We shall determine the sequence $ (\varepsilon _k)_{k\geq 1}$ later. 
We thus get  sequences $(\rho_k)_{k\geq 1}, (\varphi_k)_{k\geq 1}$ such that properties (\ref{Listprop1})-(\ref{Listprop4}) of Lemma \ref{Lemma-7.14} hold. In other words we have
\begin{enumerate} 
\item \label{Listprop1b} $\rho_k$ is supported in $\mathring K\setminus A_k$
\item \label{Listprop3b}  $\gamma (\rho_k(L_k), L_k) < \varepsilon_k$
\item \label{Listprop4b} there exists $\varphi_k$ supported in $U_k$ such that 
$\gamma (\varphi_k(\rho_k(L_k)),\rho_k(L_k)) > \frac{ \varepsilon_k }{5}$ 
\end{enumerate} 

We then set $L_{k+1}=\rho_k(L_k)$. According to Property (\ref{Listprop4}) this sequence will be $\gamma$-Cauchy if the series  $\sum_j \varepsilon _j$ converges. More precisely $\gamma (L_k, L_\infty) < \sum_{j=k}^{+\infty} \varepsilon_j$. Now assume the sequence $ (\varepsilon _j)_{j\geq 1}$ satisfies $\varepsilon _{j+1} < \frac{\varepsilon _j}{20}$ then denoting by $L_\infty$ the limit of $L_k$, we have 
$$\gamma(L_{k+1},L_\infty)< \varepsilon_{k+1} \sum_{j=0}^{+\infty} \frac{1}{20^j} = \frac{20}{19} \varepsilon_{k+1} < \frac{1}{19} \varepsilon_{k} $$ 
Then we claim
\begin{lem} 
Consider the sequence $L_k$ just defined and let $L_\infty$ be its $\gamma$-limit.  Then we have for all  $k$ large enough
$\gamma(\varphi_k(L_\infty), L_\infty) >  \frac{1 }{11} \varepsilon _k$.
\end{lem}   
\begin{proof}
Let us indeed use the triangle inequality to compute
\begin{displaymath} 
\gamma(\varphi_k(L_{k+1}),L_{k+1}) \leq \gamma(\varphi_k(L_{k+1}), \varphi_k(L_\infty)) +\gamma(\varphi_k(L_\infty), L_\infty)+ \gamma( L_\infty, L_{k+1}) 
\end{displaymath} 
so that 
\begin{displaymath} 
\gamma(\varphi_k(L_\infty), L_\infty)\geq \gamma(\varphi_k(L_{k+1}),L_{k+1}) -\gamma(\varphi_k(L_{k+1}), \varphi_k(L_\infty))-\gamma( L_\infty, L_{k+1}) 
 \end{displaymath}  
 and since $\varphi_k$ is an isometry for $\gamma$ we have 
\begin{gather*} 
\gamma(\varphi_k(L_\infty), L_\infty)\geq \gamma(\varphi_k(L_{k+1}),L_{k+1}) -2\gamma( L_\infty, L_{k+1}) \geq \\ \frac{\varepsilon _k}{5}- \frac{2}{19} \varepsilon_k> \frac{9}{95} \varepsilon_k > \frac{1}{11} \varepsilon_k 
 \end{gather*} 
\end{proof} 
\noindent {\it End of the proof of Theorem \ref{Theorem-1}}
We now consider the case where the sequence $(z_k)_{k\geq 1}$ is just the sequence of points with rational coordinates, and $U_k$ are neighbourhoods of $z_k$ making a basis of open sets for the topology of $K$. Then $\supp (L_\infty)$ is a closed set, meeting all the $\supp(\varphi_k)$ and since $\gamma (\varphi_k(L_\infty), L_\infty) >0$, it meets  all  the $U_k$ . This implies that $\supp(L_\infty)$ is dense in $K$, hence contains $K$.  
\end{proof} 
\noindent {\it End of the proof of Theorem \ref{Thm-7.12}.}
We start with (\ref{Thm-7.12-2}).  
To conclude the proof in  the case $1\leq r<n$,  we first use
\begin{prop}\label{Prop-7.17}
Let $L_{1} \in \widehat{\mathfrak L}(M_{1},\omega_{1})$ and $L_{2}\in {\mathfrak L}(M_{2},\omega_{2})$. Then we have
$$\gammasupp (L_{1}\times L_{2})=\gammasupp (L_{1})\times L_{2}$$
\end{prop} 
\begin{proof} 
We already proved the inclusion $\gammasupp (L_{1}\times L_{2})\subset \gammasupp (L_{1})\times \gammasupp (L_{2})=\gammasupp(L_{1})\times L_{2}$. The opposite inclusion follows from the following two claims
\begin{enumerate} 
\item \label{Prop-7.17-1} $\gammasupp (L_{1}\times L_{2})\cap \{z_{1}\}\times L_{2 } \neq \emptyset$ for all $z_{1}\in \gammasupp(L_{1})$
\item \label{Prop-7.17-2} $\gammasupp (L_{1}\times L_{2})$ is invariant by $\id\times \varphi$ for any $\varphi \in \DHam (M_{2},\omega_{2})$ such that $\varphi(L_{2})=L_{2}$
\end{enumerate} 
Let us first show that these two claims imply the Proposition. Indeed, for $(z_{1},z_{2})\in \gammasupp (L_{1})\times  L_{2}$, we have by (\ref{Prop-7.17-1}) that there
exists $z'_{2}$ such that $(z_{1},z'_{2})\in  \gammasupp (L_{1}\times L_{2})$. Now by (\ref{Prop-7.17-2}), choosing $\varphi$ such that $\varphi(z'_{2})=z_{2}$, which is possible since the Hamiltonian maps preserving $L_{2}$ act transitively on $L_{2}$, we get that $(z_{1},z_{2})\in  \gammasupp (L_{1}\times L_{2})$. 
For (\ref{Prop-7.17-1}), let $\varphi_{1}$ be supported in $B(z_{1}, \varepsilon )$ such that $\varphi(L_{1},L_{1})>0$. Then we have 
$\gamma (\varphi_{1}\times \Id)(L_{1}\times L_{2}), L_{1}\times L_{2})=\gamma(\varphi_{1}(L_{1})\times L_{2},L_{1}\times L_{2})=\gamma(\varphi(L_{1}),L_{1})>0$, where the last  equality follows from corollary 7.53 from \cite{Viterbo-book}. This implies that $\supp (\varphi\times \Id)$ intersects $\gammasupp(L_{1}\times L_{2})$, and letting $ \varepsilon $ go to $0$ we get $$\gammasupp(L_{1}\times L_{2})\cap \{z_1\}\times M_2\neq \emptyset$$ Since $\gammasupp (L_{1}\times L_{2})\subset \gammasupp (L_{1})\times L_{2}$, we have 
$$ \gammasupp(L_1\times L_2)\cap \{z_1\}\times M_2\subset \gammasupp (L_{1})\times L_{2}\cap \{z_1\}\times M_2 =\{z_1\}\times L_2$$
Therefore  $\gammasupp(L_1\times L_2)\cap \{z_1\}\times L_2\neq \emptyset$.

Claim  (\ref{Prop-7.17-1}) then follows from the fact that $\id\times \varphi$ preserves $L_1\times L_2$ hence preserves $\gammasupp (L_1\times L_2)$. 
\end{proof} 
We  may now conclude the proof of the case   $1\leq r<n$. Consider $L \subset T^*N$ where $N$ has dimension $r$ such that $\gammasupp(L_1)=0_N\cup [-1,1]^{2r}$. Let $M$ be $(n-r)$-dimensional and containing an embedding of $[-1,1]^{n-r}$ (this is of course always the case). Then 
$\gamma -\supp (L_1\times 0_M)=\gammasupp(L)\times 0_M=(0_L\cup [-1,1]^{2r})\times 0_M \supset 0_{L\times M} \cup K_r$.

As for (\ref{Thm-7.12-3}),  we can construct two sequences $(L_k)_{k\geq 1}$ and $(L'_k)_{k\geq 1}$ having $\gamma$-limits $L, L'$ as above so that $\gammasupp (L)=\gammasupp (L')$ and $L\neq L'$. Indeed, we may create the first ``tongue'' in different directions, so that $\gamma (L_1, L'_1) \geq \varepsilon_0$  and if we choose the $ \varepsilon_k$ such that  $\varepsilon_0 > \sum_{j=1}^{+\infty} \varepsilon _j$  we get $\gamma (L,L') > \varepsilon _0- \sum_{j=1}^{+\infty} \varepsilon_j$. Thus the support does not, in general,  determine a unique element in $\widehat{\mathfrak L}(M,\omega)$. 
\end{proof} 
\begin{rem} 
It is easy to see that $\supp(L_\infty)\subset 0_N\cup K$ so that $\supp(L_\infty)=0_N\cup K$. 
\end{rem}

\begin{rems} 
\begin{enumerate} 
\item Similarly we can find $L \in \widehat{\LL}(T^*N)$ such that $\gammasupp(L) = 0_N \cup T^*D^n(1)$.
\end{enumerate} 
\end{rems} 

\begin{Question}
Can a coisotropic submanifold ``containing no Lagrangian'', for example such that the coisotropic foliation has a dense leaf, be the $\gamma$-support of an element $L \in \widehat {\LL}(T^*N)$ ? 
One should be careful about the meaning of ``Lagrangian'' as this should be understood in a weak sense, since for a $C^0$ function $f$, the graph $\gra (df)$ will usually  not contain a smooth Lagrangian. 
\end{Question}

We consider a sequence of $\gamma$-coisotropic sets, $V_{k}$. If the $V_{k}$ are compact and contained in a bounded set, then, up to taking a subsequence,  they have a Hausdorff limit and  setting $V= \lim_{k }V_{k}$ we may ask whether $V$ is $\gamma$-coisotropic. The answer is obviously negative : take $V_{k}$ to be sphere of center $0$ and  radius $ \frac{1}{k}$. Then $V_{k}$ has for limit $\{0\}$ which is not coisotropic. However the following question is more sensible

\begin{Question}
Let $L_{k}$ be a sequence in $\widehat{\mathfrak L}(M, \omega)$ and $V_{k}=\gammasupp(L_{k})$. Assume $V=\lim_{k} V_{k}$ where the limit is a Hausdorff limit. 
Does $V$ contain a $\gamma$-coisotropic set ? 
\end{Question}
Note that Figure \ref{Fig-8} gives an example of a sequence such that the limit is not $\gamma$-coisotropic. Also the sequence cannot collapse to a set of Hausdorff dimension less than  $n$ because of the intersection properties of the $\gamma$-supports (Proposition \ref{Prop-6.7}). 
\begin{rems}  
\begin{enumerate} 
\item For the connection between Hausdorff convergence and $\gamma$-convergence (for Lagrangians) in the presence of Riemannian constrains we refer to \cite{Chasse}. 
\item Note that if all $\gamma$-supports contain an exact Lagrangian, then Proposition \ref{Prop-6.7} becomes obvious. Of course this can not be true in the usual sense, i.e. all $\gamma$-supports do not contain a smooth Lagrangians  : if $f$ is a $C^{0}$ function, $\gra (df)$ cannot contain a smooth Lagrangian.  On the other hand if this is not the case, then we get a really new class of subsets, invariant by symplectic isotopy,  having intersection properties. 
\end{enumerate} 
\end{rems} 

Note that one could hope that if $u_L$ is the graph selector associated to $L$ we have $\gammasupp(\gra(du_L)) \subset \gammasupp(L)$. This is however not the case as we see from
the following example
 \begin{example}\label{Example-selector}
 Let $u(x)= \vert x \vert $. It is then easy to see by using smooth approximations of $u$ that $\gammasupp(\gra(du)) $ is the union of $\{(x,-1) \mid x\leq 0\}$, $\{(x,1) \mid x\geq 0\}$ and $\{0\}\times [-1,1]$. However $u$ is the graph selector for the Lagrangian represented on Figure \ref{Figure-selector}
(see \cite{Chaperon-HJ, Viterbo-Ottolenghi, Viterbo-Montreal})
 \begin{figure}[H]\centering
\includegraphics[width=10cm]{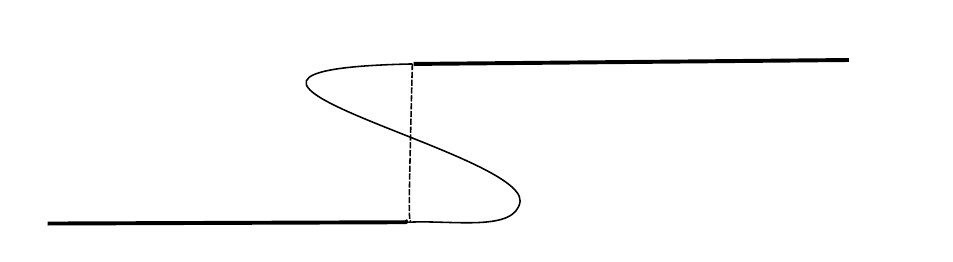}
\caption{ The Lagrangian from Exemple \ref{Example-selector}}\label{Figure-selector}
\end{figure}\label{fig-5}
 \end{example}

However it is not difficult to  prove 
 that $$\gammasupp(\gra(du_L)) \subset \Conv_p(\gammasupp(L))$$
where $\Conv_p$ is the $p$-convex hull, that is $$\Conv_p(X)= \left\{ (q,p)\in T^*N \mid \exists (p_1,...p_r)\in T_q^*N, t_j\geq 0, \sum t_j=1 , p= \sum_{j=1}^n t_j p_j\right\}$$
\begin{Question}
Does $\gammasupp (du_L)\cap \gammasupp (L)$ contain the extremal points of $ \Conv_p(\gammasupp(L))$ ?
\end{Question}

\section{Regular Lagrangians }\label{Section-Regular}
Let $(M,\omega)$ be a  symplectic manifold of dimension $2n$. 
Having established that $\gamma$-supports are always $\gamma$
-coisotropic, a natural question arises: when does a $\gamma$-support have minimal dimension, i.e., when is it an $n$-dimensional Lagrangian? Elements of $\widehat{\LL}(M,\omega)$ with this property deserve special attention.

\begin{defn} 
An element $L$ in $\widehat{{\LL}}(M,\omega)$ is said to be  {\bf regular} if $\supp (L)$ is a smooth $n$-dimensional submanifold. Such a manifold is then Lagrangian by Theorem \ref{Thm-7.12} and Proposition \ref{Prop-7.6}. It is {\bf topologically regular} if $\supp (L)$ is a topological $n$-dimensional submanifold.
\end{defn} 
\begin{Conjecture}
A regular Lagrangian coincides with an element of $\LL(M,\omega)$. More precisely if $L\in \widehat{{\LL}}(M,\omega)$ is such that $\gammasupp(L)=V$ is an smooth $n$-manifold, then $V$ is an exact Lagrangian and $L=V$.\end{Conjecture}
It follows from \cite{AGIV} (written after this paper was submitted) that 
\begin{thm}[ Regular Lagrangians in $T^*N$ (\cite{AGIV})]\label{Conj-regular}
If $L$ is regular in $T^*N$ and $\gammasupp(L)$ is exact, then $L=\gammasupp(L)$.
\end{thm}
\begin{rems} \begin{enumerate} 
\item
This means that the Lagrangian $L \in \widehat{\LL}(T^*N)$ coincides with $\gammasupp (L) \in {\LL}(T^*N) \subset \widehat{\LL}(T^*N)$. 
\item Assuming that $\gammasupp(L)$ is exact means the following:  $\gammasupp(L)$ is assumed to be a smooth submanifold, since it is $\gamma$-coisotropic, it must be Lagrangian, and we then assume $\lambda$ is exact on it. 
\item Before the proof in \cite{AGIV} was available, this was proved in special cases in the second ArXiv version of the present paper (as Theorem 8.6), using Theorem 6.3 from \cite{Viterbo-inverse-reduction}.  
 \end{enumerate} 
\end{rems} 
Following this theorem, it makes sense to set
\begin{defn} 
A {\bf topological Lagrangian} in $(M^{2n}, \omega)$  is a $C^0$-submanifold of dimension $n$,  $V$, such that  there exists $L\in \widehat{\LL}(M, \omega)$ with $V= \gammasupp(L)$.
\end{defn} 

\begin{rem} 
 There are  other possible definitions for topological Lagrangians. For example 
we could define a topological Lagrangian as an $n$-dimensional topological manifold that is $\gamma$-coisotropic. 
 \end{rem}

It follows from  \cite{BHS-C0}, where it is proved that $C^0$-convergence implies $\gamma$-convergence, that if $\varphi_k$ is a sequence of smooth symplectic maps converging $C^0$ to $\varphi$, then $\varphi =\gamma-\lim_k\varphi_k$, so $\varphi(L)=\gamma-\lim \varphi_{k}(L)$ and from Corollary \ref{Cor-3.15}
$\gammasupp (\varphi(L))=\varphi(L)$ so that $\varphi(L)$ is a topological Lagrangian. 
\begin{Question}
How far can we  extend this result ?  In particular is the property $\dim (\supp (L))=n$ enough  if we only assume $L$ is a topological manifold ? Or if we assume $\supp(L)$ contains no proper coisotropic ? 
Or if $\supp(L)$ is minimal for inclusion among the supports of elements in ${\LL}(T^*N)$ ? 
\end{Question}
It is probably useful at this point to remind the reader that smooth Lagrangians are indeed minimal among $\gamma$-coisotropic sets. 
  \begin{prop}[see \cite{Guillermou-Viterbo}, proposition 9.13]\label{Prop-12.16}
  
  Let $L= \varphi(L_0)$ where $L_0$ is smooth Lagrangian and $\varphi \in  \mathcal {H}_\gamma (M, \omega)= \widehat\DHam (M,\omega)\cap \Homeo(M)$. Then any closed proper subset of $L$ is not $\gamma$-coisotropic.
   \end{prop} 
 
 Here is another related question
\begin{Question}
For  $L_1, L_2$ such that $L_1$ is smooth, $L_2 \in \widehat{{\LL}}(T^*N)$ and $L_1 \subsetneq \gammasupp (L_2)$, what is the relation between $L_1$ and $L_2$ ? 
\end{Question}

We  may generalise this to the following
\begin{Question}
For which sets $V$ is the set of $L\in \widehat{\mathfrak L}(T^{*}N)$ such that $\gammasupp(L)=V$ is compact ?
\end{Question}

 \begin{rem} 
 Let $(L_{k})$ be a sequence in ${\mathfrak L}(T^{*}N)$  and assume it converges for the Hausdorff topology to a smooth submanifold $L$. Necessarily $L$ must have dimension at least $n$ since otherwise for $k$ large enough, $L_{k}$ would not intersect some vertical fibre. Assume it has dimension exactly $n$, we  then we claim that $L$ is Lagrangian. Indeed, $L_{k}\times S^{1} \subset T^{*}N \times T^{*}S^{1}$ converges in the Hausdorff topology  to $L\times 0_{S_{1}}$. By \cite{Laudenbach-Sikorav-IMRN}, if $L$ is not Lagrangian, $L\times 0_{S_{1}}$ is displaceable by a Hamiltonian isotopy and so will be $L_{k}\times 0_{S^{1}}$ for $k$ large enough. But this is impossible.  
 So $L$ is Lagrangian, and moreover it is exact by \cite{Membrez-Opshtein}. 
 Now remember the Conjecture from \cite{SHT}, proved in some special cases in \cite{Shelukhin-Zoll, shelukhin-sc-viterbo, Guillermou-Vichery, Viterbo-inverse-reduction}
 \begin{Conjecture}[Geometrically bounded implies spectrally bounded]\label{Conj-Viterbo-conjecture}
 Let $N$ be a closed Riemannian manifold. There exists a constant $C_N$ such that for all exact Lagrangian $L$ contained in $$DT^*L=\left \{(q,p) \in T^*N \mid \vert p \vert _g \leq 1 \right \}$$ we have
 $\gamma(L)\leq C_N$. 
 \end{Conjecture}
 \end{rem} 

This naturally extends to 
\begin{Conjecture}\label{Conjecture-8.18}
 There exists a constant $C_N$ such that for any 
 $L \in \widehat{\mathfrak L}(T^*N)$ with $\gammasupp(L) \subset DT^*N$, we have $$\gamma (L) \leq C_N$$
 \end{Conjecture}
 Note that this does not immediately follow from the conjecture in the smooth case:  even though $L$ is the limit of smooth $L_k$, we cannot claim that the $L_k$ are contained in a neighbourhood of $\gammasupp(L)$. 
 The above conjecture is proved for a certain class of manifolds in \cite{Viterbo-inverse-reduction}. 
 
 Now we claim 
 
 \begin{prop} \label{Prop-8.19}
 Let $\varphi \in \widehat \DHam_{c} (T^*T^{n})$ such that $\gammasupp(\Gamma(\varphi))=\Delta_{T^{*}T^{n}}$. Then $\varphi=\Id$.
 \end{prop} 
  \begin{proof} Consider the ${\mathbb Z}^{n}$ symplectic covering from $T^{*}(\Delta_{T^{*}T^{n}})$ to $T^*T^{n} \times \overline{T^{*}T^{n}}$. 
  Let $\Gamma(\varphi)$ be the graph of $\varphi$. For $\varphi \in \DHam_{c} (T^*T^{n})$, $\Gamma(\varphi)$ has a unique lifting $\widetilde \Gamma(\varphi)$ to  $T^{*}(\Delta_{T^{*}T^{n}})$ (note that $\Gamma(\varphi)$ and  $\widetilde \Gamma(\varphi)$ are diffeomorphic to $T^{*}T^{n}$).
   Now the projection of the covering yields a diffeomorphism between  $\widetilde \Gamma(\varphi)$ and  $ \Gamma(\varphi)$
  and\footnote{The following equality is easy to see in the context of Floer cohomology, since there is a one-to-one correspondence between holomorphic strips on the base and on the covering space, as strips are simply connected.}  $$\gamma (\widetilde \Gamma(\varphi_{1}), \widetilde \Gamma(\varphi_{2}))= \gamma (\Gamma(\varphi_{1}),\Gamma(\varphi_{2}))$$
 This implies  that an element $\varphi \in \widehat \DHam_{c} (T^*T^{n})$ defines a unique $\widetilde \Gamma(\varphi)$ in $\widehat {\mathcal L}(T^{*}(\Delta_{T^{*}T^{n}}))$. 
  Now  $\gammasupp(\widetilde \Gamma(\varphi))=\Delta_{T^{*}T^{n}}$ so, according to Theorem \ref{Conj-regular},  we have $\widetilde \Gamma(\varphi)=\Delta_{T^*T^n}$. But this means\footnote{The above equality implies that the map $\widehat \DHam (T^*T^{n}) \longrightarrow \widehat {\mathcal L}(T^{*}(\Delta_{T^{*}T^{n}}))$ is obtained by taking the completion of the  isometric embedding $\DHam_{c} (T^*T^{n}) \longrightarrow {\mathcal L}(T^{*}(\Delta_{T^{*}T^{n}}))$. It is therefore injective : two elements having the same graph are equal !}
  that $\varphi=\Id$.  
 
   \end{proof} 
  
  The same argument yields
    \begin{cor} 
  \begin{enumerate} 
  \item 
  Let $\varphi \in \widehat \DHam(T^{2n})$ such that $\gammasupp(\Gamma(\varphi))=\Delta_{T^{2n}}$. Then $\varphi=\Id$.
  \item The center of $ \widehat\DHam_{c} (T^{*}T^{n})$ is $\{\Id\}$. %
  \end{enumerate} 
  \end{cor} 
  \begin{proof} 
  The first result is obtained using the covering map 
  \begin{gather*} T^{*}\Delta_{T^{2n}} \longrightarrow T^{2n}\times \overline {T^{2n}}\\
  (x,y,X,Y) \mapsto (x+\frac{Y}{2}, y - \frac{X}{2}, x-\frac{Y}{2}, y +\frac{X}{2})
  \end{gather*}
   and repeating the above argument.  
    For the second, we denote by $\widetilde {\Gamma}(\varphi)$ the lift to $T^{*}\Delta_{T^{*}T^n}$ of $\Gamma (\varphi) \subset T^{*}T^n\times 
    \overline{T^{*}T^n}$. Because 
  $\gamma (\widetilde \Gamma(\varphi_{1}), \widetilde \Gamma(\varphi_{2}))= \gamma (\Gamma(\varphi_{1}),
    \Gamma(\varphi_{2}))$ an element 
    $\varphi \in \widehat{\DHam} (T^{*}T^n)$ also yields an element 
    $\widetilde {\Gamma}(\varphi) \in \widehat{\mathcal  L}(T^*\Delta_{T^{*}T^n})$.
  
 We see that if $\varphi \in  \widehat\DHam_c (T^{*}T^n)$ commutes with all  of $\widehat\DHam_c (T^{*}T^n)$, then for all $\psi \in \DHam (T^{*}T^n)$ we have $({\psi\times \psi}) (\widetilde\Gamma(\varphi))=\widetilde\Gamma(\varphi)$.
 But the only subset
 of the $\mathbb Z^n$ covering of $T^*T^n\times \overline{T^*T^n}$ which is invariant by all maps of the type 
 ${\psi\times \psi}$ is either $\Delta_{T^*T^n}$ or $T^*T^n\times \overline{T^*T^n}$ itself. The second case is impossible since $\varphi=\Id$ outside a compact set, so $\widetilde\Gamma(\varphi)$ coincides with the diagonal at infinity. So we must have $\widetilde\Gamma(\varphi)=\Delta_{T^*T^n}$ and by Proposition \ref{Prop-8.19} we have $\varphi=\Id$.
   
  \end{proof} 
  \begin{rem} 
  In general, we do not know if there exists $\varphi \in \widehat\DHam (M,\omega)$ such that $\gammasupp(\Gamma(\varphi))=M\times \overline M$. This justifies the next definition. 
  \end{rem} 
  \begin{defn} The set $\left \{\varphi \in \widehat\DHam (M,\omega) \mid \gammasupp (\Gamma(\varphi)) =\Delta_{M}\right \}$ is denoted $\mathcal N_{min}(M,\omega)$ is actually a normal closed subgroup in $\widehat\DHam (M,\omega)$. The same holds for 
  $\mathcal N_{max}(M,\omega)=\left \{\varphi \in \widehat\DHam (M,\omega) \mid \gammasupp (\Gamma(\varphi)) =M\times \overline M \right \}$
  \end{defn} 
  Note that the closeness follows from the fact that if $\varphi_{k} \longrightarrow \varphi$ and $\gammasupp (\varphi_{k})= \Delta_{M}$ then  $\gammasupp (\varphi)\subset \Delta_{M}$ but since $\Delta_{M}$ is Lagrangian, and no proper subset of a Lagrangian is $\gamma$-coisotropic (by \cite{Guillermou-Viterbo}, proposition 9.13, see Proposition \ref{Prop-12.16})  , we must have 
 $ \gammasupp(\varphi)=\Delta_{M}$. 
 
We have the following 
  \begin{Conjecture}\label{Conjecture-8.22}
  For any symplectic manifold $M$ the only element $\varphi \in \widehat\DHam_{c}(M,\omega)$ such that $\gammasupp(\Gamma(\varphi))=\Delta_{M}$ is
  $\varphi=\Id$, i.e. $\mathcal N_{min}(M,\omega )=\{ \Id\}$. Moreover $\mathcal N_{max}(M,\omega)=\emptyset$.  
  As a result the center of $\widehat\DHam_{c}(M,\omega)$ is $\{\Id\}$. 
  
  \end{Conjecture}
  Of course a stronger conjecture would be that $\widehat\DHam (M,\omega)$ is a simple group.

\section{On the effect of limits and reduction}
 As we know in the smooth case the property of being coisotropic is closed for the $C^1$ topology, and according to \cite{Hum-Lec-Sey2} even in the $C^0$ case (actually slightly less :  to apply \cite{Hum-Lec-Sey2} we need that the sequence $C_k$ of coisotropic to be given by $\varphi_k(C)$ where the sequence $(\varphi_k)_{k\geq 1}$ $C^0$ converges to $\varphi\in \Homeo (M, \omega)$). Also being coisotropic  is a property preserved by symplectic reduction. 

\begin{prop} 
There exists a sequence of Lagrangians such that its Hausdorff limit is non-$\gamma$-coisotropic. 
\end{prop} 
 \begin{proof} 
Indeed let us consider the sequence represented on Figure \ref{fig-4n} below. The Hausdorff limit of the curve is a half-line. But a half-line is not $\gamma$-coisotropic at its endpoint, since it is easy to move $p=0, q\geq 0$ by $H(q,p)$ such that $ \frac{\partial H}{\partial p}(q,0)>1$   outside the unit ball and $H$ is arbitrarily small (or check that it is  not cone-coisotropic). 
\end{proof}  
Note  that for the same reason the red curve does not belong to the support of the $\gamma$-limit of the $L_n$ (it is easy to see that the sequence  $(L_n)_{n\geq 1}$ is $\gamma$-Cauchy) since it is not $\gamma$-coisotropic.  
\begin{figure}[H]
\center\includegraphics[width=5cm]{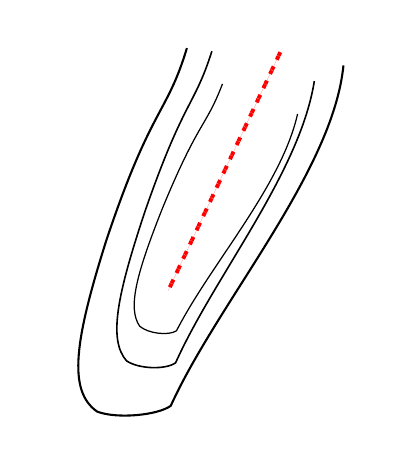}
 \put (-55,45) {$L_{n}$} 
\caption{ A Hausdorff limit (dotted line) of Lagrangians that is not $\gamma$-coisotropic.}
\label{Fig-8}
\end{figure}

\subsection{Reduction and \texorpdfstring{$\gamma$-support}{gammasupport}}
Now let us examine the effect of reduction on $\mathcal L(T^{*}N)$. For example consider $V$ a closed submanifold in $X$. The symplectic reduction $L_V=L\cap T_{V}^{*}X)/\simeq$
where $$(q,p_{1})\simeq (q,p_{2}) \Leftrightarrow p_{1}-p_{2}=0 \;\text{on}\; T_{q}V$$ of
$L \in \mathcal L(T^{*}X)$ is well-defined for $L$ transverse to $T_{V}^{*}X$  and such that the projection is an embedding. Then $L_{V}\in \mathcal L(T^{*}V)$, and  transversality, the set of $L$  for which $L_V$ is well-defined and embedded is open for the $C^{1}$-topology. We denote by $\mathcal L_{V}(T^{*}X)$ the set of such  Lagrangians. 

\begin{prop} Let  $V$ be a closed submanifold in $X$. The map $\bullet_{V}: L \mapsto L_{V}$ defined on $\mathcal L_{V}(T^{*}X)$ with image in $\mathcal L (T^{*}V)$ extends to a (well-defined) map 
$\widehat \bullet _{V}: \widehat{\mathcal L}_{V}(T^{*}X) \longrightarrow \widehat{\mathcal L}(T^{*}V) $ where  $\widehat{\mathcal L}_{V}(T^{*}X)$ is the closure of 
$\mathcal L_{V}(T^{*}X)$ in $\widehat{\mathcal L}(T^{*}X) $. 
 \end{prop} 
 \begin{proof} 
 Given a map $f$ defined on a subset $C$ of a metric space $(A,d_{A})$,   to $(B,d_{B})$ and $f$ is Lipschitz, it extends to a continuous map from the closure of $C$ to the completion of $B$, $(\widehat B, \widehat d_{B})$ since $f$ sends Cauchy sequences to Cauchy sequences. Now  we just need to apply this to $C= {\mathcal L}_{V}(T^{*}X)$, the subset of elements in ${\mathcal L}(T^{*}X)$ having a $V$-reduction and $B$ the space ${\mathcal L}(T^{*}V)$, both endowed with the metric $\gamma$. The proof is concluded by invoking proposition 5.2 in \cite{Viterbo-inverse-reduction} which claims
 $$\gamma ((L_{1})_{V}, (L_{2})_{V}) \leq \gamma(L_{1},L_{2})$$
 \end{proof} 
 
 We shall by abuse of language denote by $L_{V}$ the reduction of $L \in   {\mathcal L}(T^{*}X)$ to $\widehat{\mathcal L}(T^{*}V)$.  
Now let $C$ be a subset in $T^*(X\times Y)$ and $C_x\subset T^*Y$ the reduction of $C$ by $\{x\}\times Y$ that is $C\cap T_x^*X\times T^*Y/T_x^*X\simeq T^*Y$. 
 \begin{rem} 
 While the elements of $\mathcal L_V(T^*X)$ all yield by reduction, an embedded Lagrangian in $\mathcal L(T^*V)$, it is possible that an element of $\widehat{\mathcal L}_V(T^*V)$ reduces to a non-embedded Lagrangian in $T^*V$. This if for example the case for a $C^0$ limit of Lagrangians, but one could hope for more interesting examples.	
 \end{rem}

 \begin{prop} 
For $L \in \widehat \LL_{V} (T^*X)$,  we have   $$\gammasupp (L_V) \subset  (\gammasupp (L))_V$$   
 \end{prop} 
 \begin{proof}This follows from Proposition \ref{Prop-6.21} (\ref{Prop-6.21-f}).
\end{proof}

\begin{prop} Let $\mathcal V(X)$ be the set of compact submanifolds of $X$ endowed with the $C^r$topology ($r>1$). 
	The map $\mathcal V(X)\times \widehat {\mathcal L}_{c}(T^*X) \longrightarrow  \widehat {\mathcal L}_V(T^*X) $ given by $(V,L) \longrightarrow \gammasupp(L_V)$ is lower semicontinuous, where defined, in the following sense. For any sequences $(L_j)_{j\geq 1}, (V_j)_{j\geq 1}$ converging (for the $C^\infty$ topology) to $L,V$ we have (assuming all reductions are well-defined)
	$$\gammasupp (L_V)\subset \liminf_j \gammasupp \left (({L_j})_{V_j}\right )$$
\end{prop}
\begin{proof} 
We shall prove the map $(L,V) \mapsto L_V$ (when defined) is locally Lipschitz in both variables. First the reduction inequality (\cite[Prop. 5.1]{Viterbo-STAGGF})  implies
$\gamma(L_V,L'_V)\leq \gamma(L,L')$. Now if $d_{C^r}(V,V')$ is small,  there exists a $C^r$-small diffeomorphism $\psi$ isotopic to the identity such that
$\psi(V)=V'$ with $\psi$ $C^r$-small. Then $\psi$ is the time-one of a $C^{r-1}$-small vector field, $Z$ and the Hamiltonian $H_Z(q,p)=\langle p , X(q)\rangle$ has time one flow given by $\Psi(q,p)=(\psi(q), p\circ d\psi(q))$. If $\gammasupp(L)$ is contained in $\{(q,p) \mid \vert p \vert \leq R\}$ truncating the Hamiltonian, we get a  $C^{r-1}$-small Hamiltonian $H_{Z,R}$  such that its time-one flow $\Psi_R$ satisfies $(\Psi_R(L))_{V}=L_{V'}$ so that 
$$\gamma(L_V,L_{V'})=\gamma (\Psi_R(L)_V, L_V)\leq \gamma (\Psi_R(L), L)\leq \gamma (\Psi_R(L)_V) \leq  \Vert H_Z^R \Vert_{C^0} \leq R \varepsilon \dispdot $$
\end{proof}

\subsection{Reduction of \texorpdfstring{$\gamma$-coisotropic}{gamma-coisotropic} sets}
Another natural question is the following: given a closed set $V$, describe the set $C_V^0$ of $\gamma$-coisotropic points (resp. $Y_V^{0}$ of non-$\gamma$-coisotropic points). 
First the example of $V=[0,1]\subset \mathbb R^2$ where $C_V^0=]0,1[$ shows that $C_V^0$ is not necessarily closed. It is not necessarily open either, as can be seen  where $V$ is the union of a Lagrangian torus in $\mathbb R^2$ and a curve touching the torus, in which case $C_V^0$ is the torus, so it is closed.

We need the following characterization of $\gamma$-coisotropic points.

\begin{defn} 
We denote by $C_V^ \varepsilon $ the set of points $z$ such that there exists a positive function $\delta : {\mathbb R} \longrightarrow {\mathbb R}_+^* $ such that for any $\varphi\in \DHam_c(B(z, \varepsilon ))$ such that $\varphi(B(z,\eta))\cap V=\emptyset$ we have $\gamma(\varphi) > \delta(\eta)$. 
We set $Y_V^{\varepsilon, \eta} $ to be  the set of points $z$ in $V$ such that there exists a sequence $(\varphi_k)_{k\geq 1}$ in $\DHam_c(B(z, \varepsilon ))$with $\gamma-\lim_k\varphi_k=\Id$ and $\varphi_k(B(z,\eta))\cap V=\emptyset$. 
Finally we set $C_V^0=\bigcup_{ \varepsilon >0} C_V^\varepsilon $ to be the set of $\gamma$-coisotropic points and 
$Y_V^0=\bigcap_{ \varepsilon >0}\bigcup_{0< \eta< \varepsilon }Y_V^{\varepsilon, \eta} $ the set of non-$\gamma$-coisotropic points. 
\end{defn}

We may now state
\begin{prop} 
The set $C_V^0$ of $\gamma$-coisotropic points of the closed set $V$ is an $F_\sigma$, that is a countable union of closed sets. 
\end{prop} 
\begin{proof} We shall prove the equivalent statement that $Y_V^0$ the complement of $C_V^0$ is a $G_\delta$. 

Note that for $0< \alpha  $,  $Y^{ \varepsilon , \eta}$ is contained in $Y^{ \varepsilon + \alpha , \eta}$ and $Y^{ \varepsilon  , \eta-\alpha}$. 

 Note also that $Y_V^\varepsilon =\bigcup_{0<\eta< \varepsilon } Y_V^{ \varepsilon , \eta}$ is the complement of $C_V^\varepsilon $  and 
$Y_V^0$ is the complement of $C_V^0$ in $V$. 
Finally if  $z\in Y_V^{\varepsilon, \eta} $ and $ \vert z-z' \vert <\alpha < \eta$ we have $z'\in Y_V^{ \varepsilon +\alpha, \eta- \alpha}$. 
Then  for any $ \varepsilon < \varepsilon '$ and $z \in Y_V^\varepsilon $, a neighbourhood of $z$ is contained in $Y_V^{ \varepsilon '}$ since we must have $z\in Y_V^{ \varepsilon , \eta}$ for some $\eta>0$ and then $B(z,\alpha) \in Y_V^{ \varepsilon ', \eta-\alpha}\subset Y_V^{ \varepsilon '}$ provided $\alpha < \min \{ \varepsilon '- \varepsilon , \eta/2\}$. 
We then set $W^\varepsilon =\bigcup_{0<\eta< \varepsilon } B(Y_V^{ \varepsilon ,\eta}, \eta/2)$ and we have 
$$Y_V^\varepsilon =\bigcup_{0<\eta < \varepsilon } Y_V^{ \varepsilon , \eta} \subset W^\varepsilon =\bigcup_{0<\eta< \varepsilon } B(Y_V^{ \varepsilon ,\eta}, \eta/2) \subset Y_V^{2 \varepsilon }$$
Note that by definition $W^ \varepsilon $ is open, so $\bigcap_{ \varepsilon >0} W^\varepsilon $ is a $G_\delta$, since the intersection can be taken on the countable set $ \varepsilon \in \{ 1/n, n \in \mathbb N^*\}$.
But by the previous sequence of inclusions, we get
$$Y_V^0=\bigcap _{ \varepsilon >0} Y_V^\varepsilon \subset  \bigcap _{ \varepsilon >0} W^\varepsilon \subset 
\bigcap _{ \varepsilon >0} Y_V^{2 \varepsilon }=Y_V^0$$
We may thus conclude that $Y_V^0$ is a $G_\delta$. 
\end{proof} 

 \begin{rem}
 Even though being an $F_\sigma$  is not a very precise information, it excludes some sets as the set of coisotropic points. For example the complement of a countable dense subset of $V$ can not be the set of $\gamma$-coisotropic points. 
 \end{rem} 
 
Let us remind the reader that a {\bf nowhere dense} set is a set such that its closure has empty interior. A  {\bf meager} subset is a countable union of nowhere dense sets.  A set is {\bf nowhere meager} if its intersection with any open set is not meager. In a complete space, Baire's theorem implies that a meager set has empty interior and a nowhere meager set is everywhere dense. A set is {\bf comeager} if it is the complement of a meager set, or a countable intersection of sets of dense interior. 

 For $\gamma$-coisotropic sets, we have 
\begin{prop} \label{Prop-9.4}
Let $C$ be a $\gamma$-coisotropic set in $T^{*}(X\times Y)$. Then its symplectic reduction $C_x$ is $\gamma$-coisotropic or empty for $x$ in a nowhere-meager subset of $X$. 
\end{prop} 
\begin{rem} 
Notice that we cannot expect $C_x$ to be $\gamma$-coisotropic for all $x$. Indeed, take $X=\mathbb R$ and $C$ to be a smooth coisotropic manifold such that the coordinate $x$ has an isolated maximum with value $x_0$. Then $C_{x_0}$ is a point, hence not $\gamma$-coisotropic.  
\end{rem} 
We will need three lemmata. 
\begin{lem} \label{Lem-9.7}
Let $\varphi \in \DHam_c (B^{2n-2}(R))$ and $0<r'<r$. Then there is $\Phi \in \DHam_c (B^{2n-2}(R) \times ]-r,r[\times {\mathbb R})$ and a constant $C_{n}$ depending only on $n$ such that 
\begin{enumerate} 
\item $\gamma (\Phi)\leq C_{n} \gamma (\varphi)$
\item $\Phi(u,q_{n},p_{n})=(\varphi(u), q_{n}, \chi(u,q_{n},p_{n}))$ for $q_{n}\in [-r',r'[, p_{n}\in {\mathbb R} $
where $\chi$ is compact supported in $B^{2n-2}(R)\times ]-r,r[ \times \mathbb R)$ and equals $1$ near $B^{2n-2}(R)\times \{0\} \times \mathbb R$.
\end{enumerate} 
\end{lem} 
\begin{proof} 
See \cite[Lemma 9.3]{Viterbo-inverse-reduction}. 
\end{proof} 
\begin{lem} [see \cite{Seyfaddini2015, Hum-LeR-Sey, Ganor-Tanny}]\label{Lem-9.8}
Let $(\varphi_j)_{1\leq j\leq N}$ be in $\DHam_c(U_j)$ where the $U_j$ are symplectically separated domains. Then $$\gamma(\varphi_1\circ \ldots \circ  \varphi_N) \leq 2 \max_{1\leq j \leq N} \left \{ \gamma(\varphi_j)\right \}$$
\end{lem} 
\begin{proof} Remember that ``symplectically separated'' means  the (usual) distance between $U_{i}$ and $\psi(U_{j})$can be made arbitrarily large  by  some Hamiltonian map $\psi$. 
Applying  \cite[Thm.  44]{Hum-LeR-Sey} we get 
\begin{gather*} c_{+}(\varphi_1\circ \ldots \circ  \varphi_N) = \max_{1\leq j \leq N} \left \{ c_{+}(\varphi_j)\right \}
\\
c_{-}(\varphi_1\circ \ldots \circ  \varphi_N) = \min_{1\leq j \leq N} \left \{ c_{-}(\varphi_j)\right \}
\end{gather*} 
so taking the difference, and using  that $c_{-}(\varphi) \leq 0 \leq c_{+}(\varphi)$ we get 
$$\gamma(\varphi_1\circ \ldots \circ  \varphi_N)= \max_{1\leq j \leq N} \left \{ c_{+}(\varphi_j)\right \} - \min_{1\leq j \leq N} \left \{ c_{-}(\varphi_j)\right \} \leq 2 \max_{1\leq j \leq N} \left \{ \gamma(\varphi_j) \right \}$$
\end{proof} 
\begin{lem} \label{Lemma-product}
If $V=V' \times \mathbb R^2$ and $V$ is $\gamma$-coisotropic then $V'$ is $\gamma$-coisotropic 
\end{lem} 
\begin{proof} Assume $V'$ is not $\gamma$-coisotropic at $z$. 
Let   $\varphi'\in \DHam_c (B^{2n-2}(z; \varepsilon ))$  be such that $\varphi' (B(z;\eta))\cap V=\emptyset$. Then according to   \cite[Section 9]{Viterbo-inverse-reduction}, $\varphi'$ extends to $\varphi$ supported in $B(z;\varepsilon )\times B^2(0; \varepsilon )$ with $\gamma (\varphi)\leq C_n \gamma(\varphi')$ such that $\varphi(u, re^{i\theta})=(\varphi(u),re^{i\theta})$ for $r\leq \varepsilon /2$ (see \cite[Propostition 9.3]{Viterbo-inverse-reduction}). Then $\varphi(B^{2n}(z, \eta))\cap (V' \times \mathbb R^2)=\emptyset$.
We thus proved that if there is sequence $(\varphi'_k)_{k\geq 1}$ in $\DHam_c (B^{2n-2}(z; \varepsilon ))$, $\gamma$-converging to $\Id$ such that $\varphi'_k(B(z;\eta))\cap V=\emptyset$ then there is a sequence $(\varphi_k)_{k\geq 1}$ in $\DHam_c (B^{2n}((z,0); \varepsilon ))$ such that 
$\varphi_k(B^{2n}((z,0);\eta))\cap V=\emptyset$. In other words if $V'$ is non $\gamma$-coisotropic at $z$ then $V$ is non-$\gamma$-coisotropic at $(z,0)$. 
\end{proof} 
\begin{proof} [Proof of Proposition \ref{Prop-9.4}]
Note first that the statement is local in $x$, i.e; it is enough to prove it in a neighbourhood of $\{x_0\}\times {\mathbb R}^{2n-1}$. 

We shall prove that if the reduction $C_x$ is not $\gamma$-coisotropic for $x$ in a non-meager set, then $C$ is not $\gamma$-coisotropic. 

First, since the problem is local, and we can work by induction on the codimension of the reduction,  it is enough to deal with the case where the reduction is by a hyperplane $H_x$ with $x \in {\mathbb R} $. 
We can also assume $C\cap H_{x}$ is connected, since if $A\cup B$ is not $\gamma$-coisotropic at $x$, then both $A$ and $B$ are not $\gamma$-coisotropic at $x$ (according to Proposition \ref{Prop-7.5} (\ref{Prop-7.5-3})).

That $C$ is $\gamma$-coisotropic means: 

 for all $(z,x,p)\in C$ there exists $ \varepsilon_{(z,x,p)}>0$ and $\delta: \mathbb R \longrightarrow \mathbb R_+^*$ such that if $0<\eta < \varepsilon_{(z,x,p)}$, $\varphi\in \DHam_c (B((z,x,p); \varepsilon_{(z,x,p)})$, $\varphi(B((z,x,p); \eta)\cap C=\emptyset $ then  $\gamma (\varphi)> \delta(\eta)$. 
 
 \begin{enumerate}
 
 \item \label{Item-9.2-1} Our first task is to make $ \varepsilon_{(z,x,p)}$ bounded from below on a dense set. Indeed, set $C^\varepsilon$ be the set of points in $C$ where we can take $ \varepsilon_{(z,x,p)}> \varepsilon $. Then by assumption $C =\bigcup_ { \varepsilon>0} C^\varepsilon $, and of course the union is increasing and can be replaced by a countable union (just limiting ourselves to $ \varepsilon = \frac{1}{n}, n \in \mathbb N$). So by Baire's theorem, there exists $ \varepsilon_0 >0$ such that  $C^{\varepsilon_0}$ is not meager in $C$ (which is closed by assumption, hence complete), so there is an open set in $C$  contained in the closure  of $C^{\varepsilon_0}$. 
 \end{enumerate} 
  In the sequel we replace  $C$ by a closed set contained in this open set, so we may assume $ \varepsilon_{(z,x,p)}> \varepsilon_0$ on a dense subset. 
  \begin{enumerate}[resume]
 \item We then consider for each $(z,x,p_z,p_x)\in C$ (where $p=(p_z,p_x)\in \mathbb R^{n-1}\times \mathbb R$) the set $$C(z,x,p)=\left \{p_x-p'_x \mid (z,x,p_z,p'_x) \in C\right\} \subset \mathbb R \dispdot $$ 
 We claim that, by possibly restricting to a smaller open subset of $C$,  either  $C(z,x,p)$ contains an open set near $0$ for all $(z,x,p) \in C$ or there exists $0<\alpha< \delta < \varepsilon_0/2$ such that $C(z,x,p)\cap [\alpha, \delta]=\emptyset$, again for all $(z,x,p)\in C$.  
 
 Indeed, fix $(z_0,x_0,p_0)\in C$. Then the complement of  $C(z_0,x_0,p_0)$ is open and if it  contains a segment $[\alpha, \delta]$, then by Hausdorff semicontinuity of $x \mapsto C(x)=C\cap H_x$, the same holds for $(z,x,p)$ close to $(z_0,x_0,p_0)$. Otherwise $C(z_0,x_0,p_0)$ contains an interval of the form $[0,\delta]$. Since we can choose $(z_0,x_0,p_0)$ an arbitrary point in $C$, this means that for all $(z,x,p)$ in a neighbourhood of $(z_0,x_0,p_0)$ we have an interval $[0,\delta]$ in $C(z,x,p)$. Thus locally, $C$ is of the from $C'\times [x_1-\alpha,x_1+\alpha]\times [0,\delta]$, i.e. locally of the form $C'\times \mathbb R^2$. But then $C_x=C'$ and $C'$ must be $\gamma$-coisotropic by Lemma \ref{Lemma-product}. 
 \end{enumerate} 
To prove the Proposition, we shall argue by contradiction and  assume that for all $x$, the set $C_x=C\cap H_x/\simeq$ is nowhere-$\gamma$-coisotropic. This means that for all  $z\in C_x$ and $ \varepsilon>0 $ there exists $\eta_{z,x}>0$ and a sequence $\varphi^{z,x}_k$ supported in $B(z, \varepsilon )$ such that $\gamma(\varphi_k^{z,x})$ goes to zero and $\varphi^{z,x}_k(B(z,\eta_{z,x}))\cap C_x=\emptyset$.

We then choose $ \varepsilon = \varepsilon_0/2$ and by the same argument, possibly restricting $C$, we may assume $\eta_{z,x}> \eta_0$ on a dense set in $(z,x)$.

 Our proof  is then inspired by the proof of Proposition 2.1 from \cite{Viterbo-isoperimetric}. Let $(z,x_{0},p_{0})\in C$ with $z\in {\mathbb R}^{2n-2}$ and $ x_{0}, p_{0}\in {\mathbb R} $. The idea is that depending on whether $C_x$ is non-$\gamma$-coisotropic  we shall move $C$ in the $(z,p_z)$ direction or in the $p_x$ direction to prove that $C$ is non-$\gamma$-coisotropic. 

\begin{enumerate} [resume]
\item 
Assume for $x$ in a  dense set near $x_{0}$ there is a $\eta_0>0$ and a family $(\varphi_{x,k})_{k\in \mathbb N}$ of Hamiltonian maps    in $H_x/H_x^\omega \simeq {\mathbb R}^{2n-2}$ supported in 
$B^{2n-2}(z, \varepsilon )$ such that $$\varphi^{z,x}_{k}(C_x)\cap B^{2n-2}(z_{0}, \eta_0)=\emptyset$$ and $\gamma(\varphi^{z,x}_{k})< 1/k$. 

Pick an extension of $\varphi^{z,x}_{k}$, $\Phi^{z,x}_{k}$  
given by Lemma \ref{Lem-9.7}, such that $\gamma(\Phi^{z,x}_{k}) \leq C_{n}\gamma (\varphi^{z,x}_{k}) \leq C_n/k$. We claim that for 
$\alpha^{z,x}_{k}>0$ small enough, $$\Phi^{z,x}_{k}(C) \cap  \left (B^{2n-2}(z,\eta_0/2)\times ]x-\alpha^{z,x}_{k}, x+\alpha^{z,x}_{k}[\right ) = \emptyset$$ Indeed, if this was not 
the case, we would have 
$$\varphi^{z,x}_{k}(C_{x'}) \cap B^{2n-2}(z,\eta_0/2))\neq \emptyset$$
for some $x'$ such that $\vert x-x' \vert  \leq \alpha_x$. But since $x' \mapsto C_{x'}$ is upper semi-continuous for the Hausdorff distance, i.e. $C_{x_0}\supset \lim_{x\to x_0} C_x$, this is impossible. 
So assume the set of $x$ such that $C_x$ is not $\gamma$-coisotropic is dense in $[-R,R]$. 
Then for each $x\in [-R,R]$ there is a $\Phi^{z,x}_{k}$ such that $\Phi_x(B^{2n-2}(z,\eta_0/2)\times ]x-\alpha_x, x+ \alpha_x[\times ]-R,R[)) \cap C \cap H_y = \emptyset$ for $y\in I_x$ where $I_x$ is an open neighbourhood of $x$. Moreover we have $\gamma (\Phi^{z,x}_{k})< C_n/k$. 

\item By compactness we may cover $[-R,R]$ by a finite number of such intervals, $I_1,..., I_p$ where $I_l=]x_l-\alpha_l, x_l+\alpha_l[$ and we may assume $I_l\cap I_m=\emptyset$ for $ \vert l-m \vert > 1$  and $I_l\cap I_{l+1}$ is as small as we wish. Let us set  $\Psi=\Phi_{x_1}\circ ....\circ \Phi_{x_p}$. Obviously if $y\in I_l\setminus I_m$ for all $l\neq m$ then $$\Phi \left (B^{2n-2}(z_0,\eta/2)\times \{y\}\times [-R,R]\right ) \cap C \cap H_y=\emptyset \dispdot $$ where $\eta=\inf \{ \eta_{x_l} \mid 1\leq l \leq p\}$. 
For $y\in I_l\cap I_{l+1}$ set $\psi_j$ to be the flow of the Hamiltonian 
$\chi_l (y)$ where $\chi_{l}'(y) = \delta-\alpha$ for $y\in I_l\cap I_{l+1}$ where we assume $[\alpha, \delta]$ is the interval given in
and $
\chi_{l}(y)=0$ for $y\in I_k$ $k\neq j,j+1$.  Then $ \vert \chi_l 
\vert_{C^0} \leq 2R  \vert I_l \cap I_{l+1} \vert $ (where for a subset 
$A$ of $\mathbb R$, we define  $\vert A\vert$ as the length of $A$).  
It is then easy to check that $\Phi=\Psi \circ \psi_1\circ .. \circ \psi_k$ satisfies $\Phi(B^{2n}(R)) \cap C=\emptyset$. Using Lemma \ref{Lem-9.7} and  \ref{Lem-9.8} we see that $$\gamma (\Phi) \leq 2 C_n/k+ 4R \sup_{l\in [1,p]}  \vert  I_l\cap I_{l+1} \vert \}$$

Note that $ \vert I_l\cap I_{l+1} \vert $ can be made arbitrarily small, so we may assume their lengths are all less than $1/k$. Then 
\begin{itemize}
\item $\gamma (\Phi_k) \leq \frac{2C_n+4R}{k}$ 
\item  $\Phi_k(B^{2n}((z,x_0,p_0);\eta_0/2)) \cap C= \emptyset$
\end{itemize}
 Thus $C$ is not $\gamma$-coisotropic at $(z_0,x_0,p_0)$. 
\item Since balls are Liouville domains, we may assume by \cite[Lemma 9.3]{Viterbo-inverse-reduction} that $\gamma (\Phi) \leq C_{n} \gamma (\varphi)$, so if we had a sequence $(\varphi_{k})_{k\geq 1}$ such that $\gamma (\varphi_{k})$ converges to $0$ and such that $$\varphi_k(C_x)\cap B^{2n-2}(z, \eta)=\emptyset$$

Thus if the reduction of $C$ was not $\gamma$-coisotropic, we would have a sequence $(\varphi_k)_{k\geq 1}$ with $\gamma(\varphi_{k})$ going to zero. But then
$\Phi_{k}$ is supported in $B(z, \varepsilon )\times B^{2}( \varepsilon )$ and $\Phi_{k}$ moves $C$ away from $B(z, \eta )\times B^{2}( \eta )$
and $\gamma(\Phi_{k}) \leq C_{d}\gamma (\varphi_{k})$ which converges to $0$ and $C$ is not $\gamma$-coisotropic at $z$. 
\end{enumerate} 
 \begin{figure}[ht]
 \begin{overpic}[width=8cm]{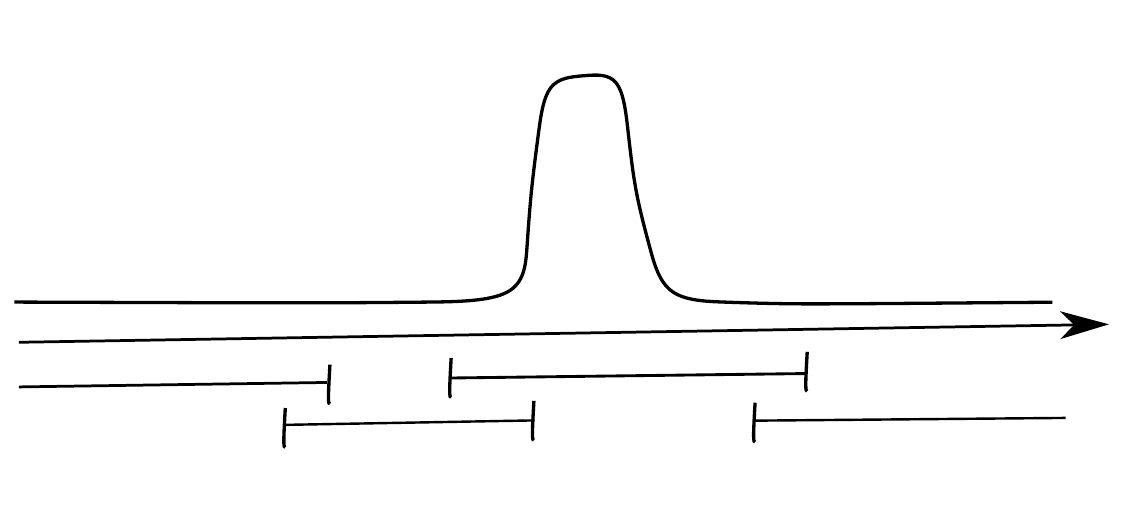}
 \put (10, 6) {$U_{l-1}$} 
 \put (30,3){$U_{l}$}
  \put (50,6) {$U_{l+1}$} 
  \put(40,32){$\chi_{l}$}
   \put (70,3) {$U_{l+2}$} 
 \end{overpic}
 \caption{Illustration of  $\chi_{l}$}\label{Figure-9}
\end{figure}
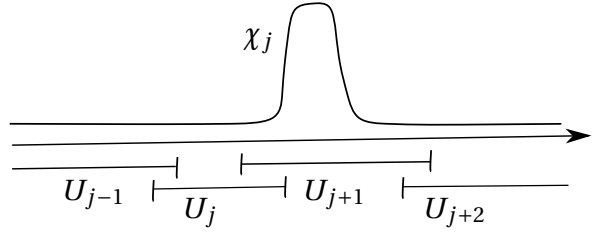

\end{proof} 

Note that $x \mapsto L_x$ is continuous for the $\gamma$-topology. However $L \mapsto \gammasupp(L)$ is not continuous for the Hausdorff topology on closed sets. It may be continuous for the $\gamma$-topology on sets that we shall define later but we have no idea on how to prove such a statement. 
Let us mention the following application
\begin{prop} \label{Prop-9.7}Let ${\mathcal H}^k$ be the $k$-dimensional Hausdorff measure. 
Let $V$ be a set in $(M^{2n}, d\lambda)$ such that ${\mathcal H}^{n}(V)=0$. Then $V$ is nowhere $\gamma$-coisotropic. 
\end{prop} 
\begin{proof} 
We argue by induction: assume this holds in a symplectic  space of dimension less than $2n-2$ and let us prove it in dimension $2n$. 

As this is a local result, we may assume we are in a symplectic vector space with symplectic basis $(e_{1},f_{1},...,e_{n},f_{n})$. We set for $x\in {\mathbb R}$, $H_{x}=xe_{1}+ {\mathbb R}^{2n-1}$ where ${\mathbb R}^{2n-1}=\langle f_{1},e_{2},f_{2},...,e_{n},f_{n}\rangle$. Then  the reduction of the set  $V$ at $H_{x}$ is given by 
$$(V \cap H_{x})/\langle f_{1}\rangle = (V/\langle f_{1}\rangle) \cap (H_{x}/\langle f_{1}\rangle ) \dispdot  $$
We denote in the sequel $K_{x}$ and $W$ for the hyperplane $H_{x}/\langle f_{1}\rangle$ and the subset  $V/\langle f_{1}\rangle$ in $ {\mathbb R}^{2n-1}$. 
Moreover we set $W(x)=K_{x}\cap W$ that is the reduction of $V$ at $x$. 
We assume ${\mathcal H}^n(V)=0$ which implies that 
${\mathcal H}^n(W)=0$.  

Now according to Federer (\cite[Thm. 2.10.25]{Federer} and  \cite[Prop. 6.2.3]{Mattila2})
$$\int_{\mathbb R} \mathcal H^{n-1}(W(x)) dx  \leq  C_n \mathcal H^{n}(W)\dispdot  $$

This implies that  $\mathcal{H}^{n-1}(W(x))=0$ for almost all $x$.  By induction $W(x)$ is nowhere $\gamma$-coisotropic (at least in a neighbourhood of $z_{0}$) and 
therefore there is a dense set of $x$ such that the $W(x)$ are  nowhere $\gamma$-coisotropic. 
According to Proposition \ref{Prop-9.4}, this implies that $V$ is nowhere $\gamma$-coisotropic.

\end{proof} 
\begin{rem} 
We prove in \cite{MCA-VH-CV} that for $V=\gammasupp(L)$ in $T^*N$ then the projection $V \longrightarrow N$ induces an injection in cohomology,  so of course  its dimension is at least $n$. 
\end{rem}

\section{Singularities of Hamiltonians in \texorpdfstring{$\DHam (M,\omega)$}{DHam(M,w)} }\label{Section-10}

Our goal here is to deal with Hamiltonians with singularities and understand these in the framework of the Humili{\`e}re completion. This was already explained in \cite{Humiliere-completion}, but here we show how $\gamma$-coisotropic sets enter the picture. 
 
We now formalize what it means for an element of $\widehat\DHam(M, \omega)$ to ``equal the identity'' on a subset. Since elements of the completion are equivalence classes of Cauchy sequences, we need three related notions: 
\begin{itemize}
\item  when a single element equals the identity on  $U$
\item  when a sequence converges to the identity on $U$, 
\item how to define the quotient space.
\end{itemize}
 \begin{defn}\label{Def-12.1} Let $U$ be  open and $W$ be closed  in $(M,\omega)$. 
 \begin{enumerate} 
 \item \label{Def-12.1-1} An element $\varphi\in \widehat\DHam(M, \omega)$ {\bf equals $\Id$ on $U$} if 
 $$\gammasupp(\Gamma(\varphi)) \cap (U\times M) =\gammasupp(\Gamma(\varphi)) \cap (M\times U)\subset  \Delta_M\dispdot  $$ We denote the set of such $\varphi$ by $\widehat\DHam(M,U, \omega)$. We then set $$\widehat\DHam(M,W, \omega)=\bigcup_{W\subset U}\widehat\DHam(M,U, \omega)\dispdot  $$
  \item \label{Def-12.1-2} We say that a sequence $(\varphi_k)_{k\geq 1}$ in $\widehat\DHam(M, \omega)$ {\bf  $\gamma$-converges to $\Id$ on $U$} if there exists a sequence 
 $(\psi_k)_{k\geq 1}$, of elements of $\widehat\DHam(M,U, \omega)$  such that  
 $\gamma(\varphi_k, \psi_k) \longrightarrow 0$. In other words $$\lim_k \gamma (\varphi_k, \widehat\DHam(M,U, \omega))= 0\dispdot  $$  
 \item \label{Def-12.1-3}We define $\widehat\DHam_{M}(U, \omega)$ as the set of equivalence classes of sequences $(\varphi_k)_{k\geq 1}$ such that for  any sequence $(l_k)_{k\geq 1}$ with $l_k\geq k$, the sequence $(\varphi_k \varphi_{l_k}^{-1})_{k\geq 1} $ $\gamma$-converges to $\Id $ on $W$. Two sequences $(\varphi_k)_{k\geq 1}$ and $(\psi_k)_{k\geq 1}$  are equivalent if  $(\varphi_k\psi_k^{-1})_{k\geq 1}$ $\gamma$-converges to $\Id$ on $U$. We define
$$ \widehat\DHam_{M}(W, \omega)=\bigcap_{W\subset U} \widehat\DHam_{M}(U, \omega)\dispdot  $$
 \end{enumerate} 
 \end{defn} 
\begin{rems} 
\begin{enumerate} 
\item Since in general $\omega$ is fixed we shall omit it from the notation. 
\item In (\ref{Def-12.1-2}), we do not assume that the sequence $\gamma$-converges ! 
\item If $\varphi$ is smooth, saying that $\varphi$ {\bf equals $\Id$ on $U$} or $W$ in the sense of Definition \ref{Def-12.1} is equivalent to its usual meaning since 
$\gammasupp(\Gamma(\varphi))=\Gamma(\varphi)$.
\item If $\gammasupp(\Gamma(\varphi)) \cap W\times \overline M \subset \Delta_M$ and $W$ is the closure of its interior, we must have  $$\supp(\Gamma(\varphi)) \cap W\times \overline M = \Delta_W\dispdot  $$
Indeed, if $L$ is  smooth Lagrangian, no proper subset of $L$ is $\gamma$-coisotropic (see Proposition \ref{Prop-12.16}).   
\item There is of course an embedding from $\DHam_{c}(U)$ to $\widehat \DHam_{M} (U)$ since an element in $\DHam_{c}(U)$ extends automatically to
$\DHam_{c}(M)$ and this extension equals $\Id$ on $M\setminus U$ if and only if we started from $\Id_{U}$. 
\item There is a bi-invariant metric still denoted by $\gamma$ on $\widehat\DHam(M,U)$
$$\gamma(\varphi, \Id)=\lim_{k}\gamma (\varphi_{k}, \DHam(M,U))$$ where $\varphi$ is represented by the sequence $(\varphi_{k})_{k\geq 1}$. The existence of the limit follows from the inequality 
\begin{gather*} \left \vert \gamma (\varphi_{k}, \widehat\DHam(M,W, \omega))- \gamma (\varphi_{l}, \widehat\DHam(M,U)) \right \vert \leq \\ \gamma (\varphi_{k}\varphi_{l}^{-1}, \widehat\DHam(M,U))
\end{gather*}  so $\gamma (\varphi_{k}, \widehat\DHam(M,U))$ is a Cauchy sequence in $ {\mathbb R} $.
There is then a
topology  on $\widehat \DHam_{M} (W)$ given by $\varphi_{k} \overset{\gamma} \longrightarrow \varphi$ 
if and only if this holds in   $\widehat\DHam_{M}(K, \omega)$ for some neighbourhood $U$ of $W$. 
  It is easy to check that this limit does not depend on the choice of the sequence and that the embedding of $\DHam_{c}(U)$ in $\widehat \DHam_{M} (U)$ yields an isometric embedding, hence an isometric embedding $\widehat \DHam(U) \longrightarrow \widehat \DHam_{M} (U)$. 
  \begin{Question}
  Is the map $\widehat \DHam(U) \longrightarrow \widehat \DHam_{M} (U)$ bijective ? 
  \end{Question}
\end{enumerate} 
\end{rems} 
 \begin{prop} \label{Prop-14.3}
We have the following properties
 \begin{enumerate} 
 \item \label{Prop-14.3-1} The set $\widehat\DHam(M,U)$  is a closed subgroup  in $\widehat\DHam(M, \omega)$ 
 \item \label{Prop-14.3-2}
 We have an exact sequence 
 $$1 \longrightarrow \widehat\DHam(M,U) \longrightarrow  \widehat\DHam(M) \longrightarrow \widehat \DHam_{M} (U) \longrightarrow 1 \dispdot  $$
 \item \label{Prop-14.3-3} 
 If an element $\varphi\in \widehat\DHam(M, \omega)$ equals  $\Id$ on $U$ and $V$ then it is equal to $\Id$ on $V\cup W$. 
If  a sequence $(\varphi_k)_{k\geq 1}$ $\gamma$-converges to $\Id$ on $U$ and $V$ then it $\gamma$-converges to $\Id$ on $U\cup V$. In other words
$$\widehat\DHam(M,U)\cap \widehat\DHam(M,V)=\widehat\DHam(M,U\cup V)\dispdot  $$
 \item \label{Prop-14.3-4} If $\psi \in \DHam(M)$ sends $U$ to $U'$, then $\varphi \mapsto \psi\varphi\psi^{-1}$ sends $\widehat\DHam(M,U)$ to $\widehat\DHam(M,U')$.
 \item \label{Prop-14.3-5}
  If Conjecture \ref{Conjecture-8.22} holds for $M$ (in particular for $M=T^{*}T^{n}$ or $T^{2n}$), an element $\varphi\in \widehat\DHam(M, \omega)$ equals $\Id$ on $M$, then $\varphi=\Id$ in $\widehat\DHam(M,\omega)$.
 \item \label{Prop-14.3-6}  If Conjecture \ref{Conjecture-8.22} holds for $M$ and if the sequence $(\varphi_k)_{k\geq 1}$ $\gamma$-converges to $\Id$ on $M$, then it  $\gamma$-converges to $\Id$ in the usual sense. 
 \end{enumerate}  \end{prop} 
 \begin{proof}
 \begin{enumerate} 
 \item 
For the first statement,  indeed, if $(\varphi_k)_{k\geq 1}$ is a $\gamma$-converging sequence such that $\gammasupp(\varphi_k) \cap U \times \overline M \subset \Delta_M$, then, since  according to Proposition \ref{Prop-lim-support}
 we have
 $ \gammasupp (\varphi) \subset \lim_k \gammasupp(\varphi_k)$ we deduce $$\gammasupp (\varphi) \cap U \times \overline M \subset \Delta_M\dispdot  $$
 
\item Let $(\varphi_{k})_{k\geq 0}$ be a $\gamma$-Cauchy sequence in $\DHam(M)$. Then  $\varphi_k \varphi_{l_k}^{-1} $ converges to $\Id$ on $M$, hence on $U$ which defines the map $$\widehat\DHam(M,\omega) \longrightarrow \widehat \DHam_{M} (U)\dispdot  $$
 If such a sequence is in the kernel, this means that 
$(\varphi_{k})_{k\geq 0}$ converges to $\Id$ on $W$, i.e. there is a sequence $(\psi_{k})_{k\geq 1}$ such that $\psi_{k}=\Id$ on $U$ and $\gamma (\varphi_{k}, \psi_{k})$ goes to $0$ as $k$ goes to $+\infty$. But then $$\gamma-\lim_{k} \psi_{k}= \gamma-\lim_{k}\varphi_{k}=\varphi$$ and since 
$$\gammasupp(\Gamma(\psi_{k}))\cap M\times U=\gammasupp(\Gamma(\psi_{k}))\cap W\times U\subset\Delta_{M}$$ and 
$\gammasupp(\Gamma(\psi)) \subset \lim_{k}\gammasupp(\Gamma(\psi_{k}))$ we have $\varphi=\Id$ on $U$, i.e. $\varphi \in  \widehat\DHam(M,U)$. 
\item  Indeed if $$\gammasupp (\Gamma (\varphi)) \cap M\times U=\gammasupp (\Gamma (\varphi)) \cap U \times M\subset \Delta_{M}$$ and  $$\gammasupp (\Gamma (\varphi)) \cap M\times V=\gammasupp (\Gamma (\varphi)) \cap  V\times M\subset \Delta_{M}$$ then 
$$\gammasupp (\Gamma (\varphi)) \cap M\times (U\cup V)\subset\Delta_{M}\dispdot  $$
\item Is obvious from the definition
\item If $\gammasupp(\Gamma(\varphi))=\Delta_{M}$, we have $\varphi=\Id$. 
\item Same as above
\end{enumerate} 

 \end{proof}

\begin{rem} 
We can also work in the general case, i.e. without assuming $M$ satisfies Conjecture \ref{Conjecture-8.22}, by replacing $\widehat {\DHam}_{c}(M,\omega)$ by  the quotient $\widehat {\DHam}_{c}(M,\omega)/ \mathcal N_{\Delta}(M,\omega)$. 
\end{rem} 
We now set 
 
 \begin{defn} 
 Let $V$ be a locally closed subset in $(M, \omega)$. We shall say that $V$ is $f$-coisotropic at $x\in V$ if for any  closed neighbourhood $U$ of $x$ in $M$ there is an element $\varphi$ in $\widehat{\DHam}_{M}(M\setminus (U\cap V))$ such that $\varphi$ is not the restriction of an element in $ \widehat{\DHam}(M, \omega)$.
 
 We say that $V$ is nowhere $f$-coisotropic, if for all $x$ in $V$, $V$ is not $f$-coisotropic at  $x$. 
 \end{defn} 
 
 Even though in a different formulation,  Humili{\`e}re proved in \cite{Humiliere-completion} that if $\dim(V)<n$ then $V$ is nowhere $f$-coisotropic. 
 
 Note that the definition of $f$-coisotropic  is reminiscent of that of a domain of holomorphy : this is a domain such that there exist holomorphic functions on 
 $\Omega$ that cannot be extended to  a bigger set. Such sets also have a local  geometric characterization as being Levi convex. 
   Here the objects that cannot be extended are element of $\widehat{\DHam}(V, \omega)$ and the extension is to $\widehat{\DHam}(M, \omega)$.

 \begin{prop} \label{Prop-14.4}
 If $V$ is nowhere $\gamma$-coisotropic, any element  in $\widehat{\DHam}_{M}(M\setminus V)$ defines a unique element in $\widehat{\DHam}(M, \omega)/{\mathcal N}_{\Delta}$. In particular any $H \in C^0(M\setminus V)$ defines a unique element in $\widehat{\DHam}(M, \omega)/{\mathcal N}_{\Delta}$.
 \end{prop} 
 Uniqueness follows from the following Lemma. 
 \begin{lem} Let $V$ be nowhere $\gamma$-coisotropic.
 If $\varphi \in \widehat{\DHam}(M, \omega)$ and $\varphi= \Id$ on $M\setminus V$  then $\varphi=\Id$ on $M$ (i.e. $\varphi \in \mathcal N_{\Delta}$).
 Similarly if $(\varphi_k)_{k\geq 1}$ $\gamma$-converges to $\Id$ on $M\setminus V$, then it converges to $\Id$ on $M$ (i.e. $\lim_{k}\gamma (\varphi_{k}, \mathcal N_{\Delta})=0$. 
 \end{lem} 
 \begin{proof} Let  $x\in V$ and  $\psi_k$ be a sequence such that $\gamma-\lim_k \psi_k=\Id$ and $\psi_k(V)\subset M\setminus B(x,\eta)$. We denote by $V_{ \varepsilon }$ an $ \varepsilon$-neighbourhood of $V$. Then there exists $ \varepsilon _{k}$ such that 
 $\psi_k (M\setminus V_{ \varepsilon_{k}}) \supset B(x, \eta/2)$ so that $\varphi \mapsto \psi_k \varphi \psi_k^{-1}$ sends 
$ \widehat{\DHam}(M, M\setminus V_{ \varepsilon_{k}})$ to $ \widehat{\DHam}(M, B(x,\eta/2))$. But since $\gamma-\lim_{k}\psi_{k}=\Id$ and $\varphi \in \widehat{\DHam}(M, M\setminus V_{ \varepsilon_{k}})$ for all $k$,  then 
$ \varphi \in \widehat{\DHam}(M, B(x,\eta/2))$, hence $\varphi \in  \widehat{\DHam}(M, B(x,\eta/2)\cup (M\setminus V))$. As a result $\varphi$ equals $\Id$ on 
$M\setminus V \cup B(x,\eta/2)$. By taking a covering of $V$ by open balls and iterating this argument, we get that $\varphi =\Id$ on $M$ that is $\varphi \in  \widehat{\DHam}(M, M)= \mathcal N_{\Delta}$.  
 
 \end{proof} 
 \begin{proof} [Proof of Proposition \ref{Prop-14.4}]
 Indeed, $(\varphi_k)_{k\geq 1}$ is Cauchy in $\widehat{\DHam}_{M}(M\setminus V)$ if and only if for any subsequence $(l_k)_{k\geq 1}$ going to infinity, the sequence $\varphi_k\varphi_{l_k}^{-1}$ converges to $\Id$ on $M\setminus V$. But then we just proved that the sequence converges to $\Id$ in $\widehat{\DHam}(M, \omega)$. As a result $(\varphi_k)_{k\geq 1}$ is Cauchy in 
 $\widehat{\DHam}(M, \omega)/{\mathcal N}_{\Delta}$ hence converges. 
 Finally we want to prove that $H\in C^{0}(M\setminus V)$ defines an element in $\widehat \DHam_{M} (M\setminus V)$. We can write $H$ as the $C^{0}$-limit on compact sets of a sequence $H_{k}$ of Hamiltonians in $C_{c}^{0}(M\setminus V)$ with $H_{k}=H_{l}$ on an exhausting sequence of open subsets. Clearly for any given compact  set $W\subset M\setminus V$ and $k,l$ large enough, their flow $\varphi_{k}$ satisfies $\varphi_{k}\varphi_{l}^{-1}$ equals $\Id$ on $W$, so that the sequence is Cauchy in $\widehat \DHam_{M}(W)$.  By definition it yields an element in $\widehat\DHam_{M}(M\setminus V)$. 
 \end{proof} 
 \begin{cor} 
If $V$ is $f$-coisotropic at $x$ then it is $\gamma$-coisotropic at $x$. 
 \end{cor} 
 \begin{proof} 
 The  statement is an obvious consequence of Proposition  \ref{Prop-14.4}. 
 \end{proof} 
  \begin{Question} Is the converse true ? 
 \end{Question}
 From Proposition \ref{Prop-9.7} we infer the following reinforcement of Humili{\`e}re's result
 \begin{cor} 
 If $\mathcal {H}^n(V)=0$ then $V$ is nowhere $f$-coisotropic. 
 \end{cor} 

\begin{prop} \label{Prop-14.6}
We have the following properties
\begin{enumerate} 
\item Being $f$-coisotropic is invariant by $ \mathcal {H}_\gamma (M, \omega)$, hence by ${\Homeo} (M, \omega)$. 
\item Being $f$-coisotropic is a {\bf local property} in $M$. It only depends on a neighbourhood of $V$ in $(M, \omega)$
\item Being $f$-coisotropic  is {\bf locally hereditary} in the following sense : if through every point $x \in V$ there is an $f$-coisotropic submanifold $V_x\subset V$, then $V$ is $f$-coisotropic. In particular if through any point of $V$ there is an element of $\LL (M, \omega)$ then $V$ is coisotropic. If  any point in $V$ has a neighbourhood contained in an element of $\mathcal S_2(M, \omega)$ then $V$ is not-$f$-coisotropic. 
\end{enumerate} 
\end{prop} 
   \begin{proof} 
   The first two statements are obvious from the definition. For the third one, if $x\in C_x\subset V$ is $f$-coisotropic, then there is  an element in $\varphi$ in $\widehat\DHam (M\setminus (C_x\cap U))$   which does not extend to  $\widehat\DHam (M)$. But since $\varphi$ belongs to $\widehat\DHam (M\setminus (V\cap U))$, this implies that $V$ is $f$-coisotropic at $x$. 
   
   \end{proof}

\section{Appendix: The space \texorpdfstring{$\mathcal L (T^*N)$}{LL(TN)} is not a Polish space}

The question studied in this section  is due to Michele Stecconi. I  thank him for the suggestion and for his help with  the proof. 
We shall prove  
\begin{prop} 
The space $(\mathcal L (T^*N), c)$ is not a Polish space.
\end{prop} 
Remember that a Baire space is a space where the conclusion of Baire's theorem holds: a countable intersection of open dense sets is dense. A topological space is a Polish space if its topology can be defined by a complete metric. Equivalently the space is a countable intersection of open dense sets in its completion. So if a space is not a Polish space,  its completion really adds a lot of points. 
 
 \begin{proof} 
Let $\widehat{\gr\circ d}: C^{0}(N, {\mathbb R} ) \longrightarrow \widehat {\mathcal L} (T^*N)$ be the  extension of the isometric embedding $\gr\circ d: (C^{\infty}(N, {\mathbb R} ), d_{C^0}) \longrightarrow (\mathcal L (T^*N), \gamma)$ given by $f\mapsto \gra (df)$ to the completions of both spaces.
Let \begin{gather*} \mathcal G (T^*N)=  \Image (\widehat{\gr \circ d})\cap \mathfrak L(T^*N)=\\  \left\{ \gra (df) \mid f \in C^0(N, {\mathbb R} ), \gra (df) \in \mathcal L(T^*N)\right\} \end{gather*} 
Thus $\mathcal G (T^*N)$ is the set of {\bf continuous} functions such that $\gra(df)$ is a {\bf smooth} Lagrangian\footnote{This is {\bf not} the set of smooth functions ! 
For example the submanifold $x=p^{3}$ in $T^{*} {\mathbb R}$ is smooth and is the graph of the differential of $f(x)= \frac{3}{4}x^{ \frac{4}{3}}$ that is not smooth.}.
Note that the image of  $\widehat{\gr\circ d}$ is closed\footnote{ Remember that $\mathcal L(T^*N)$ is a set of smooth manifolds but that $\mathcal G (T^*N) \neq \left\{ \gra(df) \mid f\in C^\infty(N, {\mathbb R} )\right \}$} , since it  is an isometry (for the natural $C^{0}$ and $c$ norms) and both spaces are complete. Then $\mathcal G (T^*N)$ is closed in $\mathcal L(T^*N)$ since it is the intersection
of the image of  $\widehat{\gr\circ d}$- which is closed in $\widehat{\mathcal L}(T^{*}N)$, as the isometric image of a complete space- and $\mathcal L(T^*N)$. As a result, if $(\mathcal L(T^*N), c)$ is Polish, so is  $\mathcal G (T^*N)$, since a closed subset of a Polish space is Polish (see \cite{Kechris}, thm 3.11,  p.17). Now let us consider the open sets
 $$U_n(x_0)=\left\{L \in \mathcal G(T^*N) \mid L=\gra (df) , \; \text{and}\; \exists t \in ]0, \frac{1}{n}[  \inf_{x\in S(x_0,t)} f(x)>f(x_0)  \right \}$$
 where $S(x_0,t)$ is the sphere of radius $t$ for some Riemannian metric on $N$. 
We claim that $U_n(x_0)$ is dense in $\mathcal G(T^*N)$. Indeed, we may modify $f$ by adding a $C^0$-small smooth function $g$ so that $f+g$ is in $U_n(x_0)$, since $\gamma (\gra (df), \gra(df+dg))= \osc (g)$. Note that $U_{n}(x_{0})$ is open, since the set of functions such that $ \inf_{x\in S(x_0,t)} f(x)>f(x_0) $ is open for the $C^{0}$-topology, hence $U_{n}(x_{0})$ is open for the $\gamma$-topology (since it coincides wiht the $C^{0}$-topology on graphs). 

 Then if $\gr(df) \in U_n(x_0)$ and $f$ is smooth in $B(x_0,\frac{1}{n})$ there must be  a local minimum $y$ of $f$ in $B(x_0, \frac{1}{n})$ so that $df(y)=0$. 
Now let $(z_k)_{k\geq 1}$ be a dense  sequence of points in $N$. We claim that
$$\bigcap_{n=1}^\infty \bigcap_{k=1}^{\infty} U_n (z_k) $$
is the zero section, since if $\gr(df)$ belongs to this intersection and is smooth on the open set  $W\subset N$ of full measure, then $df$ must vanish on some point in $B(z_k, \frac{1}{n}) $ whenever $B(z_k, \frac{1}{n})\subset W$.   But this implies that $df$ is identically zero on $W$, so $f$ is a constant. 
Our last argument uses that if $\gr(df)$ is in $\mathcal L (T^*N)$ we have that $f$ is smooth on an open set of full measure. Indeed, we proved in \cite{Viterbo-Ottolenghi} (see also the Appendix 2 in \cite{Viterbo-NCMT}) that the selector $c(1_x,L)$ is smooth on an open set of full measure, but since obviously $c(1_x,\gr (df))=f(x)$, this implies that $f$ is smooth on an open set of full measure. As a result, the intersection of the open and dense set $U_n(z_k)$ is the singleton $\{0_N\}$. 
 Thus $\mathcal G(T^*N)$  is not even a Baire space (i.e. a space where a countable intersection of open dense sets is dense) and Polish spaces are obviously Baire. 
 \end{proof} 
 Note that there is no obvious explicit description of $\mathcal G(T^*N)$ in terms of the singularities of $f$. Requiring $f$ to be smooth everywhere is too strong while only requiring $C^1$ is too weak. One possibility would be that $f$ must be  $C^1$ everywhere and smooth on an open set of full measure but even though the condition is necessary, as we saw above, we have no idea as to whether it would be sufficient. 
 \begin{rem} The same holds for $(\LL (T^*N), \gamma)$, since the image of $\mathcal G(T^*N)$ by the projection $\rm{unf}: \mathcal L(T^*N) \longrightarrow \LL (T^*N)$ that we shall denote by ${\mathfrak G}(T^*N)$ is also closed in  $\LL (T^*N)$. It can be identified with the closed subspace of $\mathcal G(T^*N)$ of $f$ such that $\int_N f(x) d\mu(x)=0$ for some Borel measure $\mu$. Moreover   $\mathcal G(T^*N)\simeq \mathfrak G(T^*N)\times \mathbb R$ so if $\mathfrak G(T^*N)$ was a Polish space so would $\mathcal G(T^*N)$. 
 
 Since on $\mathfrak G(T^*N)$ the metric $\gamma$ coincides with the Hofer metric, because $$\gamma( \gra(df),\gra(dg))=d_{Hofer}(\gra(df),\gra(dg))= \osc( f-g)$$ where $\osc(f)=\max f- \min f$ (see \cite[Theorem 2]{Milinkovic-Geodesics}). We may then conclude that  $\LL (T^*N)$ endowed with the Hofer metric is not a Polish space either. 
 \end{rem} 
\section{The set of pseudo-graphs: an example of a closed set in \texorpdfstring{$\widehat \LL (T^*N)$}{LT*N}.}

Let us now describe a closed set in $\widehat \LL (T^*N)$. 
 Remember that $FH^*(L_1,L_2;t)$,  the Floer homology of $L_1,L_2$ with action filtration below $t$ (i.e. generated by the intersection points in $L_1\cap L_2$ such that $f_{L_1}(z)-f_{L_2}(z) <t$) yields a persistence module, and as such we can associate to it a barcode. 
 
We now set
 \begin{defn} 
 A Lagrangian $L$ in $\widehat \LL (T^*N)$  is a pseudo-graph if and only if for all $x$, the barcode of $FH^*(L,V_x)$ consists of a single bar $[c(1_x,L), +\infty[$
 \end{defn} 
 \begin{prop} 
 The set of pseudo-graphs is a closed subset of $\widehat \LL (T^*N)$ and contains $\widehat{\gr\circ d}(C^{0}(N, {\mathbb R} ))$ the set of  graphs of differentials of continuous functions. 
  \end{prop} 
  \begin{proof} 
    First of all if $L_n \overset{\gamma} \longrightarrow L$,  we claim that for all $x \in N$, denoting by $V_x$ the vertical fibre $T_x^*N$,  the persistence modules $t \mapsto FH^*(L_n, V_x;t)$ converge to  $t \mapsto FH^*(L, V_x;t)$, i.e. the barcode of the persistence module $t\mapsto FH^*(L_n,V_x,t)$ converges to the barcode of $t\mapsto FH^*(L,V_x,t)$ for the bottleneck distance. This is an immediate consequence of the Kislev-Shelukhin inequality (see \cite{Ki-Sh}). Here we use the formulation in  \cite{Viterbo-inverse-reduction}, Proposition A.3, which applies 
     provided $t\mapsto FH^*(L_n,V_x,t)$  satisfies properties(1)-(3), with conditions (1)-(4) stated after Proposition A.1 in \cite{Viterbo-inverse-reduction}. 
  Now only the existence of the PSS units (i.e. Property (3)
  is non-trivial, but since $FH^*(L,V_x)=FH^*(0_N,V_x)= \mathbb K \cdot u$, this  and the conditions (1)-(4)
  are easily checked. 
  
Now if $L_n$ is a pseudo-graph, the barcode of $t \mapsto FH^*(L_n, V_x;t)$ is made of a single bar,  $[c(1_x,L_n), +\infty [$, so its limit is 
necessarily a barcode made of a single bar, and since $\lim_n c(1_x,L_n)=c(1_x,L)$ we have that the barcode of $t \mapsto FH^*(L, V_x;t)$ 
has a single bar $[c(1_x,L), +\infty [$. This proves the first part of the Proposition. 

Finally it is clear that for $f$ smooth, the graph of $df$ is a pseudo-graph. Now for $f \in C^0(N, {\mathbb R} )$, we can find a sequence of smooth functions,  $f_n$ such that 
$C^0-\lim_n f_n=f$. But this implies that $\gamma-\lim_n \gra (df_n)=\gr (df)$, and since the set of pseudo-graphs is $\gamma$-closed, this implies that $\gra (df)$ is also a pseudo-graph. 
  \end{proof} 
  \begin{prop}
 We have equality : the set of pseudo-graphs coincides with  the set of $\gra(df)$ for $f\in C^{0}(N, {\mathbb R} )$. 
  \end{prop}
 \begin{proof}  Indeed, let $f$ be the selector for the pseudo-graph $L$. Then $f$ is continuous, and $L-\gra (df)$ is a pseudo-graph  such that $c(1_x,L) \equiv 0$. Indeed, the operation $f\mapsto L-\gra(df)$ is well defined on $\mathcal L(T^*N)$, and $$\gamma(L-\gra (df_n), L-\gra(df_m))=\osc(f_n-f_m)\leq 2\vert f_n-f_m\vert_{C^0}$$ so the operation extends to $\widehat{\mathcal L}(T^*N)$. 

The proposition is then an obvious consequence of the following Lemma.
 \begin{lem} 
 Let $L\in \widehat{\mathcal L}(T^{*}N)$ be such that for each $x$ we have $FH^{*}(L,V_{x};a,b)=0$ for given $a<b$. Then we have 
 $FH^{*}(L,0_{N};a,b)=0$. In particular if this holds for all $a<b$ such that $0\notin [a,b]$, then $L=0_{N}$. 
 \end{lem} 
\begin{proof} 
 Let $\cF \in D^{b}(N\times {\mathbb R} )$ associated to $L$. According to \cite{Guillermou-Viterbo},  proposition 9.9, there is an extension of  the Lagrangian quantization map $Q: \mathcal L(T^{*}N) \longrightarrow D^{b}(N\times {\mathbb R})$ from \cite{Guillermou, Guillermou-Asterisque, Viterbo-Sheaves} to a map $\widehat Q: \widehat{\mathcal L}(T^{*}N) \longrightarrow D_{lc}^{b}(N\times {\mathbb R})$. In particular  
 $$FH^{*}(L_{1},L_{2}, a,b)= H^{*}(N\times [a,b[, R\Hom^{\star}(\cF_{1}, \cF_{2}))\dispdot  $$
 
   Then $H^{*}(N\times [a,b[, \cF)= FH^{*}(L,0_{N};a,b)$ since this holds for $\cF=\cF_{L}$ in the smooth case. 
There is then a spectral sequence with $E_{2}^{p,q}=H^{p}(N,  \Hh^{q}([a,b[,\cF_{x}))$, but $\Hh^{q}([a,b[,\cF_{x})=H^q(L, V_{x};a,b)$ which is zero by assumption, so $E_{2}^{p,q}=0$ and $$FH^{*}(L,0_{N};a,b)=H^{*}(N\times [a,b[, \cF)=0\dispdot $$
Finally if $FH^{*}(L,0_{N};a,b)=0$ whenever $0\notin [a,b]$, then  we must have $c_{+}(L)=c_{-}(L)=0$ and $L=0_{N}$. 
 \end{proof}  
 \end{proof} 
 \begin{rem} One can prove directly that $\widehat{\gr\circ d}(C^{0}(N, {\mathbb R} ))$ is closed in $\widehat \LL (T^*N)$. 
 Indeed, $\widehat{\gr\circ d}$ is an isometric embedding between $(C^0(N, {\mathbb R} ), C^0)$ to  $(\widehat \LL (T^*N), \gamma)$. But an isometric embedding between complete spaces must have closed image. The non-obvious statement in the Proposition is that all pseudo-graphs are actually graphs. 
 \end{rem}

\printbibliography
\end{document}